\documentclass [a4paper,abstracton]{scrartcl}

\usepackage[latin1]{inputenc}
\usepackage[T1]{fontenc}

\usepackage{lmodern} 

\usepackage{lscape}
\usepackage[ngerman,english]{babel}
\RequirePackage{amsmath}                       
\RequirePackage{mathtools}
\RequirePackage{amssymb}                       
\RequirePackage{latexsym}                      
\RequirePackage{stmaryrd}                      

\RequirePackage[all]{xy}                       
\UseTips                                      

\newdir^{c}{{}*!/-3pt/\dir^{(}}
\newdir_{c}{{}*!/-3pt/\dir_{(}}
\newdir^{ (}{{}*!/-3pt/\dir^{(}}
\newdir_{ (}{{}*!/-3pt/\dir_{(}}

\RequirePackage{xspace}                         
\RequirePackage{calc}                         
\RequirePackage{ifthen}                      

\RequirePackage{mathrsfs}

\RequirePackage[dvipsnames,usenames]{color}  

\RequirePackage{hyperref}
\hypersetup{
bookmarks,
bookmarksdepth=3,
bookmarksopen,
bookmarksnumbered,
pdfstartview=FitH,
colorlinks,backref,hyperindex,
linkcolor=RawSienna,
anchorcolor=BurntOrange,
citecolor=OliveGreen,
filecolor=BlueViolet,
menucolor=Yellow,
urlcolor=OliveGreen
}

   [2009/11/06 v0.1 finetuning autoref and theorem setup]

\RequirePackage{hyperref}

\makeatletter
\@ifdefinable\equationname{\let\equationname\equationautorefname}
\def\equationautorefname~#1\@empty\@empty\null{(#1\@empty\@empty\null)}%
\@ifdefinable\AMSname{\let\AMSname\AMSautorefname}
\def\AMSautorefname~#1\@empty\@empty\null{(#1\@empty\@empty\null)}%

\@ifdefinable\itemname{\let\itemname\itemautorefname}
\def\itemautorefname~#1\@empty\@empty\null{(#1\@empty\@empty\null)%
}
\makeatother

\makeatletter
\renewcommand{\theenumi}{\alph{enumi}}

\renewcommand{\theenumii}{\roman{enumii}}

\renewcommand{\p@enumii}{\theenumi$\m@th\vert$}

\renewcommand{\p@enumiii}{\theenumi.\theenumii.}

\renewcommand{\labelitemi}{$\m@th\circ$}
\renewcommand{\labelitemii}{$\m@th\diamond$}
\renewcommand{\labelitemiii}{$\m@th\star$}
\renewcommand{\labelitemiv}{$\m@th\cdot$}
\makeatother

\RequirePackage{amsthm}
\RequirePackage{aliascnt}
\newcommand{\basetheorem}[3]{
    \newtheorem{#1}{#2}[#3]
    \newtheorem*{#1*}{#2}
    \expandafter\def\csname #1autorefname\endcsname{#2}
}%
\newcommand{\maketheorem}[3]{
    \newaliascnt{#1}{#3}
    \newtheorem{#1}[#1]{#2}
    \aliascntresetthe{#1}
    \expandafter\def\csname #1autorefname\endcsname{#2}
    \newtheorem*{#1*}{#2}
}
\theoremstyle{plain}  

\basetheorem{theorem}{Theorem}{section}

\maketheorem{proposition}{Proposition}{theorem}
\maketheorem{corollary}{Corollary}{theorem}
\maketheorem{lemma}{Lemma}{theorem}
\maketheorem{conjecture}{Conjecture}{theorem}

\theoremstyle{definition}   

\maketheorem{definition}{Definition}{theorem}
\maketheorem{defprop}{Definition-Proposition}{theorem}

\theoremstyle{remark}   

\maketheorem{remark}{Remark}{theorem}
\maketheorem{remarks}{Remarks}{theorem}
\maketheorem{example}{Example}{theorem}
\maketheorem{examples}{Examples}{theorem}
\maketheorem{question}{Question}{theorem}

\usepackage{enumerate}

\newcommand{\Perv}{\operatorname{Perv}}

\newcommand{\Spec}{\operatorname{Spec}}

\newcommand{\Coh}{\operatorname{Coh}}

\newcommand{\RSHom}{\underline{\operatorname{RHom}}}
\newcommand{\SHom}{\underline{\operatorname{Hom}}}
\newcommand{\RHom}{\operatorname{RHom}}

\newcommand{\Hom}{\operatorname{Hom}}
\newcommand{\Ext}{\operatorname{Ext}}
\newcommand{\ExtS}{\underline{\operatorname{Ext}}}
\newcommand{\G}{\operatorname{G}}
\newcommand{\Sol}{\operatorname{Sol_{\kappa}}}
\newcommand{\Solu}{\operatorname{Sol}}
\newcommand{\M}{\operatorname{M}}
\newcommand{\derotimes}{\overset{\LL}{\otimes}}

\newcommand{\CA}{\mathcal{A}}
\newcommand{\CB}{\mathcal{B}}
\newcommand{\CC}{\mathcal{C}}
\newcommand{\CO}{\mathcal{O}}
\newcommand{\CM}{\mathcal{M}}
\newcommand{\CN}{\mathcal{N}}
\newcommand{\CE}{\mathcal{E}}
\newcommand{\CF}{\mathcal{F}}
\newcommand{\FF}{\mathbb{F}}
\newcommand{\CG}{\mathcal{G}}

\newcommand{\CI}{\mathcal{I}}
\newcommand{\CJ}{\mathcal{J}}
\newcommand{\CK}{\mathcal{K}}

\newcommand{\CL}{\mathcal{L}}
\newcommand{\LL}{\mathbb{L}}

\newcommand{\et}{\text{\'et}}
\newcommand{\LX}{\overline{X}}
\newcommand{\ZZ}{\mathbb{Z}}

\newcommand{\coh}{\operatorname{coh}}
\newcommand{\crys}{\operatorname{crys}}
\newcommand{\id}{\operatorname{id}}
\newcommand{\me}{_{!*}}

\newcommand{\CohC}{\operatorname{Coh}_{\kappa}}

\newcommand{\QCohC}{\operatorname{QCoh}_{\kappa}}
\newcommand{\Crys}{\operatorname{Crys}}
\newcommand{\CrysC}{\operatorname{Crys}_{\kappa}}
\newcommand{\QCrys}{\operatorname{QCrys}}
\newcommand{\QCrysC}{\operatorname{QCrys}_{\kappa}}

\newcommand{\CohG}{\operatorname{Coh}_{\gamma}}
\newcommand{\QCohG}{\operatorname{QCoh}_{\gamma}}
\newcommand{\CrysG}{\operatorname{Crys}_{\gamma}}
\newcommand{\QCrysG}{\operatorname{QCrys}_{\gamma}}

\newcommand{\LNilCohC}{\operatorname{LNilCoh}_\kappa}
\newcommand{\LNilCrysC}{\operatorname{{LNilCrys}}_\kappa}
\newcommand{\LNilC}{\operatorname{LNil}_\kappa}
\newcommand{\NilC}{\operatorname{Nil}_\kappa}

\newcommand{\Gen}{\operatorname{Gen}}
\newcommand{\Neg}{\operatorname{Neg}}

\newcommand{\lfgu}{\operatorname{lfgu }}

\newcommand{\ad}{\operatorname{ad}}
\newcommand{\coker}{\operatorname{coker}}
\newcommand{\Res}{\operatorname{Res}}
\newcommand{\tr}{\operatorname{tr}}
\newcommand{\ctr}{\operatorname{ctr}}
\newcommand{\can}{\operatorname{can}}
\newcommand{\ev}{\operatorname{ev}}
\newcommand{\qc}{\text{qc}}
\newcommand{\pr}{\operatorname{pr}}
\newcommand{\proj}{\operatorname{proj}}
\newcommand{\bc}{\operatorname{bc}}

\newcommand{\usc}[1][m]{\underline{\phantom{#1}}}

\renewcommand{\phi}{\varphi}

\renewcommand{\theta}{\vartheta}

\renewcommand{\epsilon}{\varepsilon}

\renewcommand{\to}[1][]{\xrightarrow{\ #1\ }}

\newcommand{\into}[1][]{\xhookrightarrow{\ #1\ }}

    \reversemarginpar
    \makeatletter
    \advance\marginparwidth by \ta@bcor
    \makeatother
    \newcounter{themargin}
    \def\q?#1{\textcolor{Mahogany}{\textbf{???$^{\text{\arabic{themargin}}}$}}{\marginpar{\footnotesize\color{Mahogany}\fbox{\parbox{\marginparwidth}{\textbf{? --- \arabic{themargin} ---}\addtocounter{themargin}{1}\\ #1}}} \immediate\write16{}%
    \immediate\write16{Warning: There was still a question mark . . . }%
    \immediate\write16{}}}

\title{Cartier crystals and perverse constructible \'etale $p$-torsion sheaves}
\author{Tobias Schedlmeier}
\begin{document}
\title{Cartier crystals and perverse constructible \'etale $p$-torsion sheaves}
\author{Tobias Schedlmeier}
\begin{titlepage}
	\begin{center}
		\vspace*{22mm}\huge \textbf{Cartier crystals}\\ \textbf{and perverse constructible \'etale} \\ \textbf{$\boldsymbol{p}$-torsion sheaves}\\
		\large 
		\vspace{35mm}
		Dissertation\\
		\vspace{3mm}zur Erlangung des Grades\\
		\vspace{3mm} \emph{\glqq Doktor der Naturwissenschaften\grqq} \\
		\vspace{3mm}am Fachbereich Physik, Mathematik und Informatik\\
		\vspace{3mm}der Johannes Gutenberg-Universit\"at\\
		\vspace{3mm}in Mainz\\
		\vspace{35mm}
		vorgelegt von\\
		\Large
		\vspace{3mm} \textbf{Tobias Daniel Schedlmeier}\\
		\large
		\vspace{3mm}geboren in Deggendorf\\
		\vspace {13mm}
		Mainz, im August 2016
	\end{center}
\end{titlepage}
\selectlanguage{english}
\newpage 
\thispagestyle{empty}
\quad 
\newpage
\thispagestyle{empty}
\addto\captionsUKenglish{%
	\renewcommand{\contentsname}{Contents}}
\tableofcontents
\setcounter{page}{1}
\newpage
\quad
\newpage
\addsec{Introduction}

At the International Congress of Mathematicians in 1900, David Hilbert formulated 10 of his famous 23 problems. The 21st problem grew out of Riemann's work on linear differential equations. Hilbert's formulation was the following: ``Aus der Theorie der linearen Differentialgleichungen mit einer unabh\"angigen Ver\"anderlichen z m\"ochte ich auf ein wichtiges Problem hinweisen, welches wohl bereits Riemann im Sinne gehabt hat, und welches darin besteht, zu zeigen, da{\ss} es stets eine lineare Differentialgleichung der Fuchs-schen Klasse mit gegebenen singul\"aren Stellen und einer gegebenen Monodromiegruppe giebt. Die Aufgabe verlangt also die Auffindung von n Functionen der Variabeln z, die sich \"uberall in der complexen z-Ebene regul\"ar verhalten, au{\ss}er etwa in den gegebenen singul\"aren Stellen: in diesen d\"urfen sie nur von endlich hoher Ordnung unendlich werden und beim Umlauf der Variabeln z um dieselben erfahren sie die gegebenen linearen Substitutionen. [...]'' (\cite{Hilbert})

After removing the set $S$ of singular points, a linear differential equation 
\[
	\frac{dY}{dz} = A(z)Y,
\] 
with $A(z) \in \operatorname{GL}_n(\mathbb C(z))$ for some integer $n$, defines a representation of the fundamental group of $\mathbb P^1(\mathbb C) \backslash S$ by monodromy: For a fixed point $a \in \mathbb P^1(\mathbb C) \backslash S$ there is a fundamental solution matrix $Z$ of functions which are analytic in a neighborhood of $a$ by existence theorems. Continuing $Z$ analytically along a closed path $\gamma$ in $\mathbb P^1(\mathbb C) \backslash S$ yields another fundamental matrix $Z'$ such that $Z'=ZM_{\gamma}$ for a matrix $M_{\gamma} \in \operatorname{GL}_n(\mathbb C)$. The assignment $\gamma \mapsto M_{\gamma}$ defines a homomorphism from the fundamental group $\pi_1(\mathbb P^1(\mathbb C) \backslash S;a)$ to $\operatorname{GL}_n(\mathbb C)$. Hilbert's 21st problem asks whether every representation of the fundamental group of $\mathbb P^1(\mathbb C)$ without given singular points occurs as the monodromy of a linear differential equation. 

The solution of the original problem has a history of its own, lasting almost the entire 20th century. Here we are interested in generalizations in the context of algebraic geometry. Following the introduction of \cite{Hotta}, we explain the connection between linear differential equations and $D$-modules. Every linear differential equation on an open subset $X$ of $\mathbb C^n$ with coordinates $(x_1,x_2,\dots,x_n)$ can be written in the form $Pu=0$, where $P$ is an operator of the form 
\[
	\sum_{i_1,\dots,i_n} f_{i_1,\dots,i_n}(\frac{\partial}{\partial x_1})^{i_1}(\frac{\partial}{\partial x_2})^{i_2} \cdots (\frac{\partial}{\partial x_n})^{i_n}
\]
for coefficients $f_{i_1,\dots,i_n}$ in the sheaf $\CO_X$ of holomorphic functions. In other words, we have to study the action of the sheaf $D_X$ of the non-commutative rings of linear partial differential operators with coefficients in $\CO_X$ given by differentiation. This action turns $\CO_X$ into a left $D_X$-module. One immediately checks that the set of solutions of $Pu=0$ is naturally identified with the sheaf $\SHom_{D_X}(D_X/D_XP,\CO_X)$ by the assignment 
\begin{align*}
	\SHom_{D_X}(D_X/D_XP,\CO_X) &\to \CO_X \\
	\phi &\mapsto \phi(1).
\end{align*}
More generally, systems of linear differential equations 
\[
	\sum_{j=1}^l P_{ij}u_j = 0
\]
with $P_{ij} \in D_X$ for all $1 \leq i \leq k$ correspond to coherent left $D_X$-module $M$ defined as the cokernel of the map $\phi\colon D_X^k \to D_X^l$ with
\[
	\phi(Q_1,Q_2,\dots,Q_k)_l = \sum_{i=1}^k Q_iP_{ij}.
\]
The sheaf of solutions of such a system is $\SHom_{D_X}(M,\CO_X)$. This shows how the study of solutions of partial linear differential equations is connected with the theory of $D$-modules.   

We turn to the more general situation where $X$ is a smooth variety over the complex numbers. A fundamental step to the modern version of the Riemann-Hilbert correspondence was the result by Pierre Deligne in 1970 (\cite{Deligne}), which states that there is an equivalence 
\[
	\operatorname{Conn}^{\operatorname{reg}} \to \operatorname{Loc}(X^{\operatorname{an}})
\]
between the categories of regular integrable connections on $X$ and local systems on $X^{\operatorname{an}}$, i.e.\ $\mathbb C_X$-modules which are locally free of finite rank for the analytic topology. Here an integrable connection is a $D_X$-module $M$ which is locally free of finite rank as an $\CO_X$-module. This is nothing but a locally free $\CO_X$-module of finite rank together with a $\mathbb C$-linear map $\nabla\colon M \to \Omega_X^1 \otimes_{\CO_X} M$. For the above equivalence, we have to pass to a certain subcategory, namely the regular integrable connections, and the underlying functor of the equivalence is given by taking the kernel of $\nabla$. Note that if $d_X$ denotes the dimension  of $X$, this is the cohomology in degree $-d_X$ of the de Rham complex 
\[
	0 \longrightarrow \Omega_X^1 \otimes_{\CO_X} M \longrightarrow \Omega_X^2 \otimes_{\CO_X} M \longrightarrow \cdots \longrightarrow \Omega_X^{d_X} \otimes_{\CO_X} M \longrightarrow 0
\]
located between the degrees $-d_X$ and $0$ and whose differentials are induced by $\nabla$. Let $DR_X(M)$ denote this complex.

Both categories $\operatorname{Conn}^{\operatorname{reg}}$ and $\operatorname{Loc}(X^{\operatorname{an}})$ are not closed under push-forwards. For instance, the push-forward of a local system on the origin to the affine line is obviously not a local system. The correct extensions are (regular holonomic) $D_X$-modules on the left and constructible $\mathbb C_X$-sheaves on the right. However, the functor $H^{-d_X}(DR_X(\usc))$ does not yield an equivalence between these larger categories. Again considering the example of the inclusion $i\colon \{0\} \to \mathbb A_{\mathbb C}^1$, we see that $H^{-1}(DR_{\mathbb A^1}(i_*\mathbb C)) = 0$. This is due to the fact that we loose to much information by only taking into account the $-d_X$-th cohomology of $DR_X(\usc)$. To avoid this problem, one considers the derived functor $DR_X(\usc)=\Omega_X \derotimes_{\CO_X} \usc$ between the derived categories of $D_X$-modules and constructible $\mathbb C_X$-sheaves. In the context of complex manifolds, Kashiwara (\cite{Kashiwara1} and \cite{Kashiwara2}) passed to a suitable subcategory called regular holonomic $D$-modules -- more precisely the full subcategory of the bounded derived category $D^b(D_X)$ consisting of complexes whose cohomology sheaves are regular holonomic -- and proved that the de Rham functor is an equivalence 
\[
	D_{rh}^b(D_X) \to D_c^b(\mathbb C_X),
\]  
which is compatible with the six operations $f_*$, $f_!$, $f^*$, $f^!$, $\RSHom^{\bullet}$ and $\derotimes$. Here $D_c^b(\mathbb C_X)$ denotes the full subcategory of $D(\mathbb C_X)$ of bounded complexes with constructible cohomology sheaves. This result from 1980 and 1984 is known as the Riemann-Hilbert correspondence. Around the same time, Mebkhout (\cite{Mebkhout1} and \cite{Mebkhout2}) gave a proof, which is independent of Kashiwara's work. Later on, Beilinson and Bernstein developed the Riemann-Hilbert correspondence for algebraic $D$-modules on complex algebraic varieties. Their work is explained in the unpublished notes (\cite{Bernstein}). The Riemann-Hilbert correspondence is a very general answer to Hilberts 21st problem, because $DR_X(\usc)$ is closely related to the so-called solution functor $\Solu_X = \RSHom_{D_X}^{\bullet}(\usc,\CO_X)$: for every bounded complex $M^{\bullet}$ of $D_X$-modules, we have  
\[
	DR_X(M^{\bullet}) \cong \Solu_X(\mathbb D_X M^{\bullet})[d_X],
\]
where $\mathbb D_X$ is a certain duality. As explained above, for a coherent $D_X$-module $M$, the sheaf $\SHom_{D_X}(M,\CO_X)$ can be identified with the solutions of the system of differential equations corresponding to $M$. Furthermore, there is an equivalence between representations of the fundamental group of $X$ and locally constant $\mathbb C_X$-sheaves.   

Of course the essential image of the abelian category of regular holonomic $D_X$-modules under the equivalence $DR_X$ is an abelian category inside $D_c^b(\mathbb C_X)$, but it turns out that this category differs from the category of constructible $\mathbb C_X$-sheaves. The example of the immersion of the origin into the affine line from above already is a first sign of this phenomenon. The abelian subcategory of $D_c^b(\mathbb C_X)$ given by the essential image of regular holonomic $D_X$-modules under the de Rham functor is called \emph{perverse sheaves}. There is a general tool for describing abelian subcategories of triangulated categories: the theory of $t$-structures. A $t$-structure on a triangulated category $D$ consists of two subcategories $D^{\leq 0}$ and $D^{\geq 0}$ with certain properties. The intersection $D^{\leq 0} \cap D^{\geq 0}$ is called the \emph{heart of the $t$-structure}. It is an abelian category. For example, the so-called canonical $t$-structure of $D_{rh}^b(D_X)$ is given by the two subcategories $D_{rh}^{\leq 0}(D_X)$ and $D_{rh}^{\geq 0}(D_X)$ of complexes whose cohomology is zero in positive or negative degrees. In the same way, the category of perverse sheaves on $X$ is obtained as the heart of a $t$-structure on $ D_c^b(\mathbb C_X)$ which is called the perverse $t$-structure. Indeed, the development of the theory of perverse sheaves by Beilinson, Bernstein, Deligne and Gabber was motivated by the Riemann-Hilbert correspondence. A standard reference for this is \cite{BBD}.          

At the beginning of the 21st century, the time was right for a positive characteristic version of the Riemann-Hilbert correspondence. The de Rham theory for varieties over a field of positive characteristic $p$ differs strongly from the one on complex varieties. Instead of the Poincar\'e lemma, we have the Cartier isomorphism and as a consequence, for a smooth variety $X$, the kernel of the map $\CO_X \to \Omega_X^1$ is not a locally constant $\ZZ/p\ZZ$-sheaf but given by the $p$-th powers $(\CO_X)^p$. Therefore, one has to find a different approach. The Frobenius endomorphism $F$ is a major tool in characteristic $p$. Especially sheaves with an action of the Frobenius turned out to be very useful. The starting point of these objects is the sheaf $\CO_{F,X} = \CO_X[F]$ of non-commutative rings given on an affine open subset $U \subseteq X$ by the polynomial ring $\CO_X(U)[F]$ with the relation $Fr=r^pF$ for local sections $r \in \CO_X(U)$. A simple calculation shows that left $\CO_X[F]$-modules are identified with $\CO_X$-modules $\CF$ together with a morphism $F^*\CF \to \CF$. In \cite[Proposition 4.1.1]{Katz}, Katz proved that there is an equivalence between the category of locally free \'etale $\FF_p$-sheaves and the category of coherent, locally free $\CO_X$-modules $\CE$ together with an isomorphism $F_\CE^* \to \CE$ of $\CO_X$-modules. This may be considered as an analogue of Deligne's result that there is a natural equivalence $\operatorname{Conn}^{\operatorname{reg}} \to \operatorname{Loc}(X^{\operatorname{an}})$.  

It is this result of Katz that motivated Emerton and Kisin to consider left $\CO_{F,X}$-modules for establishing an analogue of the Riemann-Hilbert correspondence for smooth varieties over a field $k$ of positive characteristic $p$. As Katz' work already suggested, certain unit left $\CO_{F,X}$-modules, i.e.\ $\CO_{F,X}$-modules $\CF$ whose structural morphism $F^*\CF \to \CF$ is an isomorphism together with some finiteness condition, is the subcategory to look at. In 2004, Emerton and Kisin published \cite{EmKis.Fcrys}, where they proved that the functor $\Solu= \RSHom_{\CO_{F,X_{\et}}}^{\bullet}(\usc_{\et},\CO_{X_{\et}})[d_X]$ yields an anti-equivalence
\[
	D_{\lfgu}^b(\CO_{F,X}) \to D_c^b(X_{\et},\ZZ/p\ZZ) 
\]
between the bounded derived categories of locally finitely generated unit (lfgu for short) left $\CO_{F,X}$-modules on the one hand, and the bounded derived category of constructible $\ZZ/p\ZZ$-sheaves on the \'etale site $X_{\et}$ of $X$ on the other hand. Their correspondence is shown to be compatible with half of the six cohomological operations, namely $f^!$, $f_+$ and $\derotimes$. They also prove that under the correspondence the abelian category $\mu_{\lfgu}(X)$ of locally finitely generated unit modules corresponds to the category of perverse sheaves $\Perv(X_{\et},\ZZ/p\ZZ)$ defined by Gabber in \cite{Gabber.tStruc} on $D_c^b(X_{\et},\ZZ/p\ZZ)$. In this Riemann-Hilbert type correspondence, the sheaf of partial differential operators is substituted by the sheaf $\CO_{F,X}$. Every $\CO_{F,X}$-module naturally has the structure of a $D_X$-module. The crucial point is that the ring $D_X$ of arithmetic differential operators introduced by Berthelot equals the union $\bigcup \operatorname{End}_{{\CO_X}^{p^e}}(\CO_X)$ (\cite{Berthelot1}, \cite{Berthelot2}). The details of the $D_X$-module structure of an $\CO_{F,X}$-module are explained in \cite{BliDMod}. It follows that the category considered by Emerton and Kisin is a subcategory of the category of left modules over the sheaf of rings of differential operators.

The sequence 
\[
	0 \longrightarrow \CO_{X_{\et}} \xrightarrow{1-F} \CO_{X_{\et}} \longrightarrow 0
\]
in some sense plays the role of the de Rham complex for varieties over $\mathbb C$. For instance, we can compute $\Solu(\CO_X) = \RSHom_{\CO_{F,X_{\et}}}^{\bullet}(\CO_{X_{\et}},\CO_{X_{\et}})[d_X]$ using the resolution
\[
	0 \longrightarrow \CO_{F,X_{\et}} \xrightarrow{1-F} \CO_{F,X_{\et}}
\]
of $\CO_{X_{\et}}$ by free left $\CO_{F,X_{\et}}$-modules. As a consequence of Artin-Schreier theory, the sequence
\[
	0 \longrightarrow (\ZZ/p\ZZ)_X \longrightarrow \CO_{X_{\et}} \xrightarrow{1-F} \CO_{X_{\et}} \longrightarrow 0
\]
is exact and therefore $\Sol(\CO_X) \cong (\ZZ/p\ZZ)_X[d_X]$. This observation is fundamental in the proof of Emerton and Kisin's Riemann-Hilbert correspondence.  

In \cite{BliBoe.CartierFiniteness}, Blickle and B\"ockle show that if $X$ is smooth and $F$-finite (i.e.\ the Frobenius morphism is a finite map), then Emerton-Kisin's category $\mu_{\lfgu}(X)$ is equivalent to their category $\CrysC(X)$ of Cartier crystals on $X$. This category is obtained by localizing the category of coherent sheaves $M$ on $X$ equipped with a right action by Frobenius, i.e.\ a map $F_*M \to M$, at the Serre subcategory consisting of those $M$ where the structural map is nilpotent. 

The category of Cartier crystals is also defined on singular schemes, and a Kashiwara type equivalence holds in this context \cite[Theorem 4.1.2]{BliBoe.Cartier}, showing that Cartier crystals on a closed subscheme $Z \subseteq X$ are ``the same'' as Cartier crystals on $X$ supported in $Z$. This suggests that for singular schemes, the category of Cartier crystals should be a reasonable replacement for Emerton-Kisin's theory, which was only developed for $X$ smooth. Hence one expects a natural equivalence of categories
\[
	\CrysC(X) \to \Perv(X_{\et},\ZZ/p\ZZ)
\]
for any $F$-finite scheme $X$. In this paper we show this result under the assumption that $X$ is embeddable into a smooth $F$-finite variety. The closed immersion of $X$ into a smooth, $F$-finite scheme $Y$ enables us to employ the Kashiwara equivalence to show that the category of Cartier crystals on $X$ is equivalent to the category of lfgu modules on $Y$ supported in $X$. This equivalence on the level of abelian categories then extends to a derived equivalence 
\[
	D_{\crys}^b(\QCrysC(X)) \cong D_{\lfgu}^b(\CO_{F,Y})_X.
\]
The details of this equivalence are worked out in Section 2 and involve showing that the equivalence sketched by Blickle and B\"ockle between Cartier crystals and $\mu_{\lfgu}(X)$ alluded to above is compatible with pull-back functors for immersions of smooth, $F$-finite schemes and push-forward functors for arbitrary morphisms between smooth, $F$-finite schemes. 

In Section 4 we give an intrinsic proof of the fact that the category $D_{\lfgu}^b(\CO_{F,Y})_X$ is independent of the embedding of $X$ into a smooth scheme $Y$. If one had resolution of singularities in characteristic $p$, one would have natural isomorphisms of functors $\Solu f_+ \cong f_! \Solu$ for every morphism $f$ between smooth $k$-schemes \cite[Theorem 9.7.1]{EmKis.Fcrys}. This would enable us to work with derived categories of constructible \'etale sheaves, which are defined on singular schemes as well. Consequently, the independence of a chosen embedding is an easy exercise. As resolution of singularities is an open problem in higher dimensions, we are required to extend the adjunction between the functors $f^!$ and $f_+$ for proper $f$ from Emerton-Kisin to the case that $f$ is proper over some closed subset, which is somewhat technical. The source of this is a general adjunction statement for quasi-coherent sheaves, which we show in section 3, and which we believe to be of independent interest:
\begin{theorem*} 
Let $f\colon X \to Y$ be a separated and finite type morphism of Noetherian schemes and let $i\colon Z \to Y$ and $i'\colon Z' \to X$ be closed immersions with a proper morphism $f'\colon Z' \to Z$ such that the diagram
\[
	\xymatrix{
		Z' \ar[r]^-{i'} \ar[d]^-{f'} & X \ar[d]^-f \\
		Z \ar[r]^-i & Y
	}
\]
commutes. Then there is a natural transformation $\tr_f\colon Rf_*R\Gamma_{Z'}f^! \to \id$ such that, for all $\CF^{\bullet} \in D_{\qc}^-(\CO_X)_{Z'}$ and $\CG^{\bullet} \in D_{\qc}^+(\CO_Y)_Z$, the composition 
\begin{align*}
	\xymatrix{
		Rf_* \RSHom_{\CO_X}^{\bullet}(\CF^{\bullet},R\Gamma_{Z'}f^!\CG^{\bullet}) \ar[r] &  \RSHom_{\CO_Y}^{\bullet}(Rf_* \CF^{\bullet},Rf_*R\Gamma_{Z'}f^! \CG^{\bullet}) \ar[d]^{\tr_f} \\
		& \RSHom_{\CO_Y}^{\bullet}(Rf_* \CF^{\bullet}, \CG^{\bullet})}
\end{align*}
is an isomorphism. Here the first arrow is the natural map. In particular, taking global sections, the functor $Rf_*$ is left adjoint to the functor $R\Gamma_{Z'}f^!$. 
\end{theorem*}
Combining these steps, the following theorem summarizes the main results in this paper:
\begin{theorem*} 
Let $X$ be a $k$-scheme embeddable into a smooth, $F$-finite $k$-scheme. Then there are natural equivalences of categories
\[
	D_{\crys}^b(\QCrysC(X)) \overset{\sim}{\longrightarrow} D_{\lfgu}^b(\CO_{F,Y})_X \overset{\sim}{\longrightarrow} D_c^b(X_{\et},\ZZ/p\ZZ),
\]
where $Y$ is a smooth, $F$-finite $k$-scheme with a closed immersion $X \to Y$. Here the middle category is independent of the embedding. These equivalences are compatible with the respectively defined push-forward and pull-back functors for immersions. Furthermore, the standard $t$-structure on the left corresponds to Gabber's perverse $t$-structure on the right.
\end{theorem*}
\begin{corollary*}
The abelian category $\CrysC(X)$ of Cartier crystals on a variety $X$ embeddable into a smooth, $F$-finite variety is naturally equivalent to the category $\Perv(X_{\et},\ZZ/p\ZZ)$ of perverse constructible \'etale $p$-torsion sheaves. 
\end{corollary*}
While in the final stages of writing up these results, the preprint \cite{Ohkawa} appeared. Therein the author shows that Emerton-Kisin's Riemann-Hilbert correspondence can be extended to the case that $X$ is embeddable into a proper smooth $W_n$-scheme. The case $n=1$ hence also implies the right half of the just stated theorem in the case that $X$ is embeddable into a proper smooth scheme.
 
Finally, in section 6 we define the intermediate extension functor $j\me\colon \CrysC(U) \to \CrysC(C)$ of Cartier crystals for a non-empty open subset $U \subseteq X$. For a Cartier crystal $\CM$ on $U$, its intermediate extension $j\me \CM$ is the smallest subcrystal $\CN$ of $j_*\CM$ such that the restriction of $\CN$ to $U$ is naturally isomorphic to $\CM$. This is equivalent to being an extension of $\CM$ to $X$ without non-trivial subobjects or quotients supported on the complement of $U$. The latter property is a characterization of the intermediate extension defined in \cite{BBD} for a recollement situation. In fact, we show that $j\me \CM$ is naturally isomorphic to the intermediate extension of $\Sol(\CM)$ in the context of perverse sheaves. 

To illustrate the importance of the intermediate extension in general, we turn back to algebraic varieties over $\mathbb C$. In order to recover Poincar\'e duality for singular projective complex algebraic varieties, Goresky and MacPherson considered the cohomology groups $IH^i(X)=H^i(X,j\me \mathbb C_{X_{\text{reg}}}[-d_X])$ instead of $H^i(X,\mathbb C_X)$ for the inclusion $j\colon X_{\text{reg}} \into X$ of the regular locus of $X$ (\cite{GoreskyMacPherson}). For $0 \leq i \leq 2d_X$, there is an isomorphism 
\[
	IH^i(X) = [IH^{2d_X-i}(X)]^*,
\]
which generalizes the Poincar\'e duality
\[
	H^i(X,\mathbb C_X) = [H^{2d_X-i}(X,\mathbb C_X)]^*
\]
for smooth projective varieties.

\section*{Acknowledgments} 

I cordially thank my supervisor, Manuel Blickle, for his excellent guidance and various inspiring conversations. I also thank Kay R\"ulling, who explained to us how the trace map of section 3 could be constructed. Moreover, I thank Gebhard B\"ockle for many useful comments, Axel St\"abler for advancing discussions and Sachio Ohkawa for a careful reading and commenting on an earlier draft of parts of this thesis. The author was partially supported by SFB / Transregio 45 Bonn-Essen-Mainz financed by Deutsche Forschungsgemeinschaft.

\section*{Notation and conventions}

Unless otherwise stated, all schemes are locally Noetherian and separated over the field $\FF_p$ for some fixed prime number $p > 0$. For such a scheme $X$, we let $F_X$ or $F$, if no ambiguity is possible, denote the Frobenius endomorphism $X \to X$ which is the identity on the underlying topological space and which is given by $r \mapsto r^p$ on local sections. Often we will assume that our scheme is $F$-finite, i.e.\ that $F$ is a finite morphism.

Working with Emerton and Kisin's category of locally finitely generated unit modules forces us at some points to restrict to varieties, i.e. to schemes which are of finite type over a field $k$ containing $\FF_p$. With ``schemes over $k$'' or ``$k$-scheme'' we always mean schemes which are separated and of finite over $k$ and we will assume that $k$ is perfect. For a smooth scheme $X$ over a perfect field $k$, the sheaf of top differential forms $\omega_X$ is an invertible sheaf with a canonical morphism $\omega_X \to F^! \omega_X$ of $\CO_X$-modules given by the Cartier operator, see \autoref{ex.StandardCartierOnOmega} for the affine space. One can check that it is an isomorphism. In general, if $X$ is regular and $F$-finite, we will assume that there is a dualizing sheaf $\omega_X$ with an isomorphism $\kappa_X\colon \omega_X \to F^!\omega_X$. For example, this assumption holds if $X$ is a scheme over a local Gorenstein scheme $S = \Spec R$ (\cite[Proposition 2.20]{BliBoe.CartierFiniteness}). Moreover, we assume that $\omega_X$ is invertible. 

As in \cite{EmKis.Fcrys}, for a smooth $k$-scheme $X$, we let $d_X$ denote the function
\[
	x \mapsto \text{dimension of the component of } X \text{ containing } x.
\]
If $f\colon X \to Y$ is a morphism of smooth $k$-schemes, the relative dimension $d_{X/Y}$ is given by $d_{X/Y} = d_X - d_Y \circ f$. 

\setcounter{section}{0}
\section{Review of Cartier crystals and locally finitely generated unit modules}

We begin by reviewing the definitions and results from the theory of Cartier crystals as developed by Blickle and B\"ockle in \cite{BliBoe.CartierFiniteness} and \cite{BliBoe.Cartier}. In short, a coherent Cartier module $M$ on $X$ is a coherent $\CO_X$-module together with a \emph{right action} of the Frobenius $F$. These form an abelian category and the category of Cartier crystals is obtained by localizing at the full Serre subcategory of those $M$ on which $F$ acts nilpotently. The resulting localized category is an abelian category, which has been shown in \cite{BliBoe.CartierFiniteness} to enjoy strong finiteness properties: All objects have finite length and all endomorphism sets are finite dimensional $\FF_p$-vector spaces.

\subsection{Cartier modules and Cartier crystals}

\begin{definition}
A \emph{Cartier module} on $X$ is a quasi-coherent $\CO_X$-module $M$ together with a morphism of $\CO_X$-modules
 \[
	\kappa\colon F_*M \to M.
\]		
\end{definition}
Equivalently, a Cartier module $M$ is a sheaf of right $\CO_{F,X}$-modules whose underlying sheaf of $\CO_X$-modules is quasi-coherent. Here $\CO_{F,X}$ is the sheaf of (non-commutative) rings $\CO_X[F]$, defined affine locally on $\Spec R$ as the ring
\[
	R[F]:=R \{ F \} / \langle Fr - r^pF \ | \ r \in R\rangle.
\]
On the level of abelian sheaves, $M$ and $F_*M$ are equal, hence we may view the structural map $\kappa$ of a Cartier module $M$ as an additive map $\kappa \colon M \to M$ which satisfies $\kappa(r^p \cdot m) = r\kappa(m)$ for all local sections $r \in \CO_X$ and $m \in M$. In this way it is clear that defining the right action of $F$ on $M$ via $\kappa$ defines a right action of $\CO_X[F]$ on $M$, and vice versa. Iterations of $\kappa$ are defined inductively: $\kappa^n:=\kappa \circ F_*\kappa^{n-1}$. Considering $\kappa$ as an additive map of abelian sheaves, $\kappa^n$ is the usual $n$-th iteration.

For a finite morphism $f\colon X \to Y$ of schemes, the functor $f_*$ is left adjoint to the functor $f^{\flat} := \overline{f}^*\SHom_{\CO_Y}(f_*\CO_X,\usc)$, where $\overline{f}$ is the flat morphism of ringed spaces $(X,\CO_X) \to (Y,f_*\CO_X)$, see \cite[III. 6]{HartshorneRD}. Hence the structural morphism of a Cartier module $M$ on an $F$-finite scheme may also be given in the form $\tilde{\kappa}\colon M \to F^{\flat}M$. 
\begin{example}\label{ex.StandardCartierOnOmega}
The prototypical example of a Cartier module is the sheaf $\omega_X$ of top differential forms on a smooth variety over a perfect field $k$. If $X = \Spec k[x_1,\ldots,x_n]$, then $\omega_X$ is the free $k[x_1,\dots,x_n]$-module of rank $1$ generated by $dx_1 \wedge \dots \wedge dx_n$. This module has a natural homomorphism $\kappa\colon F_*\omega_X \to \omega_X$ called the \emph{Cartier operator} given by the formula
\[
	x_1^{i_1}\cdot \dots \cdot x_n^{i_n} dx_1 \wedge \dots \wedge dx_n \mapsto x_1^{\frac{(i_1+1)}{p}-1}\cdot \dots \cdot x_n^{\frac{(i_n+1)}{p}-1} dx_1 \wedge \dots \wedge dx_n
\]
where a non-integral exponent anywhere renders the whole expression zero. 
\end{example}

A morphism $\phi\colon M \to N$ of Cartier modules is a morphism of the underlying quasi-coherent sheaves making the following diagram commutative:
\[
	\xymatrix{
		F_*M \ar[r]^{F_* \phi} \ar[d]_{\kappa_M} & F_*N \ar[d]^{\kappa_N} \\
		M \ar[r]^{\phi} & N.
	}
\]
As $F_*$ is exact, one immediately verifies that \emph{Cartier modules form an abelian category}, the kernels and cokernels being just the underlying kernels and cokernels in $\CO_X$-modules with the induced structural morphism. We denote the category of Cartier modules on $X$ by $\QCohC(X)$. The full subcategory of \emph{coherent Cartier modules} $\CohC(X)$ consists of those Cartier modules whose underlying $\CO_X$-module is coherent. A Cartier module $(M,\kappa)$ is called \emph{nilpotent} if some power of $\kappa$ is zero; $(M,\kappa)$ is called \emph{locally nilpotent} if it is the union of its nilpotent Cartier submodules. By $\LNilC(X)$ we denote the full subcategory of $\QCohC(X)$ consisting of locally nilpotent Cartier modules, and $\NilC(X)$ denotes the intersection $\CohC(X) \cap \LNilC(X)$. The full subcategory of $\QCohC(X)$ consisting of extensions of coherent and locally nilpotent Cartier modules (in either order) we denote by $\LNilCohC(X)$. One has the following inclusions
\begin{equation*}\label{CSh-CCatDiag}
{\parbox{0cm}{{
\xymatrix@C-12pt@R-30pt{
&\LNilC(X)\ar@{^{ (}->}[dr]&&\\
\NilC(X) \ar@{^{ (}->}[ur] \ar@{_{ (}->}[dr] &
&\LNilCohC(X) \ar@{^{ (}->}[r]
&\QCohC(X)\rlap{,}\\
&\CohC(X) \ar@{_{ (}->}[ur] &&\\}
}}}
\end{equation*}
and each of the full subcategories are Serre subcategories in their ambient category\footnote{A Serre subcategory is a full abelian subcategory which is closed under extensions.}. This leads us to our key construction.
\begin{definition} \label{defCartCrys}
The category of \emph{Cartier quasi-crystals} is the localization of the category of quasi-coherent Cartier modules $\QCohC(X)$ at its Serre subcategory $\LNilC(X)$. It is an abelian category, which we denote by $\QCrysC(X)$.
 
Similarly, the category of \emph{Cartier crystals} on $X$ is the localization of the category $\CohC(X)$ of coherent Cartier modules at its Serre subcategory $\NilC(X)$. It is an abelian category, which we denote by $\CrysC(X)$. Cartier crystals also can be obtained by localizing $\LNilCohC(X)$ at the subcategory $\LNilC(X)$.
\end{definition}
In order to define derived functors we have to make sure that the considered categories have enough injectives.
\begin{proposition} \label{enoughinjectiveskappa}
The category $\QCohC(X)$ is a Grothendieck category with enough injectives whose underlying $\CO_X$-module is injective. Its Serre subcategory $\LNilC(X)$ is localizing and hence $\QCrysC(X)$ has enough injectives.
\end{proposition}
\begin{proof}
The first statements were shown in \cite[Theorem 2.0.9 and Proposition 3.3.17]{BliBoe.Cartier}. That $\LNilC(X)$ is localizing now follows from Corollaire 1 on p. 375 of \cite{Gabrielthesis} and from the fact that each $M \in \QCohC(X)$ has a maximal locally nilpotent $\kappa$-subsheaf $M_{nil}$, see \cite[Lemma 2.1.3]{BliBoe.Cartier}. Then Corollaire 2 of \cite{Gabrielthesis} shows that the associated quotient category $\QCrysC(X)$ has enough injectives.

Concretely, if $T \colon \QCohC(X) \to \QCrysC(X)$ denotes the exact localization functor, then the fact that $\LNilC(X)$ is localizing asserts the existence of a right adjoint $V \colon \QCrysC(X) \to \QCohC(X)$. If $M/M_{nil} \into I$ is an injective hull in $\QCohC(X)$, then it is shown in op.\ cit.\ that $TI$ is an injective hull of $T(M/M_{nil})$.
\end{proof}

The following finiteness statements are the main results of \cite{BliBoe.CartierFiniteness}:
\begin{theorem}[\protect{\cite[Corollary 4.7, Theorem 4.17]{BliBoe.CartierFiniteness}}] \label{crystalfinite}
Let $X$ be a locally Noetherian, $F$-finite scheme of positive characteristic $p$.
\begin{enumerate}
\item Every object in the category of Cartier crystals $\CrysC(X)$ satisfies the ascending \emph{and descending} chain condition on its subobjects.
\item The $\Hom$-sets in $\CrysC(X)$ are finite dimensional $\FF_p$-vector spaces.
\end{enumerate}
\end{theorem}

These finiteness properties are precisely the ones one expects from a category of perverse constructible sheaves in the topological context. It is this result (and the related statement \cite[Theorem 11.5.4]{EmKis.Fcrys} in the smooth case) that prompted our investigation of a connection between Cartier crystals and Gabber's category of perverse constructible $\ZZ/p\ZZ$-sheaves on $X_{\et}$, which is the content of this article.

\subsection{Cartier crystals and morphisms of schemes}

Up to now, we studied different categories stemming from quasi-coherent sheaves on a single scheme $X$. In this subsection, we consider morphisms $f\colon X \to Y$ of schemes and construct functors between the categories of Cartier (quasi-)crystals on $X$ and on $Y$. With the notation $D^*(\CA)$ for an abelian category $\CA$ and $* \in \{+,-,b\}$, we mean the subcategory of the derived category $D(\CA)$ of bounded below, bounded above or bounded complexes.

The first result is concerned with the derived functor $Rf_*$ for quasi-coherent $\CO_X$-modules. For a large class of morphisms it behaves well with the additional structure of Cartier modules and with localization at nilpotent objects. In principle, to any quasi-coherent Cartier module $M$ with structural map $\kappa_M$, we assign the quasi-coherent $\CO_Y$-module $f_*M$ together with the composition
\[
	F_{Y*}f_*M \overset{\sim}{\longrightarrow} f_*F_{X*}M \xrightarrow{f_*\kappa_M} f_*M.
\]
This is the underived functor $f_*\colon \QCohC(X) \to \QCohC(Y)$. 
\begin{theorem}[\protect{\cite[Corollary 3.2.12]{BliBoe.Cartier}}]
Let $f \colon  X \to Y$ be a morphism of $F$-finite schemes. Suppose $* \in \{+,-,b \}$. The functor $Rf_*$ on quasi-coherent sheaves induces a functor
\[
    Rf_* \colon D^*(\QCohC(X)) \to D^*(\QCohC(Y)).
\]
It preserves local nilpotence and hence induces a functor
\[
    Rf_* \colon D^*(\QCrysC(X)) \to D^*(\QCrysC(Y)).
\]
If $f$ is of finite type (but not necessarily proper!) then it restricts to a functor
\[
    Rf_* \colon D^*_{\crys}(\QCrysC(X)) \to D^*_{\crys}(\QCrysC(Y))
\]
where the subscript $\crys$ indicates that the cohomology lies in $\LNilCrysC$. 
\end{theorem}
For essentially \'etale morphisms and for closed immersions there are pull-back functors. 
\begin{theorem}
Let $f\colon X \to Y$ be a morphism of schemes.
\begin{enumerate}
\item Suppose $* \in \{+,-,b \}$. If f is essentially \'etale, the exact functor $f^*$ induces a functor 
\[
	f^!\colon D_{\crys}^*(\QCrysC(Y)) \to D_{\crys}^*(\QCrysC(X)),
\]
which is left adjoint to $Rf_*$. 
\item Suppose $* \in \{+,b \}$. If f is a closed immersion of $F$-finite schemes, the functor $f^{\flat} = \overline{f}^*\SHom_{\CO_Y}(f_*\CO_X,\usc)$, where $\overline{f}$ denotes the flat morphism $(X,\CO_X) \to (Y, f_*\CO_X)$ of ringed spaces, induces a functor 
\[
	f^!\colon D_{\crys}^*(\QCrysC(Y)) \to D_{\crys}^*(\QCrysC(X)), 
\]
which is right adjoint to $Rf_*$. 
\end{enumerate} 
\end{theorem}
\begin{proof}
Let $M$ be a quasi-coherent Cartier module on $Y$. For essentially \'etale $f$, there is a canonical isomorphism $\bc\colon F_{X*}f^* \overset{\sim}{\longrightarrow} f^*F_{Y*}$. Hence we may equip $f^*M$ with the structural morphism given by the composition
\[
	F_{X*}f^*M \xrightarrow{\bc} f^*F_{Y*}M \xrightarrow{f^*\kappa_M} f^*M.
\]
As $f^*$ preserves coherence, we obtain a functor $\CohC(Y) \to \CohC(X)$. It is easy to see that $f^*$ preserves nilpotency. Therefore, and by exactness of $f^*$, we obtain the desired functor $D_{\crys}^b(\QCrysC(Y)) \to D_{\crys}^b(\QCrysC(Y))$. 

In the case of a closed immersion $f$, the composition
\[
	f^{\flat}M \xrightarrow{f^{\flat}\tilde{\kappa}} f^{\flat}F_Y^{\flat}M \overset{\sim}{\longrightarrow} F_X^{\flat}f^{\flat}M,
\]
where $\tilde{\kappa}$ is the adjoint of $\kappa$, is a natural Cartier structure for the $\CO_X$-module $f^{\flat}M$. Once again it remains to check that it gives rise to a functor $f^!D_{\crys}^b(\QCrysC(Y)) \to D_{\crys}^b(\QCrysC(X))$. The adjunctions of $Rf_*$ and $f^*$ or $f^!$ follow from the corresponding adjunctions for quasi-coherent sheaves. For more details see \cite[Proposition 3.3.19]{BliBoe.Cartier} and \cite[Corollary 3.3.24]{BliBoe.Cartier}.
\end{proof}
Now let $i\colon Z \to X$ be a closed immersion and $j\colon U \to X$ the open immersion of the complement $X \backslash Z$. Note that for a closed immersion $i$, the functor $i_*$ is exact and therefore we drop the $R$ indicating derived functors. The units and counits of the adjunctions between $i_*$ and $i^!$ and between $Rj_*$ and $j^*$ lead to a familiar distinguished triangle.
\begin{theorem}[\protect{\cite[Theorem 4.1.1]{BliBoe.Cartier}}] \label{Cartiertriangle}
In $D_{\crys}^+(\QCrysC(X))$ there is a distinguished triangle
\[
	i_*i^! \longrightarrow \id \longrightarrow Rj_*j^* \longrightarrow i_*i^![1]. 
\]
\end{theorem}
This theorem shows the equivalence mentioned in the following definition. 
\begin{definition} \label{Cartiersupport}
A complex $\CM^{\bullet}$ of $D_{\crys}^b(\QCrysC(X))$ is \emph{supported on $Z$} if $j^!\CM^{\bullet}=0$ or, equivalently, if the natural morphism $i_*i^!\CM^{\bullet} \to \CM^{\bullet}$ is an isomorphism. We let $D^b_{\crys}(\QCrysC(X))_Z$ denote the full triangulated subcategory consisting of complexes supported in $Z$.
\end{definition}

For Cartier crystals, there is a natural isomorphism of functors $i_*i^! \cong R\Gamma_Z$ where $R\Gamma_Z$ is the local cohomology functor, see \cite[Proposition 2.5]{BliBoe.CartierFiniteness} for the basic result concerning the abelian categories of Cartier modules and the proof of \cite[Theorem 4.1.1]{BliBoe.Cartier}. This isomorphism identifies the distinguished triangle of \autoref{Cartiertriangle} with the fundamental triangle of local cohomology. The following theorem is a formal consequence of \autoref{Cartiertriangle}:   
\begin{theorem}[\protect{\cite[Theorem 4.1.2]{BliBoe.Cartier}}] \label{Kashiwara}
Let $i\colon Z \into X$ be a closed immersion. The functors $i_*$ and $i^!$ are a pair of inverse equivalences 
\[
  \xymatrix{
     D_{\crys}^b(\QCrysC(Z))) \ar@<.5ex>[r]^-{i_*} & D_{\crys}^b(\QCrysC(X))_Z \, . \ar@<.5ex>[l]^-{i^!}
   }
\]
\end{theorem}
We call this equivalence the \emph{Kashiwara equivalence}. If $Z$ is a singular scheme which is embeddable into a smooth scheme $X$, the Kashiwara equivalence enables us to work with objects in $D_{\crys}^b(\QCrys(X))$ instead of $D_{\crys}^b(\QCrys(Z))$.

\subsection{Review of locally finitely generated unit modules}

In \cite{EmKis.Fcrys}, Emerton and Kisin consider left $\CO_{F,X}$-modules, i.e.\ $\CO_X$-modules $\CM$ with a structural morphism $F^*\CM \to \CM$. Instead of localizing, they pass to a certain subcategory. If we speak of $\CO_{F,X}$-modules we mean \emph{left} $\CO_{F,X}$-modules. In this subsection, all schemes are separated and of finite type over a field $k$ containing $\FF_p$.  
\begin{definition}
Let $X$ be a variety over $k$. A quasi-coherent $\CO_{F,X}$-module is an $\CO_{F,X}$-module whose underlying $\CO_X$-module is quasi-coherent. If the structural morphism $F^*\CM \to \CM$ of a quasi-coherent $\CO_{F,X}$-module $\CM$ is an isomorphism, then $\CM$ is called \emph{unit}. We let $\mu(X)$ and $\mu_{\operatorname{u}}(X)$ denote the abelian categories of quasi-coherent and quasi-coherent unit $\CO_{F,X}$-modules.  
\end{definition}
The term ``locally finitely generated'' for an $\CO_{F,X}$-module $\CM$ means that $\CM$ is locally finitely generated as a left $\CO_{F,X}$-module. Emerton and Kisin's focus is on locally finitely generated unit modules, lfgu for short, on smooth schemes, where they form an abelian category.
\begin{definition}
We let $\mu_{\lfgu}(X)$ denote the abelian category of locally finitely generated unit $\CO_{F,X}$-modules. We let $D_{\lfgu}(\CO_{F,X})$ denote the derived category of complexes of $\CO_{F,X}$-modules whose cohomology sheaves are lfgu.
\end{definition}
\begin{proposition} \label{EmKisPullBack}
Let $f\colon X \to Y$ be a morphism of smooth $k$-schemes. The functor $f^!\colon D(\CO_{F,Y}) \to D(\CO_{F,X})$ defined by
\[
	f^!\CM^{\bullet} = \CO_{F,X \rightarrow Y} \derotimes_{f^{-1}\CO_{F,Y}} f^{-1}\CM^{\bullet}[d_{X/Y}]
\]
restricts to a functor 
\[
	f^!\colon D_{\lfgu}^b(\CO_{F,Y}) \to D_{\lfgu}^b(\CO_{F,X}).
\]
Here $\CO_{F,X \rightarrow Y}$ denotes $\CO_{F,X}$ with the natural $(\CO_{F,X},f^{-1}\CO_{F,Y})$-bimodule structure.
\end{proposition}
\begin{proof}
This is Lemma 2.3.2 and Proposition 6.7 of \cite{EmKis.Fcrys}.
\end{proof}
\begin{example}
Let $f\colon U \to X$ be an open immersion of smooth $k$-schemes. Then we have $d_{U/X} = 0$ and the inverse image of $\CO_{F,X}$ is the restriction to $U$: 
\[
	f^{-1}\CO_{F,X} = \CO_{F,X}|_U = \CO_{F,U}. 
\]
Hence we regard $\CO_{F,U \rightarrow X}$ as $\CO_{F,U}$ with the usual $(\CO_{F,U},\CO_{F,U})$-bimodule structure. It follows that $f^!\CM=f^*\CM$ with the natural structure as a left $\CO_{F,U}$-module for every left $\CO_{F,X}$-module $\CM$. 
\end{example}
The construction of the push-forward is more involved. Emerton and Kisin first show that $\CO_{F,Y \leftarrow X} = f^{-1}\CO_{F,Y} \otimes_{f^{-1}\CO_Y} \omega_{X/Y}$ is naturally an $(f^{-1}\CO_{F,Y},\CO_{F,X})$-bimodule. We summarize the construction of the right $\CO_{F,X}$-module structure from Proposition-Definition 1.10.1, Proposition-Definition 3.3.1 and Appendix A.2 of \cite{EmKis.Fcrys}: The relative Frobenius diagram is the diagram
\begin{align} \label{relativeFrobenius}
	\xymatrix{
		X \ar[r]^{F_{X/Y}} \ar[dr]_f & X' \ar[r]^{F_Y'} \ar[d]^-{f'} & X \ar[d]^-f \\
		& Y \ar[r]^{F_Y} & Y.
	}
\end{align}
Here $X'$ is the fiber product of $X$ and $Y$ considered as a $Y$-scheme via the Frobenius and $F_{X/Y}$ is the map obtained from the Frobenius $F_X\colon X \to X$ and the morphism $f$. We call $F_{X/Y}$ the relative Frobenius. Unlike $F_X$, it is a morphism of $Y$-schemes. Locally, for $X=\Spec S$ and $Y=\Spec R$, the structure sheaf of $X'$ is given by the tensor product $R \otimes_R S$, where $R$ is viewed as an $R$-module via the Frobenius $F_R$. Globally we have an isomorphism $\CO_{X'} \cong f^{-1}\CO_X F \otimes_{f^{-1}\CO_Y} \CO_Y$ where $\CO_X F$ denotes the submodule of $\CO_X[F]=\CO_{F,X}$ generated as a left $\CO_X$-module by $F$. Consequently, for any $\CO_X$-module $M$, $F_Y'^*M$ may be viewed as $f^{-1}\CO_X F \otimes_{f^{-1}\CO_Y} M$.  

Let $\gamma\colon \CO_Y \to F_Y^*\CO_Y$ be the canonical isomorphism. The adjoint of the composition 
\[
	f^!\CO_Y \overset{\sim}{\longrightarrow} F_{X/Y}^!f'^! \CO_Y \xrightarrow{F_{X/Y}^!f'^!\gamma} F_{X/Y}^!f'^!F_Y^*\CO_Y \overset{\sim}{\longrightarrow} F_{X/Y}^!F_Y'^*f^!\CO_Y 
\]
yields a morphism $C_{X/Y}\colon F_{X/Y*}\omega_{X/Y} \to F_Y'^*\omega_{X/Y}$ called the \emph{relative Cartier operator}. Note that $F_{X/Y}$ is the identity on the underlying topological spaces of $X$ and $X'$. Therefore $C_{X/Y}$ defines a map of abelian sheaves $\omega_{X/Y} \to F_Y'^*\omega_{X/Y}$. Together with the identification $F_Y'^*\omega_{X/Y} \cong f^{-1}\CO_X F \otimes_{f^{-1}\CO_Y} \omega_{X/Y}$ and the inclusion $\CO_X F \subset \CO_{F,X}$ the relative Cartier defines a map $\omega_{X/Y} \to f^{-1}\CO_{F,Y} \otimes_{f^{-1}\CO_Y} \omega_{X/Y}$. Now we can state the structure of $f^{-1}\CO_{F,Y} \otimes_{f^{-1}\CO_Y} \omega_{X/Y}$ as a right $\CO_{F,X}$-module. The endomorphism on $f^{-1}\CO_{F,Y} \otimes_{f^{-1}\CO_Y} \omega_{X/Y}$ induced by multiplication with $F \in \CO_{F,X}$ on the right is given by the composition
\begin{align*}
	f^{-1}\CO_{F,Y} \otimes_{f^{-1}\CO_Y} \omega_{X/Y} &\xrightarrow{C_{X/Y}} f^{-1}\CO_{F,Y} \otimes_{f^{-1}\CO_Y} f^{-1}\CO_{F,Y} \otimes_{f^{-1}\CO_Y} \omega_{X/Y} \\
	&\overset{m}{\longrightarrow} f^{-1}\CO_{F,Y} \otimes_{f^{-1}\CO_Y} \omega_{X/Y},
\end{align*}  
where $m$ is the the multiplication $a \otimes b \mapsto ab$ in the sheaf of rings $f^{-1}\CO_{F,Y}$.

The functor $f_+\colon D(\CO_{F,X}) \to D(\CO_{F,Y})$ is defined by
\[
	f_+\CM^{\bullet}=Rf_*(\CO_{F,Y \leftarrow X} \derotimes_{\CO_{F,X}} \CM^{\bullet}).
\]
\begin{proposition}
The functor $f_+\colon D(\CO_{F,X}) \to D(\CO_{F,Y})$ restricts to a functor 
\[
	f_+\colon D_{\lfgu}^b(\CO_{F,X}) \to D_{\lfgu}^b(\CO_{F,Y}).
\]
\end{proposition}
\begin{proof}
This is Theorem 3.5.3 and Proposition 6.8.2 of \cite{EmKis.Fcrys}.
\end{proof}
\begin{example} \label{openforward}
Once again, let $f\colon U \to X$ be an open immersion. Then $F_{U/X}$ is an isomorphism, identifying $X'$ with the open subset $U$ of $X$, and $\omega_{U/X} = f^!\CO_X = f^*\CO_X = \CO_U$. Therefore we have 
\[
	\CO_{F,X \leftarrow U} = f^{-1}\CO_{F,X} \otimes_{f^{-1}\CO_X} \CO_U = \CO_{F,U} \otimes_{\CO_U} \CO_U.
\]
The left $\CO_{F,U}$-module structures of $\CO_{F,U} \otimes_{\CO_U} \CO_U$ and $\CO_{F,U}$ are obviously compatible with the natural isomorphism $\CO_{F,U} \otimes_{\CO_U} \CO_U \cong \CO_{F,U}$. We verify that this isomorphism identifies the right $\CO_{F,U}$-module structure on $\CO_{F,U} \otimes_{\CO_U} \CO_U$ with the natural one on $\CO_{F,U}$. Identifying $X'$ with $U$ via $F_{U/X}$, the relative Frobenius diagram reads as follows:
\[
	\xymatrix{ 
		U \ar[r]^{\id} \ar[dr]_j & U \ar[r]^{F_U} \ar[d]^-j & U \ar[d]^-j \\
		& X \ar[r]^{F_X} & X.
	}
\]
The relative Cartier operator $\id_*\CO_U = \CO_U \to F_X^*\CO_U$ is just the usual isomorphism. On global sections $R = \CO_V(V)$ of an affine open $V = \Spec R$ of $U$, its inverse is the map $R' \otimes_R R \to R$ given by $r \otimes s \mapsto rs^p$. Here $R'$ denotes $R$ viewed as an $R$-module via the Frobenius. Locally on $V$, the right action of $F$ on $\CO_{F,U} \otimes_{\CO_U} \CO_U$ is given by the composition 
\[
	\sigma \otimes r \mapsto \sigma \otimes rF \mapsto \sigma r F \otimes 1
\]
since, under the isomorphism $F_U^*\CO_U \overset{\sim}{\longrightarrow} \CO_U F \otimes_{\CO_U} \CO_U$, an element $r \otimes s \in R' \otimes_R R$ corresponds to $rF \otimes s \in \CO_U F \otimes_{\CO_U} \CO_U$. On the other hand, the natural isomorphism $\CO_{F,U} \otimes_{\CO_U} \CO_U \cong \CO_{F,U}$ maps $\sigma \otimes r$ to $\sigma r$. Multiplication by $F$ on the right yields $\sigma rF$, which corresponds to $\sigma rF \otimes 1$ in $\CO_{F,U} \otimes_{\CO_U} \CO_U \cong \CO_{F,U}$. This shows that the diagram
\[
	\xymatrix{
		\CO_{F,U} \otimes_{\CO_U} \CO_U \ar[r]^-{\sim} \ar[d]^F & \CO_{F,U} \ar[d]^F \\
		\CO_{F,U} \otimes_{\CO_U} \CO_U \ar[r]^-{\sim} & \CO_{F,U} 
	}
\]  
commutes, where $F$ denotes multiplication by $F$ on the right.
\end{example}
Depending on $f$, there are adjunction relations between $f^!$ and $f_+$. If $f$ is a closed immersion, a Kashiwara-type equivalence for unit modules holds.
\begin{lemma}[\protect{\cite[Lemma 4.3.1.]{EmKis.Fcrys}}]
If $f\colon X \to Y$ is an open immersion of smooth $k$-schemes, then, for any $\CM^{\bullet} \in D^-(\CO_{F,Y})$ and any $\CN^{\bullet} \in D^+(\CO_{F,X})$, there is a natural isomorphism
\[
	\RSHom_{\CO_{F,Y}}^{\bullet}(\CM^{\bullet},f_+\CN^{\bullet}) \overset{\sim}{\longrightarrow} Rf_*\RSHom_{\CO_{F,X}}^{\bullet}(f^!\CM^{\bullet},\CN^{\bullet})
\]
in $D^+(X,\ZZ/p\ZZ)$.
\end{lemma}
\begin{theorem}[\protect{\cite[Theorem 4.4.1]{EmKis.Fcrys}}] \label{EmKisadj}
Let $f\colon X \to Y$ be a proper morphism of smooth $k$-schemes. For every $\CM^{\bullet}$ in $D_{\qc}^b(\CO_{F,X})$ and every $\CN^{\bullet}$ in $D_{\qc}^b(\CO_{F,Y})$, there is a natural isomorphism in $D^+(X,\ZZ/p\ZZ)$:
\[
	\RSHom_{\CO_{F,Y}}^{\bullet}(f_+\CM^{\bullet},\CN^{\bullet}) \overset{\sim}{\longrightarrow} Rf_*\RSHom_{\CO_{F,X}}^{\bullet}(\CM^{\bullet},f^!\CN^{\bullet}).
\]
Here $D_{\qc}^b(\CO_{F,X})$ denotes the subcategory of $D^b(\CO_{F,X})$ of complexes whose cohomology sheaves are quasi-coherent and analogously for $D_{\qc}^b(\CO_{F,Y})$.  
\end{theorem}
For the proof, Emerton and Kisin show that the trace map $f_*f^{\Delta}E^{\bullet} \to E^{\bullet}$ for the residual complex $E^{\bullet}$ of $\CO_X$ is compatible with the natural map $E^{\bullet} \to F_X^*E^{\bullet}$. Here $f^{\Delta}$ denotes the functor $f^!$ for residual complexes, see \cite[VI.3]{HartshorneRD}. Thus it induces a morphism $f_+\CO_{F,X}[d_{X/Y}] \to \CO_{F,Y}$, and with the isomorphisms
\[
	f_+f^!\CF^{\bullet} \to f_+(\CO_{F,X} \otimes_{\CO_{F,X}} f^!\CF^{\bullet}) \to f_+\CO_{F,X}[d_{X/Y}] \derotimes_{\CO_{F,Y}} \CF^{\bullet},
\]
the second one being a projection formula (\cite[Lemma 4.4.7]{EmKis.Fcrys}), we obtain a trace map $\tr\colon f_+f^!\CF^{\bullet} \to \CF^{\bullet}$ for every $\CF^{\bullet} \in D_{\lfgu}^b(\CO_{F,Y})$. Similarly, as in the case of the adjunction between $Rf_*$ and $f^!$ in Grothendieck-Serre duality, the natural transformation of the theorem is obtained by the composition
\[
	\xymatrix{
		Rf_*\RSHom_{\CO_{F,X}}^{\bullet}(\CM^{\bullet},f^!\CN^{\bullet}) \ar[r] & \RSHom_{\CO_{F,Y}}^{\bullet}(f_+\CM^{\bullet},f_+f^!\CN^{\bullet}) \ar[d]^{\tr} \\
		& \RSHom_{\CO_{F,Y}}^{\bullet}(f_+\CM^{\bullet},\CN^{\bullet}),
	}
\]
where the horizontal arrow is a natural transformation constructed in \cite[Proposition 4.4.2]{EmKis.Fcrys}. It is this adjunction between $f_+$ and $f^!$ that we want to extend to morphisms which are only proper over the support of the considered complexes. This will be done in section 4.

Finally, for a closed immersion of smooth varieties, we have a Kashiwara type equivalence.
\begin{theorem}[\protect{\cite[Theorem 5.10.1]{EmKis.Fcrys}}]
If $f\colon X \to Y$ is a closed immersion of smooth $k$-schemes, then the adjunction of \autoref{EmKisadj} provides an equivalence between the category of unit $\CO_{F,X}$-modules and the category of unit $\CO_{F,Y}$-modules supported on $X$. The fact that the natural map $f_+f^!\CM \to \CM$ is an isomorphism implies that $H^0(f^!)\CM \cong f^!\CM$.  
\end{theorem}

\section{From Cartier crystals to locally finitely generated unit modules}

In order to construct an equivalence between Cartier crystals and locally finitely generated unit modules, one uses an equivalence between Cartier modules and so-called $\gamma$-sheaves. It will induce an equivalence between Cartier crystals and $\gamma$-crystals. The latter in turn are known to be equivalent to lfgu modules. 

\subsection{Cartier modules and $\gamma$-sheaves}

We note that for a regular scheme $X$, the Frobenius $F_X\colon X \to X$ is a flat morphism and hence $F_X^*$ is exact (\cite[Theorem 2.1]{Kunz}). 
\begin{definition}
A \emph{$\gamma$-sheaf} on a regular, $F$-finite scheme $X$ is a quasi-coherent $\CO_X$-module $N$ together with a morphism $\gamma_N\colon N \to F^*N$.
\end{definition} 
The theory of $\gamma$-sheaves is very similar to that of Cartier modules. With the obvious morphisms, $\gamma$-sheaves form an abelian category with the nilpotent $\gamma$-sheaves being a Serre subcategory. We obtain $\gamma$-crystals in the same way as we obtained Cartier crystals and so forth. In this section we revisit the connection between Cartier modules and $\gamma$-sheaves as explained in section 5.2.1 of \cite{BliBoe.CartierFiniteness} and give some details of the proof.

\begin{definition}
For any isomorphism $\phi\colon \CE_1 \to \CE_2$ of invertible $\CO_X$-modules, while $\phi^{-1}\colon \CE_2 \to \CE_1$ denotes the inverse, let $\phi^{\vee}$ denote the induced isomorphism $\CE_2^{-1} \to \CE_1^{-1}$ between the duals.   

If we speak of the $\gamma$-sheaf $\CO_X$ we mean the structure sheaf of $X$ together with the natural isomorphism $\gamma_X\colon \CO_X \to F^*\CO_X$. By abuse of notation we call this isomorphism the \emph{Frobenius}.

For a regular, $F$-finite scheme $X$, let $\kappa_X$, or $\kappa_R$ if $X = \Spec R$ is affine, denote the natural isomorphism $\omega_X \overset{\sim}{\longrightarrow} F^{\flat}\omega_X$, which is the adjoint of the Cartier operator if $X$ is a smooth variety. 
\end{definition}  
The next lemma makes explicit a fundamental isomorphism, which will be used repeatedly. 
\begin{lemma}[\protect{\cite[Lemma 5.7]{BliBoe.CartierFiniteness}}] \label{flatdecompose}
Let $f\colon X \to Y$ be a finite and flat morphism of schemes. For every quasi-coherent $\CO_Y$-module	$\CF$, there is a natural isomorphism
\[
	\can\colon f^{\flat}\CO_X \otimes_{\CO_X} f^*\CG \overset{\sim}{\longrightarrow} f^{\flat} \CG.
\]
\end{lemma}
\begin{proof}
It suffices to construct a natural isomorphism locally and therefore we can identify $f$ with a ring homomorphism $R \to S$ and $\CF$ with an $R$-module $M$. Define the homomorphism
\[
	\can\colon \Hom_R(S,R) \otimes_S (M \otimes_R S) \to \Hom_R(S,M)
\]
of $S$-modules by mapping $\alpha \otimes (m \otimes t)$ to the homomorphism $s \mapsto \alpha(st)m$. Since $f$ is finite flat, we can assume that $S$ is a free $R$-module and choose a basis $s_1,s_2, \dots,s_n$. Let $\phi_1, \dots,\phi_n$ be the dual basis, i.e.\ $\phi_i \in \Hom_R(S,R)$ and $\phi_i(j) = \delta_{ij}$. One easily checks that the map
\begin{align*}
	\Hom_R(S,M) &\to \Hom_R(S,R) \otimes_S (M \otimes_R S) \\
	\phi &\mapsto \sum_{i=1}^n \phi_i \otimes \phi(s_i).
\end{align*}
is inverse to $\can$. 
\end{proof}
The following definition is extracted from \cite[Theorem 5.9]{BliBoe.CartierFiniteness}.
\begin{definition} 
\label{CartierGammaDef}
Let $X$ be a regular, $F$-finite scheme. 
\begin{enumerate} 
	\item For every Cartier module $M$ with structural morphism $\kappa$, the sheaf $M \otimes \omega_X^{-1}$ has a natural $\gamma$-structure given by the composition
	\[
		\xymatrix@C70pt{
			M \otimes \omega_X^{-1} \ar[d]^{\kappa_M \otimes (\kappa_X^{\vee})^{-1}} & \\
			F^{\flat}M \otimes (F^{\flat}\omega_X)^{-1} \ar[r]^-{\can^{-1} \otimes \can^{\vee}} \ar[d]^{\sim} & F^{\flat}\CO_X \otimes F^*M \otimes (F^{\flat}\CO_X)^{-1} \otimes F^*\omega_X^{-1} \ar[d]^{\sim} \\
			F^*(M \otimes \omega_X^{-1}) & F^{\flat}\CO_X \otimes (F^{\flat}\CO_X)^{-1} \otimes F^*M \otimes F^*\omega_X^{-1}, \ar[l]^-{\ev}
			}
	\]
	where the vertical arrow on the right is the permutation and $\ev_{\CL}\colon \CL \otimes_{\CO_X} \CL^{-1} \overset{\sim}{\longrightarrow} \CO_X$ is the evaluation map $l \otimes \phi \mapsto \phi(l)$. This morphism is called the $\gamma$-structure of $M \otimes \omega_X^{-1}$ induced by $\kappa_M$.
	\item For every $\gamma$-module $N$ with structural morphism $\gamma_N$, the sheaf $N \otimes \omega_X$ has a natural Cartier structure given by the composition
	\[
		\xymatrix@C70pt{
			N \otimes \omega_X \ar[d]^{\gamma_N \otimes \kappa_X} & \\
			F^*N \otimes F^{\flat}\omega_X \ar[r]^-{id \otimes \can^{-1}} \ar[d]^{\sim} & F^*N \otimes F^{\flat}\CO_X \otimes F^*\omega_X \ar[d]^{\sim} \\
			F^{\flat}(N \otimes \omega_X) & F^{\flat}\CO_X \otimes F^*N \otimes F^*\omega_X, \ar[l]^-{\can}
		}
	\]
where the vertical arrow on the right is the permutation. This morphism is called the Cartier structure of $N \otimes \omega_X$ induced by $\gamma_N$.
\end{enumerate}
\end{definition}
\begin{remark}
Thanks to the fact that the proof of \autoref{flatdecompose} contains explicit formulas for the isomorphism $\can$ and its inverse, we can concretely describe the induced Cartier structure of $N \otimes \omega_R$ for a $\gamma$-module $N$ over a regular ring $R$ such that $F_*R$ is free with basis $s_1,\dots ,s_r$. For $m \in \omega_R$ set $\phi_m := \kappa_R(m)$ and let $\phi_1, \dots , \phi_r \in \Hom_R(F_*R,R)$ be the dual basis of $s_1, \dots , s_r$, this means $\phi_i(s_j)=\delta_{ij}$. Following the arrows of \autoref{CartierGammaDef}, we see that the Cartier structure $N \otimes \omega_X \to F^{\flat}(N \otimes \omega_X)$ is given by
\begin{align*}
	n \otimes m &\mapsto \gamma(n) \otimes \phi_m \\
	&\mapsto \sum_i \gamma(n) \otimes \phi_i \otimes \phi_m(s_i) \\
	&\mapsto (s \mapsto \sum_i \gamma(n) \otimes \phi_i(s) \otimes \phi_m(s_i)).
\end{align*}
We will need this concrete version later on to prove that, for affine schemes, assigning a Cartier module to a $\gamma$-sheaf commutes with certain pullbacks, see \autoref{localGammapullback}. The use of the isomorphism $F^{\flat}\CO_X \otimes (F^{\flat}{\CO_X})^{-1} \cong \CO_X$ involves the concrete formula for the structural morphism of the $\gamma$-sheaf associated to a Cartier module.
\end{remark}
\begin{lemma}
Let $(N,\gamma_N)$ be a $\gamma$-sheaf on a regular, $F$-finite scheme $X$. The adjoint $F_*(N \otimes \omega_X) \to N \otimes \omega_X$ of the structural morphism of the Cartier module $N \otimes \omega_X$ is given by the composition
\[
	F_*(N \otimes \omega_X) \xrightarrow{\gamma_N} F_*(F^*N \otimes \omega_X) \overset{\sim}{\longrightarrow} N \otimes F_*\omega_X \xrightarrow{\tilde{\kappa}_X} N \otimes \omega_X,
\]
\end{lemma}
where the isomorphism in the middle is given by the projection formula.
\begin{proof}
By construction, the structural morphism of $N \otimes \omega_X$ is the composition of the upper horizontal and the rightmost vertical arrow of the following diagram: 
\[
	\xymatrix{
		N \otimes \omega_X \ar[r]^-{\gamma_N} \ar[d]^{\text{adj}} & F^*N \otimes \omega_X \ar[r]^-{\text{adj}} \ar[d]^{\text{adj}} & F^*N \otimes F^{\flat}F_*\omega_X \ar[d]^{\sim} \ar[r]^-{\tilde{\kappa}_X} & F^*N \otimes F^{\flat}\omega_X \ar[d]^{\sim} \\
		F^{\flat}F_*(N \otimes \omega_X) \ar[r]^-{\gamma_N} & F^{\flat}F_*(F^*N \otimes \omega_X) \ar[r]^-{\text{proj}^{-1}} & F^{\flat}(N \otimes F_*\omega_X) \ar[r]^-{\tilde{\kappa}_X} & F^{\flat}(N \otimes \omega_X).
	}
\]
Here $\text{adj}$ denotes the respective adjunction morphism and $\text{proj}$ is the isomorphism from the projection formula. The third and the fourth vertical morphism are isomorphisms stemming from $\can$. For example, the morphism $F^*N \otimes F^{\flat}\omega_X \to F^{\flat}(N \otimes F_*\omega_X)$ is the composition  
\[
	F^*N \otimes F^{\flat}\omega_X \xrightarrow{\id \otimes \can^{-1}} F^*N \otimes F^*\omega_X \otimes F^{\flat} \CO_X \xrightarrow{\can} F^{\flat}(N \otimes \omega_X).
\]
Following the leftmost vertical and the lower horizontal arrows we obtain the adjoint of the morphism which is claimed to be the adjoint of the Cartier structure of $N \otimes \omega_X$. Hence it suffices to show that the diagram above is commutative. 

The first and the last square commute by functoriality. The commutativity of the square in the middle can be checked locally on affine open subsets of $X$ because $F$ is an affine morphism.
\end{proof} 
For the proof of \autoref{CartierGamma}, we need the isomorphism $\omega_X \otimes \omega_X^{-1} \cong \CO_X$ of $\gamma$-sheaves, which is a consequence of the following general lemma.  
\begin{lemma} \label{ExampleBasic}
Let $f\colon X \to Y$ be a morphism of schemes and $\CL$ an invertible $\CO_Y$-module.
\begin{enumerate}
	\item If $\rho\colon \CL \overset{\sim}{\longrightarrow} \CL_1 \otimes_{\CO_X} \CL_2$ is an isomorphism with invertible $\CO_X$-modules $\CL_1$ and $\CL_2$, then the diagram
	\[
		\xymatrix@C50pt@R30pt{
			\CL \otimes \CL^{-1} \ar[r]^-{\rho \otimes (\rho^{\vee})^{-1}} \ar[d]_{\ev_{\CL}} & \CL_1 \otimes \CL_2 \otimes \CL_1^{-1} \otimes \CL_2^{-1} \ar[r]^-{\id \otimes \ev_{\CL_1}} & \CL_2 \otimes \CL_2^{-1} \ar[dll]^{\ev_{\CL_2}} \\
			\CO_X & &
		}
	\]
	commutes.
	\item The diagram of canonical isomorphisms
	\[
		\xymatrix{
			 f^*(\CL \otimes_{\CO_Y} \CL^{-1}) \ar[d] \ar[r]& f^*\CL \otimes_{\CO_X}(f^*\CL)^{-1}  \ar[d] \\
			f^*\CO_X \ar[r] & \CO_Y
		}
	\]
	commutes.
\end{enumerate} 		
\end{lemma}
\begin{proof}
It suffices to verify the claims for an affine scheme $X = \Spec R$ and, for (b), for a morphism of affine schemes $\Spec S \to \Spec R$ and an $R$-module $M$.

(a) Let $l_1$, $l_2$, $\phi_1$, $\phi_2$ be elements of $\CL_2(X)$, $\CL_2(X)$, $\CL_1^{-1}(X)$, $\CL_2^{-1}(X)$. For $l := \rho^{-1}(l_1 \otimes l_2)$ and $\Phi := \rho^{\vee}(\phi_1 \otimes \phi_2)$ we have
\[
	(\ev_{\CL} \circ (\rho \otimes (\rho^{\vee})^{-1})^{-1})(l_1 \otimes l_2 \otimes \phi_1 \otimes \phi_2) = \ev_{\CL}(l \otimes \Phi) = \Phi(l) = \phi_1(l_1) \cdot \phi_2(l_2)
\]
because $\Phi$ is given by the composition
\[
	\CL \xrightarrow{\rho} \CL_1 \otimes \CL_2 \xrightarrow{\phi_1 \cdot \phi_2} R.
\]
On the other hand 
\[
	(\ev_{\CL_2} \circ (\id \otimes \ev_{\CL_1}))(l_1 \otimes l_2 \otimes \phi_1 \otimes \phi_2) = \ev_{\CL_2}(\phi_1(l_1) \cdot (l_2 \otimes \phi_2)) = \phi_1(l_1) \cdot \phi_2(l_2).
\]
(b) We have to show that the diagram  
\[
	\xymatrix{
		(M \otimes_R \Hom_R(M,R)) \otimes_R S \ar[d]^{\beta} \ar[r]^-{\alpha} & (M \otimes_R S) \otimes_S \Hom_S(M \otimes_R S,S) \ar[d]^{\gamma} \\
		R \otimes_R S \ar[r]^-{\delta} & S
	}
\]
with the maps
\begin{align*}
	&\alpha\colon (m \otimes \phi) \otimes s \mapsto (m \otimes s) \otimes (n \otimes t \mapsto t \cdot \phi(n)) \\
	&\beta\colon (m \otimes \phi) \otimes s \mapsto \phi(m) \otimes s \\
	&\gamma\colon (m \otimes s) \otimes \psi \mapsto \psi(m \otimes s) \\
	&\delta\colon r \otimes s \mapsto rs
\end{align*}
commutes. This follows from
\[
	(\gamma \circ \alpha)((m \otimes \phi) \otimes s) = \gamma((m \otimes s) \otimes (n \otimes t \mapsto t \phi(n))) = s\phi(m)
\]
and
\[
	(\delta \circ \beta)((m \otimes \phi) \otimes s) = \delta(\phi(m) \otimes s) = s\phi(m).
\]
\end{proof}   
\begin{example} \label{GenGamma}
The $\gamma$-sheaf $\omega := \omega_X \otimes \omega_X^{-1}$ on a regular, $F$-finite scheme $X$ is canonically isomorphic to the structure sheaf $\CO_X$ equipped with the natural morphism $\CO_X \to F^*\CO_X$. We have to show that the diagram
\[
	\xymatrix{
		\omega_X \otimes \omega_X^{-1} \ar[d]_{(\can^{-1} \circ \kappa_X) \otimes (\can^{-1} \circ (\kappa_X^{\vee})^{-1})} \ar[r]^-{\ev_{\omega_X}} & \CO_X \ar[dd]^{\id} \\
		F^{\flat} \CO_X \otimes F^* \omega_X \otimes (F^{\flat}\CO_X)^{-1} \otimes F^*\omega_X^{-1} \ar[d]_{\ev_{F^{\flat}\CO_X} \otimes \id}^{\sim} & \\
		F^*\omega_X \otimes F^*\omega_X^{-1} \ar[r]_-{\ev_{F^*\omega_X}}^-{\sim} \ar[d]^{\sim}  & \CO_X \ar[d]^{\gamma_X}  \\
		F^*(\omega_X \otimes \omega_X^{-1}) \ar[r]_-{F^*\ev_{\omega_X}}^-{\sim} & F^*\CO_X  
	}
\]
commutes. The commutativity of both the top and the bottom rectangle follows from \autoref{ExampleBasic}. Note that the horizontal isomorphisms of the lower square are the inverses of the natural isomorphisms of part (b) of \autoref{ExampleBasic}. By definition, the $\gamma$-structure of $\omega_X \otimes \omega_X^{-1}$ is given by the composition of the vertical arrows on the right. Hence the $\gamma$-structure of $\CO_X$ with respect to the natural isomorphism $\CO_X \xrightarrow{\ev_{\omega_X}^{-1}} \omega_X \otimes \omega_X^{-1}$ is the Frobenius.

Similarly, starting with the $\gamma$-module $\CO_X$, the induced Cartier structure of $\CO_X \otimes \omega_X$ is compatible with $\kappa_X$ with respect to the isomorphism $\CO_X \otimes \omega_X \cong \omega_X$.       
\end{example}
\begin{definition}
Let $\QCohG(X)$ denote the category of $\gamma$-sheaves and let $\CohG(X)$ denote the category of $\gamma$-sheaves whose underlying $\CO_X$-module is coherent. We let $\QCrysG(X)$ and $\CrysG(X)$ denote the corresponding categories of crystals, see \autoref{defCartCrys}.
\end{definition}
\begin{proposition} \label{CartierGamma}
If $X$ is a regular, $F$-finite scheme, then tensoring with $\omega_X$ and its inverse induces inverse equivalences of categories between Cartier modules and $\gamma$-sheaves on $X$:
\[
	\xymatrix@C70pt{
	\QCohC(X) \ar@<0.1cm>[r]^-{\usc \otimes_{\CO_X}\omega_X^{-1}} & \QCohG(X) \ar@<0.1cm>[l]^-{\usc \otimes_{\CO_X}\omega_X}
	}
\]
and
\[
	\xymatrix@C70pt{
	\CohC(X) \ar@<0.1cm>[r]^-{\usc \otimes_{\CO_X}\omega_X^{-1}} & \CohG(X) \ar@<0.1cm>[l]^-{\usc \otimes_{\CO_X}\omega_X}.
	}
\]
In terms of this equivalence, the Cartier module $(\omega_X,\kappa_X)$ corresponds to the $\gamma$-sheaf $(\CO_X,\gamma_X)$.
\end{proposition}
\begin{proof}
Let $\omega$ denote the $\gamma$-sheaf $\omega_X^{-1} \otimes \omega_X$ and $\gamma_{\omega}$ its structural morphism. (Note that there is no considerable difference between $\omega_X \otimes \omega_X^{-1}$ and $\omega_X^{-1} \otimes \omega_X$.) We start with a Cartier module $(M,\kappa_M)$. Consider the diagram \autoref{Car} on page \pageref{Car}. Passing through the top arrow we follow the construction of the structural morphism of $M \otimes \omega_X^{-1} \otimes \omega_X$ while the structural morphism of $\kappa_M$ is given by the composition of the horizontal arrows on the bottom: $\kappa_M = \can \circ (\can^{-1} \circ \kappa_M)$. Hence we have to show that \autoref{Car} is commutative. Here $\mu$ denotes a permutation of the tensor product followed by the evaluation map, similar to the top most horizontal arrow. More precisely, it is the composition
\[
	\xymatrix{
		F^{\flat}\CO_X \otimes F^*M \otimes (F^{\flat}\CO_X)^{-1} \otimes F^*\omega_X^{-1} \otimes F^{\flat}\CO_X \otimes F^*\omega_X \ar[d]^{\sim} \\
		F^{\flat}\CO_X \otimes F^*M \otimes F^*\omega_X^{-1} \otimes ((F^{\flat}\CO_X)^{-1} \otimes F^{\flat}\CO_X) \otimes F^*\omega_X \ar[d]^{\ev} \\
		F^{\flat}\CO_X \otimes F^*M \otimes F^*\omega_X^{-1} \otimes F^*\omega_X \ar[d]^{\sim} \\
		F^{\flat}\CO_X \otimes F^*(M \otimes \omega_X^{-1} \otimes \omega_X)
	}
\]
of natural isomorphisms and $\ev$. For simplicity, we will not distinguish between $F^*(M \otimes \omega_X^{-1} \otimes \omega_X)$ and $F^*M \otimes F^*\omega_X^{-1} \otimes F^*\omega_X$. Consider the upper square. The map through the top arrow is given by
\[
	a \otimes m \otimes \beta \otimes \sigma \otimes b \otimes t \mapsto \beta(a) \cdot (b \otimes m \otimes \sigma \otimes t)
\]
and $\mu$ is the map
\[
	a \otimes m \otimes \beta \otimes \sigma \otimes b \otimes t \mapsto \beta(b) \cdot (a \otimes m \otimes \sigma \otimes t).
\]
Because 
\[
	\beta(a) \cdot b = \beta(ab) = \beta(b) \cdot a,
\]
these two maps are equal. In the lower left square the map from $M \otimes \omega_X \otimes \omega_X^{-1}$ to $ F^{\flat}\CO_X \otimes F^*(M \otimes \omega_X^{-1} \otimes \omega_X)$ is the composition 
\[
	M \otimes \omega_X \otimes \omega_X^{-1} \xrightarrow{\can^{-1} \circ \kappa_M} F^{\flat}\CO_X \otimes F^*M \otimes \omega_X^{-1} \otimes \omega_X \xrightarrow{\id \otimes \gamma_{\omega}} F^{\flat}\CO_X \otimes F^*(M \otimes \omega_X^{-1} \otimes \omega_X).
\]
Hence it suffices to show that 
\begin{align} \label{Carout}
	\xymatrix@C60pt{
		F^*M \otimes \omega_X^{-1} \otimes \omega_X \ar[d]^{F^*M \otimes \gamma_{\omega}} & F^*M \otimes \CO_X \ar[l]_-{\id \otimes \delta}^-{\sim} \ar[r]_-{\sim} \ar[d]_{\id \otimes F} & F^*M \ar[d]_{\sim} \\
		F^*M \otimes F^*(\omega_X^{-1} \otimes \omega_X) & F^*M \otimes F^*\CO_X \ar[l]_-{\id \otimes F^*\delta}^-{\sim} \ar[r]_-{\sim} & F^*(M \otimes \CO_X)
	}
\end{align}
is commutative. The commutativity of the left square is \autoref{GenGamma} tensored with $F^*M$. That the right square commutes can easily be checked by hand: For an arbitrary commutative ring $R$, an $R$-algebra $S$ and an $R$-module $M$, the diagram
\[
	\xymatrix{
		(M \otimes_R S) \otimes_S S \ar[r] \ar[d] & M \otimes_R S \ar[d] \\
		(M \otimes_R S) \otimes_S (R \otimes_R S) \ar[r] & (M \otimes_R R) \otimes_R S
	}
\]
of natural homomorphisms is commutative. Along both ways an element $(m \otimes s_1) \otimes s_2$ is mapped to $(m \otimes 1) \otimes s_1s_2$. Locally the right square of \autoref{Carout} is just a special case of this diagram. 

Now let $(N,\gamma_N)$ be a $\gamma$-sheaf. By definition, the structural morphism $\gamma_N'$ of $N \otimes \omega_X \otimes \omega_X^{-1}$ is the line in the middle of the diagram \autoref{Gam}. The upper squares of this diagram commute by construction. Here the horizontal morphism to the top right corner is given by $\id \otimes (\can^{-1} \circ \tilde{\kappa}_X)$ and the horizontal morphism below is the unique morphism making the upper right square commute. Therefore we see that $\gamma_N'$ is the tensor product of $\gamma_N$ and $\gamma_{\omega}$, i.e.\ the bottom rectangle of \autoref{Gam} is commutative. 

Now consider the diagram
\[
	\xymatrix@C50pt@R30pt{
		N \otimes \omega \ar[r]^-{\id \otimes \gamma_{\omega}} & N \otimes F^* \omega \ar[r]^-{\gamma_N \otimes \id} & F^*N \otimes F^* \omega \ar[r]^-{\sim} & F^*(N \otimes \omega) \\
		N \otimes \CO_X \ar[u]^{\id \otimes \delta} \ar[r]^-{\id \otimes \gamma_X} & N \otimes F^*\CO_X \ar[u]^{\id \otimes F^*\delta} \ar[r]^-{\gamma_N \otimes \id} & F^*N \otimes F^*\CO_X \ar[u]^{\id \otimes F^*\delta} \ar[r]^-{\sim} & F^*(N \otimes \CO_X) \ar[u]^{F^*(\id \otimes \delta)} \\
		N \ar[u]^{\sim} \ar[rr]^-{\gamma_N} & & F^*N \otimes \CO_X \ar[u]^{\id \otimes \gamma_X} \ar[r]^-{\sim} & F^*N. \ar[u]^{\sim} 
	}
\]
The upper left square is the commutative diagram of \autoref{GenGamma} tensored with $N$. The upper square in the middle and the bottom left rectangle are clearly commutative. The square to the left of it commutes because of the naturality of the isomorphism $F^*N \otimes F^*(\usc) \overset{\sim}{\longrightarrow} F^*(N \otimes \usc)$. We already have seen that the bottom right square commutes: It is the same square as the left one of diagram \autoref{Carout} with $M$ replaced by $N$. Moreover, the composition of the leftmost arrow is $N \otimes \delta$ and the composition of the rightmost arrow is $F^*(N \otimes \delta)$.

Hence the structural morphism of $N$ is compatible with $\gamma_N \otimes \gamma_{\omega}$, which turned out to be compatible with the structural morphism induced from the Cartier module $N \otimes \omega_X$. Thus we can extract the commutative diagram
\[
	\xymatrix@C50pt{
		N \otimes \omega_X \otimes \omega_X^{-1} \ar[r]^-{\gamma_N'} & F^*(N \otimes \omega_X \otimes \omega_X^{-1}) \\
		N \ar[u]^{\id \otimes \delta} \ar[r]^-{\gamma_N} & F^*N. \ar[u]^{F^*(\id \otimes \delta)}
	}
\]
It follows that the functors $\usc \otimes \omega_X^{-1}$ and $\usc \otimes \omega_X$ are inverse equivalences. From \autoref{GenGamma} we know that $\usc \otimes \omega_X^{-1}$ maps $\omega_X$ with the structural morphism $\kappa_X$ to $\CO_X$ with the structural morphism $\gamma_X$. Consequently, $\usc \otimes \omega_X$ maps the $\gamma$-sheaf $(\CO_X, \gamma_X)$ to the Cartier module $(\omega_X, \kappa_X)$. 
\end{proof}
\begin{corollary} \label{QCrysCartierGamma}
Tensoring with $\omega_X$ and with $\omega_X^{-1}$ induces equivalences of categories
\[
	\xymatrix@C70pt{
	\QCrysC(X) \ar@<0.1cm>[r]^-{\usc \otimes_{\CO_X}\omega_X^{-1}} & \QCrysG(X) \ar@<0.1cm>[l]^-{\usc \otimes_{\CO_X}\omega_X}
	}
\]
and
\[
	\xymatrix@C70pt{
	\CrysC(X) \ar@<0.1cm>[r]^-{\usc \otimes_{\CO_X}\omega_X^{-1}} & \CrysG(X) \ar@<0.1cm>[l]^-{\usc \otimes_{\CO_X}\omega_X}
	}
\]
for every regular, $F$-finite scheme $X$. 
\end{corollary}
\begin{corollary}
If $X$ is regular and $F$-finite, the categories $\QCohG(X)$ and $\QCrysG(X)$ have enough injectives.
\end{corollary}
\begin{proof}
This follows from \autoref{enoughinjectiveskappa} and \autoref{QCrysCartierGamma}.
\end{proof}
\begin{landscape}
\begin{align} \label{Car} 
	\xymatrix{
		(F^{\flat}\CO_X \otimes (F^{\flat}\CO_X)^{-1}) \otimes F^*M \otimes F^*\omega_X^{-1} \otimes F^{\flat}\CO_X \otimes F^*\omega_X \ar[r]^-{\ev} & F^*M \otimes F^*\omega_X^{-1} \otimes F^{\flat}\CO_X \otimes F^*\omega_X \ar[d]_{\sim} & \\
		F^{\flat}\CO_X \otimes F^*M \otimes (F^{\flat}\CO_X)^{-1} \otimes F^*\omega_X^{-1} \otimes F^{\flat}\CO_X \otimes F^*\omega_X \ar[u]_{\sim} \ar[r]^-{\mu} & F^{\flat}\CO_X \otimes F^*(M \otimes \omega_X^{-1} \otimes \omega_X) \ar[r]^-{\can} & F^{\flat}(M \otimes \omega_X^{-1} \otimes \omega_X) \\
		M \otimes \omega_X^{-1} \otimes \omega_X \ar[u]^{\can^{-1}(\kappa_M \otimes ({\kappa}_X^{\vee})^{-1} \otimes \kappa_X)} & & \\
		M \ar[u]_{\sim}^{\id \otimes \delta} \ar[r]^-{\can^{-1} \circ \kappa_M} & F^{\flat}\CO_X \otimes F^*M \ar[uu]^{F^{\flat}\CO_X \otimes F^*(\id \otimes \delta)} \ar[r]^-{\can} & F^{\flat}M \ar[uu]_{F^{\flat}\phi}
	}
\end{align}
\vspace{2cm}
\begin{align} \label{Gam}
	\xymatrix{
		F^*N \otimes F^{\flat}\CO_X \otimes F^*\omega_X \otimes \omega_X^{-1} \ar[r]^-{\sim} & F^{\flat}\CO_X \otimes F^*N \otimes F^*\omega_X^{-1} \otimes \omega_X^{-1} \ar[r] \ar[d]_{\can}^{\sim} & F^{\flat}\CO_X \otimes F^*N \otimes F^*\omega_X^{-1} \otimes (F^{\flat}\CO_X)^{-1} \otimes F^*\omega_X^{-1} \ar[d]_{\sim}^{\ev_{F^{\flat}\CO_X} \otimes \id} \\
		N \otimes \omega_X \otimes \omega_X^{-1} \ar[u]^{\gamma_N \otimes (\can^{-1} \circ \kappa_X) \otimes \id} \ar[r]^-{\kappa_{N \otimes \omega_X}} \ar@{=}[d]& F^{\flat}(N \otimes \omega_X) \otimes \omega_X^{-1} \ar[r] & F^*(N \otimes \omega_X \otimes \omega_X^{-1}) \ar@{=}[d] \\
		N \otimes \omega_X \otimes \omega_X^{-1} \ar[r]^-{\gamma_N \otimes \gamma_{\omega}} & F^*N \otimes F^*(\omega_X \otimes \omega_X^{-1}) \ar[r]^-{\sim} & F^*(N \otimes \omega_X \otimes \omega_X^{-1})
	}
\end{align}
\end{landscape}

\subsection{Compatibility with pull-back}

The pull-back of quasi-coherent sheaves defines a \emph{pull-back functor on $\gamma$-sheaves}:
\begin{definition}
Let $f\colon X \to Y$ be a morphism of regular schemes and $N$ a $\gamma$-sheaf on $Y$ with structural morphism $\gamma_N$. The $\gamma$-structure for $f^*N$ is defined as the composition 
\[
	f^*N \xrightarrow{f^*\gamma_N} f^*F_Y^*N \overset{\sim}{\longrightarrow} F_Y^*f^* N.
\]
\end{definition} 
First we consider a closed immersion $i\colon X \to Y$ of regular $F$-finite schemes. The aim is to prove the following theorem:
\begin{theorem} \label{pullbackcomp}
Let $i\colon X \to Y$ be a closed immersion of regular, $F$-finite schemes with codimension $n$. Then there is a canonical isomorphism of functors
\[
	\usc \otimes \omega_X^{-1} \circ R^ni^{\flat} \cong i^* \circ \usc \otimes \omega_Y^{-1} 
\]
inducing a corresponding isomorphism of functors of crystals, i.e.\ the diagram
\[
	\xymatrix@R30pt@C70pt{
		\Crys_{\kappa}(Y) \ar[d]^{R^ni^{\flat}} \ar@<1mm>[r]^-{\usc \otimes_{\CO_Y} \omega_Y^{-1}} &  \Crys_{\gamma}(Y) \ar@<1mm>[l]^-{\usc \otimes_{\CO_Y} \omega_Y} \ar[d]^{i^*} \\
		\Crys_{\kappa}(X) \ar@<1mm>[r]^-{\usc \otimes_{\CO_X} \omega_X^{-1}} & \Crys_{\gamma}(X) \ar@<1mm>[l]^-{\usc \otimes_{\CO_X} \omega_X}
	}
\]
is commutative. 
\end{theorem}
We begin with the affine case. Noting that any closed immersion of regular schemes is a local complete intersection morphism, it suffices to consider the case of a complete intersection, where the pull-back of a Cartier module can be computed by using the Koszul complex. 
\begin{lemma} \label{Koszul}
Let $\underline{f}=f_1,f_2, \dots ,f_n$ be a regular sequence of elements of a commutative ring $R$. Let $I$ be the ideal generated by the $f_i$. Then for every Cartier module $M$ with structural map $\kappa$, there is an isomorphism
\[
	\phi_{\underline{f}}\colon \Ext_R^n(R/I,M) \overset{\sim}{\longrightarrow} M/IM
\]
where $M/IM$ is viewed as an $R/I$-module with the Cartier structure
\begin{align*}
	\kappa_{M/IM}\colon F_*M/IM & \to M/IM \\
	m+IM & \mapsto \kappa_M((f_1\cdot f_2 \cdots f_n)^{p-1}m) + IM.
\end{align*}
\end{lemma}
\begin{proof}
By definition we have to compute $\RHom_R(R/I,M)$. The structural morphism $\kappa_{i^{\flat}M}\colon F_*i^{\flat}M \to i^{\flat}M$ equals the composition 
\[
	F_*\RHom_R(R/I,M) \to \RHom_R(F_*R/I,F_*M) \to \RHom_R(R/I,M),
\]
where the first morphism is the canonical one and the second is induced by $F\colon R/I \to F_*R/I$ in the first and $\kappa\colon F_*M \to M$ in the second argument. A free resolution of the $R$-module $R/I$ is given by the Koszul chain complex $\CK(\underline{f})$. It is the total tensor product complex in the sense of \cite[2.7.1]{Weibel} of the following complexes $\CK(f_i)$ 
\[
	0 \longrightarrow R \overset{f_i}{\longrightarrow} R \longrightarrow 0
\]
concentrated in degrees $-1$ and $0$. Each complex $\CK(f_i)$ admits a lift of the Frobenius $F\colon R \to F_*R$ in degree $0$ by mapping $r$ to $r^pf_i^{p-1}$. This means the diagrams
\[
	\xymatrix{
		0 \ar[r] & R \ar[d]^{f_i^{p-1}F} \ar[r]^-{f_i} & R \ar[d]^F \ar[r] & 0 \\
		0 \ar[r] & F_*R \ar[r]^-{f_i} & F_*R \ar[r] & 0
	}
\]
commute. The maps $R \xrightarrow{f_i^{p-1}} F_*R$ give rise to a map of complexes $F\colon \CK(\underline{f}) \to F_* \CK(\underline{f})$ lifting the Frobenius in degree zero: In general, if $\phi\colon M \to F_*M$ and $\psi\colon N \to F_*N$ are $R$-linear maps, it is easy to check that the map
\begin{align*}
	M \otimes_R N &\to M \otimes_R N \\
	m \otimes n &\mapsto \phi(n) \otimes \psi(n)
\end{align*}
of abelian groups is $p$-linear, i.e.\ it is an $R$-linear map $M \otimes N \to F_*(M \otimes N)$. Thus we inductively obtain an $R$-linear morphism of complexes $F\colon \CK(\underline{f}) \to F_* \CK(\underline{f})$. Unwinding the definition of the tensor product of complexes, we see that the left end of this map is the square
\[
	\xymatrix@R30pt@C60pt{
		0 \ar[r] & R \ar[d]_{\prod_{i=1}^nf_i^{p-1}} \ar[r]^-{((-1)^{i+1}f_i)_i} & R^n \ar[d]^{(\prod_{j \neq i}f_j^{p-1})_i} \ar[r] & \cdots \\
		0 \ar[r] & F_*R \ar[r]^-{((-1)^{i+1}f_i)_i} & F_*R^n \ar[r] & \cdots 
  }     
\]
Consequently, the $n$-th degree of the composition 
\[
	F_*\Hom_R(\CK(\underline{f}),M) \xrightarrow{\text{canonical}} \Hom_R(F_*\CK(\underline{f}),F_*M) \xrightarrow{F \circ \usc \circ \kappa} Hom_R(\CK(\underline{f}),M)
\]
maps $m$ to $\kappa((\prod f_i^{p-1})m)$ by the identification $\Hom_R(R,M) \cong M$ via $\phi \mapsto \phi(1)$. The differential $\Hom_R(\CK_{n-1}(\underline{f}),M) \to \Hom_R(\CK_n(\underline{f}),M)$ corresponds to the map 
\begin{align*} 
	M^n &\to M \\
	(m_i)_i &\mapsto \sum f_im_i.
\end{align*}
The image is the submodule $IM$ and thus the n-th cohomology of $\Hom_R(\CK(\underline{f}),M)$ is isomorphic to $M/IM$, equipped with the claimed Cartier structure.   
\end{proof}
\begin{remark} \label{RegSequenceChange}
The isomorphism $\Ext_R^n(R/I,M) \to M/IM$ of the underlying sheaves is not canonical. It depends on the choice of the regular sequence $\underline{f}$. From the construction of this isomorphism we see that if $\underline{g}=g_1, \dots ,g_n$ is another regular sequence of $R$ generating $I$ and $g_i=\sum c_{ij}f_j$, the automorphism $\tau$ on $M/IM$ making the diagram
\[
	\xymatrix{ 
		& M/IM \ar[dd]^{\tau} \\
		\Ext_R^n(R/I,M) \ar[ur]^{\phi_{\underline{f}}} \ar[dr]_{\phi_{\underline{g}}} & \\
		& M/IM
	}
\]
commutative is given by multiplication with $\det(c_{ij})$.
\end{remark}
Nevertheless, interpreting the top-$\Ext$-groups as quotient modules in the case we are interested in, namely $\Ext_R^n(R/I,M) \otimes \Ext_R^n(R/I,\omega_R)^{-1}$, leads to isomorphisms, which are independent of the regular sequence generating $I$, since the correcting factors from both terms cancel.   
\begin{lemma} \label{localGammapullback}
Let $R$ be a commutative ring such that $F_*R$ is finite free, $N$ a $\gamma$-module over $R$ and $I\subseteq R$ an ideal which is generated by a regular sequence of length $n$. There is a canonical isomorphism between Cartier modules 
\[
	\Ext_R^n(R/I,N \otimes \omega_R) \cong N/IN \otimes \omega_{R/I}.
\]
\end{lemma}
\begin{proof}
Choose a regular sequence $\underline{f}=f_1, \dots ,f_n$ such that $I$ is generated by the $f_i$. Also choose a basis $r_1,\dots,r_t$ of $R$ viewed as a free $R$-module via the Frobenius. The dual basis $\phi_1, \dots, \phi_t$ is given by $\phi_i(r_j) = \delta_{ij}$. Let $\kappa$ denote the intrinsic Cartier structure of $\Ext_R^n(R/I,N \otimes \omega_R)$ and $\tilde{\kappa}$ denote the Cartier structure of $N/IN \otimes \omega_{R/I}$ as explained in \autoref{CartierGamma}. Identifying $\Ext_R^n(R/I,N \otimes \omega_R)$ with $(N \otimes \omega_R) / I(N \otimes \omega_R)$ via the isomorphism $\phi_{\underline{f}}$ from \autoref{Koszul}, we obtain the map 
\begin{align*}
	\kappa'\colon F_*((N \otimes \omega_R) / I(N \otimes \omega_R)) &\to (N \otimes \omega_R) / I(N \otimes \omega_R) \\
	n \otimes m + I(N \otimes \omega_R) &\mapsto \sum_{i=1}^t\phi_i(1)\gamma_N(n) \otimes \kappa_R(r_if^{p-1}m) + I(N \otimes \omega_R)
\end{align*}
as the induced Cartier structure on $(N \otimes \omega_R) / I(N \otimes \omega_R)$. Here $\gamma_N$ is the $\gamma$-structure of $N$. 

Also identifying $\omega_{R/I} \cong \Ext_R^n(R/I, \omega_R)$ with $\omega_R/I\omega_R$ via $\underline{f}$, its Cartier structure $\kappa_{R/I}$ is given by $m + I\omega_Y \mapsto \kappa_R(f^{p-1}m) + I(\omega_Y)$. From this perspective, the Cartier structure of $\tilde{\kappa}$ of $N/IN \otimes \omega_{R/I}$ induces the structural morphism 
\begin{align*}
	\tilde{\kappa}'\colon F_*(N/IN \otimes \omega_R / I\omega_R) &\to (N/IN \otimes \omega_R / I\omega_R) \\
	\overline{n} \otimes \overline{m} &\mapsto \sum_{i=1}^n\overline{\phi_i(1)} \cdot \overline{\gamma_N(n)} \otimes \overline{\kappa_R(r_if^{p-1}m)} \\
	&= \sum_{i=1}^t\phi_i(1)\gamma_N(n) \otimes \kappa_R(r_if^{p-1}m) + I(N \otimes \omega_R)
\end{align*}
on $N/IN \otimes \omega_R / I\omega_R$. Finally, there is a natural isomorphism 
\[
	\tau\colon (N \otimes_R \omega_R) / I(N \otimes_R \omega_R) \overset{\sim}{\longrightarrow} N/IN \otimes_{R/I} \omega_R/I \omega_R
\]
of $R$-modules, mapping $\overline{n \otimes m}$ to $\overline{n} \otimes \overline{m}$. The explicit formulas for $\kappa'$ and $\tilde{\kappa}'$ show that $\tau$ makes the square in the middle of the diagram
\[
	\xymatrix{
		\Ext_R^n(R/I,N \otimes_R \omega_R) \ar[d]^{\phi_{\underline{f}}} \ar[r]^-{\kappa} & F^{\flat} \Ext_R^n(R/I,N \otimes_R \omega_R) \ar[d]^{F^{\flat}\phi_{\underline{f}}} \\
		(N \otimes \omega_R)/I(N \otimes \omega_R) \ar[r]^-{\kappa'} \ar[d]^{\tau} & F^{\flat}((N \otimes \omega_R)/I(N \otimes \omega_R)) \ar[d]^{F^{\flat} \tau} \\
		N/IN \otimes \omega_R /I \omega_R \ar[r]^-{\tilde{\kappa}'} & F^{\flat}(N/IN \otimes \omega_R /I \omega_R) \\
		i^*N \otimes_R \omega_{R/I} \ar[r]^-{\tilde{\kappa}} \ar[u]_{\id \otimes \phi_{\underline{f}}}&  F^{\flat}(i^*N \otimes_R \omega_{R/I}) \ar[u]_{F^{\flat}(\id \otimes \phi_{\underline{f}})}
	}
\]
commutative. The squares above and below commute by construction. Let $\Phi$ be the composition $(\id \otimes \phi_{\underline{f}}^{\omega_R})^{-1} \circ \tau \circ \phi_{\underline{f}}^{N \otimes \omega_R}$. We have just seen that the diagram
\[
	\xymatrix{
		 \Ext_R^n(R/I,N \otimes_R \omega_R) \ar[d]^{\Phi} \ar[r]^-{\kappa} & F^{\flat} \Ext_R^n(R/I,N \otimes_R \omega_R) \ar[d]^{F^{\flat}\Phi} \\
		i^*N \otimes_R \omega_{R/I} \ar[r]^-{\tilde{\kappa}} &  F^{\flat}(i^*N \otimes_R \omega_{R/I})
	}
\]	
commutes. Furthermore, $\Phi$ is natural: Let $\underline{g} = g_1, \dots ,g_n$ be another regular sequence generating $I$ with $g_i = \sum c_{ij}f_j$. Then, by \autoref{RegSequenceChange}, 
\begin{align*}
	(\id \otimes \phi_{\underline{g}}^{\omega_R})^{-1} \circ \tau \circ \phi_{\underline{g}}^{N \otimes \omega_R} &= \det(c_{ij})^{-1}(\id \otimes  \phi_{\underline{f}}^{\omega_R})^{-1} \circ \tau \circ \det(c_{ij}) \phi_{\underline{f}}^{N \otimes \omega_R} \\
	&= (\id \otimes \phi_{\underline{f}}^{\omega_R})^{-1} \circ \tau \circ \phi_{\underline{f}}^{N \otimes \omega_R}.
\end{align*}
\end{proof}
\begin{proposition} \label{Gammapullback}
Let $i\colon X \into Y$ be a closed immersion of regular, $F$-finite schemes and let $N$ be a $\gamma$-sheaf on $Y$. There is a canonical isomorphism
\[
	\Phi\colon \overline{i}^*\ExtS_{\CO_Y}^n(i_*\CO_X,N \otimes_{\CO_Y} \omega_Y) \overset{\sim}{\longrightarrow} i^*N \otimes_{\CO_X} \omega_X
\]
of Cartier modules, which is functorial in $N$. Here $\overline{i}$ denotes the flat morphism of ringed spaces $(X,\CO_X) \to (Y,i_*\CO_X)$.  
\end{proposition}
\begin{proof}
Choose an affine open covering $\{U_k\}_k = \Spec R_k$ of $Y$ such that $i_{U_k}\colon i^{-1}(U_k) \into U_k$ is a complete intersection. By refining the covering we can assume that $F_*R_k$ is free. Let $I_k \subseteq R_k$ be the ideal such that $i|_{U_k}$ corresponds to the ring homomorphism $R_k \to R_k/I_k$. By \autoref{localGammapullback}, we have an isomorphism $\Ext_R^n(R_k/I_k, N|_{U_k}\otimes \omega_{R_k})  \cong i^*N|_{U_k} \otimes \omega_{R_k/I_k}$, which is natural and therefore, we can glue the local isomorphisms to the desired global map $\Phi$. 
\end{proof}
\autoref{Gammapullback} shows that there is a natural isomorphism of functors $R^ni^{\flat} \circ \usc \otimes_{\CO_Y} \omega_Y \cong \usc \otimes_{\CO_X} \omega_X \circ i^*$. This enables us to prove \autoref{pullbackcomp} because $\usc \otimes_{\CO_Y} \omega_Y$ and $\usc \otimes_{\CO_X} \omega_X$ are equivalences of categories.
\begin{proof}[Proof of \autoref{pullbackcomp}]
For every $\gamma$-sheaf $N$ on $Y$, there is an isomorphism
\begin{align*}
	R^ni^{\flat}(N \otimes_{\CO_Y} \omega_Y) &\cong i^{-1}\ExtS_{\CO_Y}^n(i_*\CO_X,N \otimes_{\CO_Y} \omega_Y) \\
	&\cong i^*N \otimes_{\CO_X} \omega_X
\end{align*}
by \autoref{Gammapullback}, which is functorial in $N$. As $\usc \otimes_{\CO_Y} \omega_Y$ and $\usc \otimes_{\CO_X} \omega_X$ are equivalences of categories, even the diagram
\[
	\xymatrix@R30pt@C70pt{
		\Coh_{\kappa}(Y) \ar[d]^{R^ni^!} \ar@<1mm>[r]^-{\usc \otimes_{\CO_Y} \omega_Y^{-1}} &  \Coh_{\gamma}(Y) \ar@<1mm>[l]^-{\usc \otimes_{\CO_Y} \omega_Y} \ar[d]^{i^*} \\
		\Coh_{\kappa}(X) \ar@<1mm>[r]^-{\usc \otimes_{\CO_X} \omega_X^{-1}} & \Coh_{\gamma}(X) \ar@<1mm>[l]^-{\usc \otimes_{\CO_X} \omega_X}
	}
\]
commutes. Passing to crystals finishes the proof.
\end{proof}
Now we turn to open immersions. 
\begin{proposition}\label{Gopencomp}
Let $j\colon U \to X$ be an open immersion of regular, $F$-finite schemes and $M$ a Cartier module on $X$. Then there is a natural isomorphism of $\gamma$-sheaves 
\[
	j^!M \otimes_{\CO_U} \omega_U^{-1} \cong j^*(M \otimes \omega_X^{-1}).
\]
\end{proposition}
\begin{proof}
One easily checks that, for a Cartier module $M$ on $X$ with structural morphism $\tilde{\kappa}\colon M \to F_X^{\flat}M$, the Cartier structure on $j^*M$ is the composition
\[
	j^*M \xrightarrow{j^*\tilde{\kappa}} j^*F^{\flat}M \overset{\sim}{\longrightarrow} F_U^{\flat}j^*M.
\]
The dualizing sheaf $\omega_U$ of $U$ is given by $j^*\omega_X$. Therefore we have $j^!M \otimes \omega_U^{-1} \cong j^*(M \otimes \omega_X^{-1})$ and the diagram
\[
	\xymatrix{
		j^*M \otimes j^*\omega_X^{-1} \ar[d]^{\sim} \ar[r] & F^{\flat}j^*M \otimes F^{\flat}j^*\omega_X^{-1} \ar[d]^{\sim} \ar[r]^-{\sim} & F^*(j^*M \otimes j^*\omega_X^{-1}) \ar[d]^{\sim} \\
		j^*(M \otimes \omega_X^{-1}) \ar[r] & j^*(F^{\flat}M \otimes F^{\flat} \omega_X^{-1}) \ar[r]^-{\sim} & j^*F^*(M \otimes \omega_X^{-1})
	}
\]
commutes. Here the horizontal arrows are the $\gamma$-structures of $j^*M \otimes_{\CO_U} \omega_U^{-1}$ and $j^*(M \otimes \omega_X^{-1})$. 
\end{proof}

\subsection{Compatibility with push-forward}

For the construction of a push-forward for $\gamma$-sheaves, we follow the construction given in subsection 6.3 of \cite{BliBoe.CartierEmKis}. Then we show that the equivalence between Cartier modules and $\gamma$-sheaves given by tensoring with the dualizing sheaf is compatible with push-forward for morphisms of regular schemes. This proof is also mainly the one given in ibid. By abuse of notation, let $\kappa_X\colon F_{X*}\omega_X \to \omega_X$ be the adjoint of the Cartier structure of $\omega_X$.   
\begin{definition}[\protect{\cite[Definition 6.3.1]{BliBoe.CartierEmKis}}]
Let $f\colon X \to Y$ be a morphism of smooth, $F$-finite $k$-schemes. Let $N$ be a $\gamma$-sheaf on $X$. Then we define the \emph{push-forward $f_+N$} as the twist of the push-forward $f_*$ of Cartier modules, i.e.
\[
	f_+N = f_*(N \otimes_{\CO_X} \omega_X) \otimes_{\CO_Y} \omega_Y^{-1}.
\]
The push-forward for $\gamma$-crystals is the one induced by the just given push-forward of $\gamma$-sheaves.
\end{definition}
By construction, the push-forward for $\gamma$-sheaves is compatible with the push-forward for Cartier modules. In order to show that $f_+$ is compatible with the equivalence between $\gamma$-crystals and lfgu modules, we need a different description of $f_+$ for $\gamma$-sheaves based on the relative Cartier operator. 

We recall two general constructions, which are repeatedly used in this subsection. For this we consider a morphism $f\colon X \to Y$ of arbitrary schemes over $\Spec \ZZ$. Let $\CF$ be a quasi-coherent $\CO_X$-module and $\CE$ a quasi-coherent $\CO_Y$-module. The adjoint of the composition
\[
	f^*(f_*\CF \otimes_{\CO_X} \CE) \cong f^*f_*\CF \otimes f^*\CE \xrightarrow{\widetilde{\ad}_f \otimes \id} \CF \otimes f^*\CE
\]
where $\widetilde{\ad}_f\colon f^*f_* \to \id$ is the counit of the adjunction, yields a natural morphism 
\[
	\operatorname{proj}\colon f_*\CF \otimes_{\CO_X} \CE \to f_*(\CF \otimes f^*\CE).
\]
As a consequence of the projection formula (\cite[Exercise III.8.3]{Hartshorne}), it is an isomorphism if $f$ is quasi-compact and separated and if $\CE$ is locally free. 

For two morphisms $f\colon X \to S$ and $g\colon Y \to S$, let $f'\colon X \times_S Y \to Y$ and $g'\colon X \times_S Y \to X$ be the projections such that the square 
\[
	\xymatrix{
		X \times_S Y \ar[r]^-{f'} \ar[d]^{g'} & Y \ar[d]^g \\
		X \ar[r]^f & S
	}
\]
is cartesian. There is a canonical morphism of functors 
\[
	\bc\colon f^*g_* \to g'_*f'^*
\]
of quasi-coherent sheaves given by the adjoint of the composition
\[
	g_* \xrightarrow{g_* \ad_{f'}} g_*f'_*f'^* \cong f_*g'_*f'^*,
\]
where $\ad_{f'}\colon \id \to f'_*f'^*$ is the unit of the adjunction. If $g$ is affine, $\bc$ is an isomorphism. To see this, we can assume that $S$, $X$ and $Y$ are affine, because $g$ and therefore $g'$ is an affine morphism. Then the claim is a well known property of the tensor product. The morphism $\bc$ is also an isomorphism if $X$ and $Y$ are Noetherian, $f$ is flat and $g$ is separated of finite type (\cite[Proposition III.9.3]{Hartshorne}).   
The next lemma relates these two isomorphisms. 
\begin{lemma} \label{twoprojection}
Let $f\colon X \to S$ and $g\colon Y \to S$ be morphisms of schemes and let $f'\colon X \times_S X \to Y$ and $g'\colon X \times_S Y \to X$ be the projections. Then, for every quasi-coherent $\CO_X$-module $\CF$ and every quasi-coherent $\CO_S$-module $\CE$, the diagram
\[
	\xymatrix{
		f_*\CF \otimes g_*\mathcal E \ar@{=}[r] \ar[d]^{\proj} & f_*\CF \otimes g_*\mathcal E \ar[d]^{\proj} \\
		f_*(\CF \otimes f^*g_*\mathcal E) \ar[d]^{\bc} & g_*(g^*f_*\CF \otimes \mathcal E) \ar[d]^{\bc} \\
		f_*(\CF \otimes g'_*f'^*\mathcal E) \ar[d]^{\proj} & g_*(f'_*g'^*\CF \otimes \mathcal E) \ar[d]^{\proj} \\
		f_*g'_*(g'^*\CF \otimes f'^*\mathcal E) \ar[r]^-{\sim} & g_*f'_*(g'^*\CF \otimes f'^*\mathcal E)
	}
\]
commutes.
\end{lemma} 
\begin{proof}
For a morphism $h$ of schemes, let $\widetilde{\ad}_h$ denote the counit of adjunction $h^*h_* \to \id $. The diagram of which we want to prove the commutativity is obtained by adjunction from the diagram
\[
	\xymatrix{
		g'^*f^*f_*\CF \otimes g'^*f^*g_*\CE \ar[r]^-{\sim} \ar[dd]_{\widetilde{\ad}_f \otimes \id} & f'^*g^*f_*\CF \otimes f'^*g^*g_*\CE \ar[d]^{\bc \otimes \id} \\
		& f'^*f'_*g'^*\CF \otimes f'^*g^*g_*\CE \ar[d]^{\widetilde{\ad}_{f'} \otimes \id} \\
		g'^*\CF \otimes g'^*f^*g_*\CE \ar[r]^-{\sim} \ar[d]_{\id \otimes \bc} & g'^*\CF \otimes f'^*g^*g_*\CE \ar[dd]^{\id \otimes \widetilde{\ad}_g} \\
		g'^*\CF \otimes g'^*g'_*f'^*\CE \ar[d]_{\id \otimes \widetilde{\ad}_{g'}} & \\
		g'^*\CF \otimes f'^*\CE \ar@{=}[r] & g'^*\CF \otimes f'^*\CE.
	}
\]
Both parts of the diagram are commutative by construction of the morphism $\bc$. 
\end{proof}		
We turn back to the situation of a morphism $f\colon X \to Y$ of smooth schemes over a field $k$ containing $\FF_p$. For simplicity, let $\omega_f$ denote the relative dualizing sheaf $\omega_{X/Y}=f^!\CO_Y$.
\begin{lemma} \label{alterforward}
Let $f\colon X \to Y$ be a morphism of smooth, $F$-finite schemes. For every $\gamma$-sheaf $N$ on $X$, there is a natural isomorphism
\[
	f_*(N \otimes_{\CO_X} \omega_X) \otimes_{\CO_Y} \omega_Y^{-1} \to f_*(N \otimes_{\CO_X} \omega_f)
\]
of quasi-coherent sheaves.
\end{lemma} 
\begin{proof}
Since $X$ and $Y$ are smooth, any morphism $X \to Y$ is regular, i.e.\ it is a composition of a closed immersion $X \to W$ such that $X$ is a local complete intersection in $W$, followed by a smooth morphism $W \to Y$. (For a smooth morphism $f$, the graph factorization
\[
	X \xrightarrow{(\id,f)} X \times_k Y \xrightarrow{\pr_Y} Y,
\]
where $\pr_Y$ denotes the projection, satisfies this requirement.) For a closed immersion $i$ we have the isomorphism 
\[
	i^{\flat}(\omega_Y) \cong Li^*\omega_Y \otimes \omega_f[-n] 
\]
of \cite[Corollary III.7.3]{HartshorneRD} and a smooth morphism is quasi-perfect, see \autoref{EssentiallyPerfect} later on. Overall, we see that there are natural isomorphisms  
\[
	\omega_X \cong f^!\omega_Y \overset{\sim}{\longrightarrow} f^!\CO_Y \otimes Lf^*\omega_Y \overset{\sim}{\longrightarrow} \omega_f \otimes f^*\omega_Y.
\]
Now we obtain the desired isomorphism as the composition
\[
	f_*(N \otimes \omega_X) \otimes \omega_Y^{-1} \xrightarrow{\proj} f_*(N \otimes \omega_X \otimes f^*\omega_Y^{-1}) \cong f_*(N \otimes \omega_f).
\]
\end{proof}
With the relative Frobenius diagram
\[
	\xymatrix@C50pt{
		X \ar[r]^-{F_{X/Y}} \ar[dr]_f \ar@/^23pt/[rr]^{F_X} & X' \ar[r]^-{F_Y'} \ar[d]^-{f'} & X \ar[d]^-f \\
		& Y \ar[r]^{F_Y} & Y,
	}
\]
see diagram \autoref{relativeFrobenius} for the notation, we can define a $\gamma$-structure $\gamma_{N,f}$ for $f_*(N \otimes \omega_f)$ by the composition
\begin{align*}
	f_*(N \otimes \omega_f) &\overset{\sim}{\longrightarrow} f'_*F_{X/Y*}(N \otimes \omega_f) \\
	&\overset{\gamma_N}{\longrightarrow} f'_*F_{X/Y*}(F_{X/Y}^*F_Y'^*N \otimes \omega_f) \\
	&\xrightarrow{\proj^{-1}} f'_*(F_Y'^*N \otimes F_{X/Y*}\omega_f) \\
	&\xrightarrow{C_{X/Y}} f'_*(F_Y'^*N \otimes F_Y'^*\omega_f) \\
	&\overset{\bc}{\longrightarrow} F_Y^*f_*(N \otimes \omega_f).
\end{align*}
We will show that $\gamma_{N,f}$ is the structural morphism of $f_+N$ via the isomorphism of \autoref{alterforward}. But first, we clarify how the relative Cartier operator is related to $\kappa_X$ and $\kappa_Y$.
\begin{lemma}
With the notation of the preceding lemma, the composition
\begin{align*}
	F_{Y*}(\omega_f \otimes f^*\omega_Y) &\overset{\sim}{\longrightarrow} F_{Y*}'F_{X/Y*}(\omega_f \otimes F_{X/Y}^*f'^*\omega_Y) \\
	&\xrightarrow{\proj^{-1}} F_{Y*}'(F_{X/Y*}\omega_f \otimes f'^*\omega_Y) \\
	&\xrightarrow{C_{X/Y}} F_{Y*}'(F_Y'^*\omega_f \otimes f'^*\omega_Y) \\
	&\xrightarrow{\proj} \omega_f \otimes F_{Y*}'f'^*\omega_Y \\
	&\overset{\bc}{\longrightarrow} \omega_f \otimes f^*F_{Y*}\omega_Y \\
	&\overset{\kappa_Y}{\longrightarrow} \omega_f \otimes f^*\omega_Y
\end{align*}
is compatible with the Cartier structure of $\omega_X$ under the canonical isomorphism $\omega_X \cong \omega_f \otimes f^*\omega_Y$.
\end{lemma}
\begin{proof}
In the appendix A.2.3. (iii) of \cite{EmKis.Fcrys}, Emerton and Kisin explain how the relative Cartier operators $C_{X/Y}$, $C_{Y/Z}$ and $C_{X/Z}$ are related for a composition $X \overset{f}{\longrightarrow} Y \overset{g}{\longrightarrow} Z$ of morphisms. Our lemma is the special case where $Z=\Spec k$ and $g$ is the structural morphism of the $k$-scheme $Y$.   
\end{proof}
\begin{proposition}
Let $f\colon X \to Y$ be a morphism of smooth, $F$-finite schemes over $k$. Let $N$ be a $\gamma$-sheaf on $X$. The canonical isomorphism 
\[
	f_*(N \otimes \omega_f) \overset{\sim}{\longrightarrow} f_*(N \otimes \omega_X) \otimes \omega_Y^{-1}
\]
of quasi-coherent $\CO_Y$-modules from \autoref{alterforward} is an isomorphism of $\gamma$-sheaves.
\end{proposition}
\begin{proof}
As $\usc \otimes \omega_X$ and $\usc \otimes \omega_Y^{-1}$ are equivalences between the categories of $\gamma$-sheaves and Cartier modules on $X$ and on $Y$, it suffices to show that the canonical isomorphism $f_*(N \otimes \omega_f) \otimes \omega_Y \overset{\sim}{\longrightarrow} f_*(N \otimes \omega_X)$ is an isomorphism of Cartier modules on $Y$. The left hand side of the diagram
\[
	\xymatrix@C13pt{
		F_{Y*}(f_*(N \otimes \omega_f) \otimes \omega_Y) \ar[r]^-{\proj} \ar[d]^{\gamma_N} & F_{Y*}f_*(N \otimes \omega_f \otimes f^*\omega_Y) \ar[d]^{\gamma_N} \\
		F_{Y*}(f_*(F_Y^*N \otimes \omega_f) \otimes \omega_Y) \ar[r]^-{\proj} \ar[d]^{\sim} &  F_{Y*}f_*(F_X^*N \otimes \omega_f \otimes f^*\omega_Y) \ar[d]^{\sim} \\
		F_{Y*}(f'_*F_{X/Y*}(F_{X/Y}^*F_Y'^*N \otimes \omega_f) \otimes \omega_Y) \ar@{=}[d] & F_{Y*}f'_*F_{X/Y*}(F_{X/Y}^*F_Y'^*N \otimes F_{X/Y}^*f'^*\omega_Y \otimes \omega_f) \ar[d]^{\sim} \\
		F_{Y*}(f'_*F_{X/Y*}(F_{X/Y}^*F_Y'^*N \otimes \omega_f) \otimes \omega_Y) \ar[d]^{\proj^{-1}} & F_{Y*}f'_*F_{X/Y*}(F_{X/Y}^*(F_Y'^*N \otimes f'^*\omega_Y) \otimes \omega_f) \ar[d]^{\proj^{-1}} \\
		F_{Y*}(f'_*(F_Y'^*N \otimes F_{X/Y*}\omega_f) \otimes \omega_Y) \ar[d]^{C_{X/Y}} \ar[r]^-{\proj} & F_{Y*}f'_*(F_Y'^*N \otimes f'^*\omega_Y \otimes F_{X/Y*}\omega_f) \ar[d]^{C_{X/Y}} \\
		F_{Y*}(f'_*F_Y'^*(N \otimes \omega_f) \otimes \omega_Y) \ar[d]^{\bc^{-1}} \ar[r]^-{\proj} & F_{Y*}f'_*(F_Y'^*(N \otimes \omega_f) \otimes f'^*\omega_Y) \ar[d]^{\sim} \\
		F_{Y*}(F_Y^*f_*(N \otimes \omega_f) \otimes \omega_Y) \ar[dd]^{\proj^{-1}} & f_*F_{Y*}'(F_Y'^*(N \otimes \omega_f) \otimes f'^*\omega_Y) \ar[d]^{\proj^{-1}} \\
		& f_*(N \otimes \omega_f \otimes F_{Y*}'f'^*\omega_Y) \ar[d]^{\bc^{-1}} \\
		f_*(N \otimes \omega_f) \otimes F_{Y*}\omega_Y \ar[d]^{\kappa_Y} \ar[r]^-{\proj} & f_*(N \otimes \omega_f \otimes f^*F_{Y*} \omega_Y) \ar[d]^{\kappa_Y} \\
		f_*(N \otimes \omega_f) \otimes \omega_Y \ar[r]^-{\proj} & f_*(N \otimes \omega_f \otimes f^*\omega_Y)  
	}
\]
is the structural morphism of the Cartier module $f_*(N \otimes \omega_f) \otimes \omega_Y$. It is easy to see that the right hand side is the structural morphism of the Cartier module $f_*(N \otimes \omega_X) \cong f_*(N \otimes \omega_F \otimes f^*\omega_Y)$. Hence we have to show that the diagram above commutes. The three upper squares and the bottom square commute by the functoriality and the compatibility of the projection formula with compositions of morphisms. The commutativity of the fourth square from above follows from \autoref{twoprojection}.    
\end{proof}

\subsection{Cartier crystals and locally finitely generated unit modules}

The category of $\gamma$-sheaves was just an intermediate step on the way to locally finitely generated unit modules. Recall that there is a functorial way of associating a unit $\CO_X[F]$-module to a $\gamma$-sheaf $N$ on $X$. 
\begin{definition}
Let $\mu_{\operatorname{u}}(X)$ denote the category of unit left $\CO_{F,X}$-modules whose underlying $\CO_X$-module is quasi-coherent. For a smooth $k$-scheme $X$, let $\Gen$ be the functor 
\[
	\QCohG(X) \to \mu_{\operatorname{u}}(X)
\]
which assigns to any quasi-coherent $\gamma$-sheaf $N$ with structural morphism $\gamma\colon N \to F^*N$ the direct limit $\CN$ of
\[
	N \overset{\gamma}{\longrightarrow} F^*N \xrightarrow{F^*\gamma} F^{2*}N \xrightarrow{F^{2*\gamma}} \cdots
\]
together with the inverse of the induced isomorphism $\CN \overset{\sim}{\longrightarrow} F^*\CN$. 
\end{definition}
\begin{lemma} \label{QCrysGammaunit}
Let $X$ be a smooth, $F$-finite $k$-scheme. The functor $\Gen$ is essentially surjective and induces an equivalence of categories
\[
	\QCrysG(X) \overset{\sim}{\longrightarrow} \mu_{\operatorname{u}}(X).
\]
\end{lemma}
\begin{proof}
Let $\Neg\colon \mu_{\operatorname{u}}(X) \to \QCohG(X)$ be the functor which assigns to a quasi-coherent unit $\CO_{F,X}$-module $\CM$ with structural morphism $u\colon F^*\CM \to \CM$ the quasi-coherent $\gamma$-sheaf $\CM$ whose structural morphism is given by the inverse of $u$. Obviously there is a natural isomorphism
\[
	\Gen \circ \Neg \overset{\sim}{\longrightarrow} \id,
\]
whence the surjectivity of $\Gen$.

For a quasi-coherent $\gamma$-sheaf $N$, the corresponding crystals $\CN$ is nil-isomorphic to the corresponding crystal of $\Neg \circ \Gen(N)$. The reason for this is the fact that the structural morphism $\gamma\colon N \to F^*N$ of a quasi-coherent $\gamma$-sheaf  $N$ is a nil-isomorphism: It is immediate that the structural map of the kernel and the cokernel of $\gamma$, interpreted as a morphism of $\gamma$-sheaves, is the zero map.     
\end{proof}
The image of $\Gen$ of the subcategory $\Coh_{\gamma}(X)$ of coherent $\gamma$-sheaves on $X$ is the category $\mu_{\lfgu}(X)$. Indeed, after localizing at nilpotent $\gamma$-sheaves and considering $\gamma$-crystals, $\Gen$ induces an equivalence of categories.
\begin{proposition}[\protect{\cite[Proposition 5.12]{BliBoe.CartierFiniteness}}] \label{Gammalfgu}
For a smooth, $F$-finite $k$-scheme $X$, the functor 
\[
	\Gen_X\colon \Coh_{\gamma}(X) \to \mu_{\lfgu}(X)
\]
factors through $\Crys_{\gamma}(X)$, inducing an equivalence of categories:
\[
	\xymatrix{\Coh_{\gamma}(X) \ar[dr]^{\Gen} \ar[d] & \\
	\Crys_{\gamma}(X) \ar[r]^-{\sim} & \mu_{\lfgu}(X).
	}	
\]
\end{proposition}
\begin{theorem}
Let $X$ be an $F$-finite, smooth $k$-scheme. Let $\G$ denote the composition of the exact functors $\usc \otimes \omega_X^{-1}$ and $\Gen$. It induces an equivalence of derived categories 
\[
	\G\colon D_{\crys}^b(\QCrysC(X)) \to D_{\lfgu}^b(\CO_{F,X}).
\]
\end{theorem}
\begin{proof}
Combining \autoref{QCrysCartierGamma} and \autoref{QCrysGammaunit}, we see that $\G$ induces an equivalence of abelian categories $\QCrysC(X) \to \mu_{\operatorname{u}}(X)$ and therefore an equivalence of derived categories $D^b(\QCrys(X)) \to D^b(\mu_{\operatorname{u}}(X))$. Since $G$ is exact and restricts to an equivalence $\CrysC(X) \to \mu_{\lfgu}(X)$, we obtain an equivalence $D_{\crys}^b(\QCrysC(X)) \to D_{\lfgu}(\mu_{\operatorname{u}}(X))$. 

It remains to show that $D_{\lfgu}^b(\mu_{\operatorname{u}}(X))$ is naturally equivalent to $D_{\lfgu}^b(\CO_{F,X})$. The inclusion $\mu_{\lfgu}(X) \to \mu(X)$ induces an equivalence $D^b(\mu_{\lfgu}) \to D_{\lfgu}^b(\CO_{F,X})$ (\cite[11.6]{EmKis.Fcrys}). As the inclusion $\mu_{\lfgu}(X) \to \mu(X)$ factors through $\mu_{\operatorname{u}}(X)$, this implies an equivalence $D_{\lfgu}^b(\mu_{\operatorname{u}}(X)) \to D_{\lfgu}^b(\CO_{F,X})$.
\end{proof}     
Finally, we prove that the equivalence $\G$ of derived functors is compatible with pull-backs. Note that for a morphism $f\colon X \to Y$ of smooth schemes, the functor 
\[
	f^!\colon D_{\lfgu}^b(\CO_{F,Y}) \to D_{\lfgu}^b(\CO_{F,X})
\]
is obtained from a right-exact functor of abelian categories. 
\begin{definition}
Let $f\colon X \to Y$ be a morphism of smooth $k$-schemes. The \emph{(underived) pull-back} $f^*\CM$ of an $\CO_{F,Y}$-module $\CM$ is given by
\begin{align*}
	f^*\CM &= \CO_{F,X \rightarrow Y} \otimes_{f^{-1}\CO_{F,Y}} f^{-1}\CM \\
	&= \CO_{F,X} \otimes_{f^{-1}\CO_{F,Y}} f^{-1}\CM,	
\end{align*}
cf.\ \autoref{EmKisPullBack}. The pull-back $f^!$ for complexes $\CM^{\bullet}$ of $\CO_{F,Y}$-modules from Definition 2.3.1 of \cite{EmKis.Fcrys} is the left derived functor of $f^*$, shifted by $d_{X/Y}$:
\[
	f^!\CM^{\bullet} = \CO_{F,X \rightarrow Y} \derotimes_{f^{-1}\CO_{F,Y}} f^{-1}\CM^{\bullet}[d_{X/Y}].
\]
\end{definition}
\begin{corollary} \label{lfgupullbackcomp}
Let $f\colon X \to Y$ be a closed immersion of smooth, $F$-finite $k$-schemes of relative dimension $d_{X/Y}=n$. There is a natural equivalence of functors $\CrysC(Y) \to \mu_{\lfgu}(X)$:
\[
	f^* \circ \G_Y \cong \G_X \circ R^nf^{\flat}.
\]
\end{corollary}
\begin{proof}
Consider the following diagram of functors:
\[
	\xymatrix@C50pt{
		\CrysC(Y) \ar[r]^-{\usc \otimes \omega_Y^{-1}} \ar[d]^{R^nf^{\flat}} & \Crys_{\gamma}(Y) \ar[r]^-{\Gen_Y} \ar[d]^{f^*} & \mu_{\lfgu}(Y) \ar[d]^{f^*} \\
		\CrysC(X) \ar[r]^-{\usc \otimes \omega_X^{-1}} & \Crys_{\gamma}(X) \ar[r]^-{\Gen_X} & \mu_{\lfgu}(X).
	}
\]
The left square commutes by \autoref{pullbackcomp}. The right square also commutes because there is a natural isomorphism $\Gen_X \circ f^* \cong f^* \circ \Gen_Y$. For a $\gamma$-sheaf $N$ on $Y$, let $\CN$ denote $\Gen_Y(N)$, which is the direct limit $\varinjlim F_Y^{i*}N$. As direct limits commute with pull-back of quasi-coherent sheaves, we have a natural isomorphism
\begin{align*}
	\Gen_X f^*(N) &= \varinjlim(\CO_X \otimes_{f^{-1}\CO_Y} f^{-1}F_Y^{i*}N) \\
	&\overset{\sim}{\longrightarrow} \CO_X \otimes_{f^{-1}\CO_Y} f^{-1}(\varinjlim F_Y^{i*}N) \\
	&= \CO_X \otimes_{f^{-1}\CO_Y} f^{-1}(\CN).
\end{align*}
One checks that for a left $\CO_{F,Y}$-module $\CM$, the underived pullback $f^*\CM = \CO_{F,X} \otimes_{f^{-1}\CO_{F,Y}} f^{-1}\CM$ is the quasi-coherent sheaf $f^*\CM = \CO_X \otimes_{f^{-1}\CO_Y} f^{-1}\CM$ with the natural morphism $F_X^*f^*\CM \to f^*\CM$ induced by the structural morphism $F_Y^*\CM \to \CM$. Hence $\CO_X \otimes_{f^{-1}\CO_Y} f^{-1}(\CN)$ is isomorphic to $f^*\Gen_Y(N)$.   
\end{proof}
\begin{lemma} \label{flatcon}
Let $i\colon X \to Y$ be a closed immersion of smooth, $F$-finite schemes over $k$. Let $P$ be a locally free left $\CO_{F,Y}$-module. Then
\[
	R^n((\usc \otimes \omega_X^{-1}) \circ i^! \circ (\usc \otimes \omega_Y))P = 0 \text{ for all } n \neq -d_{X/Y},
\]
where $(\usc \otimes \omega_X^{-1}) \circ i^! \circ (\usc \otimes \omega_Y)$ is understood as the composition of functors
\[
	D_{\lfgu}^b(\CO_{F,Y}) \xrightarrow{\usc \otimes \omega_Y} D_{\crys}^b(\QCrysC(Y)) \xrightarrow{i^!} D_{\crys}^b(\QCrysC(X)) \xrightarrow{\usc \otimes \omega_X^{-1}} D_{\lfgu}^b(\CO_{F,X}).
\]
\end{lemma}
\begin{proof}
Locally free left $\CO_{F,Y}$-modules are in particular locally free as quasi-coherent $\CO_Y$-modules. Thus we have
\begin{align*}
	R^n((\usc \otimes \omega_X^{-1}) \circ i^! \circ (\usc \otimes \omega_Y))P &\cong ((\usc \otimes \omega_X^{-1}) \circ R^n i^! \circ (\usc \otimes \omega_Y)) \oplus_{j \in J}\CO_Y \\
	&\cong (\usc \otimes \omega_X^{-1}) \circ R^n i^! \oplus_{j \in J}\omega_Y \\
	&\cong 0
\end{align*}
locally for all $n \neq -d_{X/Y}$ on the underlying quasi-coherent sheaves.
\end{proof}
\begin{theorem}
For closed immersions $i\colon X \to Y$ of smooth, $F$-finite $k$-schemes, the equivalences $D_{\crys}^b(\QCrysC(X)) \to D_{\lfgu}^b(\CO_{F,X})$ and $D_{\crys}^b(\QCrysC(Y)) \to D_{\lfgu}^b(\CO_{F,Y})$ of derived categories induced by $G_X$ and $G_Y$ are compatible with the pull-backs $i^!$, i.e.\ we have a canonical isomorphism
\[
	\G_X \circ i^! \cong i^! \circ \G_Y 
\]
of functors from $D_{\crys}^b(\QCrysC(Y))$ to $D_{\lfgu}^b(\CO_{F,X})$.
\end{theorem} 
\begin{proof}
This is an application of the following general result concerning derived functors: 
\begin{proposition}[\protect{\cite[Proposition I.7.4]{HartshorneRD}}]
Lat $\CA$ and $\CB$ be abelian categories, where $\CA$ has enough injectives, and let $F_1\colon\CA \to \CB$ be an additive functor which has cohomological dimension $\leq n$ on $\CA$. Let $P$ be the set of objects $X$ of $\CA$ such that $R^iF_1(X)=0$ for all $i \neq n$, and assume that every object of $\CA$ is a quotient of an element of $P$. Let $F_2 = R^nF_1$. Then $RF_1$ and $LF_2$ exist, and there is a functorial isomorphism
\[
	RF_1 \overset{\sim}{\longrightarrow} LF_2[-n].
\]
\end{proposition}
First we have to check the requirements. Let $F$ denote the functor $(\usc \otimes \omega_X^{-1}) \circ R^0i^{\flat} \circ (\usc \otimes \omega_Y)$. As done in the proof of \autoref{Koszul}, the derived functors of $R^0i^{\flat}=\overline{i}^*\SHom_{\CO_Y}(i_*\CO_X,\usc)$ may be computed locally by resolving $i_*\CO_X$ by the Koszul complex. Since this complex has length $-d_{X/Y}$, the cohomological dimension of $F$ is smaller or equal $-d_{X/Y}$. As $Y$ is smooth, every left $\CO_{F,Y}$-module is the quotient of a locally free left $\CO_{F,Y}$-module (\cite[Lemma 1.6.2]{EmKis.Fcrys}). Finally, for every locally free left $\CO_{F,Y}$-module $P$, \autoref{flatcon} states that $R^nF(P)=0$ for all $n \neq -d_{X/Y}$. 

It follows from \cite[Proposition I.7.4]{HartshorneRD} that $RF \cong Li^*[d_{X/Y}]=i^!$ because $i^* \cong R^{-d_{X/Y}}F$ (\autoref{pullbackcomp}). Thus 
\[
	(\usc \otimes \omega_X^{-1}) \circ i^! \circ (\usc \otimes \omega_Y) \cong i^!,
\]
i.e.\ the diagram
\begin{align*}
	\xymatrix@R30pt@C70pt{
		D_{\crys}^b(\QCrysC(Y)) \ar[d]^{i^!} \ar@<1mm>[r]^-{\usc \otimes \omega_Y^{-1}} &  D_{\lfgu}^b(\CO_{F,Y}) \ar@<1mm>[l]^-{\usc \otimes \omega_Y} \ar[d]^{i^!} \\
		D_{\crys}^b(\QCrysC(X)) \ar@<1mm>[r]^-{\usc \otimes \omega_X^{-1}} & D_{\lfgu}^b(\CO_{F,X}) \ar@<1mm>[l]^-{\usc \otimes \omega_X}
	}
\end{align*}
is commutative.		 
\end{proof}
\begin{corollary}
For every closed immersion of smooth, $F$-finite $k$-schemes $i\colon X \to Y$, there is a canonical isomorphism
\[
	\G_Y \circ i_* \cong i_* \circ \G_X.
\]
\end{corollary}
\begin{proof}
This follows formally as $\G_X$ and $\G_Y$ are equivalences of categories and since $i_*$ is uniquely determined as a left adjoint functor of $i^!$.
\end{proof}
\begin{proposition} \label{Gopenresult}
Let $j\colon X \to Y$ be an open immersion of smooth, $F$-finite schemes. Then there are natural isomorphisms
\[
	\G_X \circ j^* \cong j^! \circ \G_Y \text{ and } \G_Y \circ Rj_* \cong j_+ \circ \G_X.
\]
\end{proposition}
\begin{proof}
We already have seen that $\usc \otimes {\omega_X}^{-1} \circ j^* \cong j^* \circ \usc \otimes \omega_Y^{-1}$ (\autoref{Gopencomp}) and that $\Gen_X \circ j^* \cong j^! \circ \Gen_Y$ (see the proof of \autoref{lfgupullbackcomp}, this part holds for an arbitrary flat morphism of smooth $k$-schemes). Therefore $\G_X \circ j^* \cong j^! \circ \G_Y$. The rest follows from the adjunction of $Rj_*$ or $j_+$ and $j^*$ or $j^!$. 
\end{proof}
Up to now, we have seen that the equivalence $G$ between Cartier crystals and lfgu modules is compatible with the (derived) push-forward for open and closed immersions by showing the compatibility for the adjoint pull-back functors. In fact, $G$ is compatible with push-forward for arbitrary morphisms of smooth schemes, but we can give a proof only up to the following theorem\footnote{We will not discuss this theorem here as its theoretical background, for example $\infty$-categories, goes beyond the scope of this work. We just note that the requirement that $f_+ \Gen$ lifts to a functor of the corresponding stable $\infty$-categories is satisfied, because $f_+$ is a composition of left and right derived functors, which have this property (\cite[Example 1.3.3.4]{Lurie}).}, which is a result of Lurie, see \cite[Theorem 1.3.3.2]{Lurie}. 
\begin{theorem} \label{metaresult}
Let $F\colon D(\CA) \to D(\CB)$ be a functor between derived categories of abelian categories $\CA$ and $\CB$, which is a morphism of triangulated categories. If $F$ lifts to an exact functor of the stable $\infty$-categories whose homotopy categories are the cohomologically bounded below derived categories $D^+(\CA)$ and $D^+(\CB)$, if $F$ is $t$-left exact for the canonical $t$-structure, i.e.\ $F$ maps $D^{\geq 0}(\CA)$ to $D^{\geq 0}(\CB)$, and if the cohomology of $F(I)$ is concentrated in degree $0$ for every injective object $I$ of $\CA$, then $F$ arises as a right derived functor between the abelian categories $\CA$ and $\CB$.   
\end{theorem}     
\begin{proposition} \label{GammalfguForward}
Let $f\colon X \to Y$ be a morphism of smooth, $F$-finite $k$-schemes. There is a natural isomorphism
\[
	\G_Y \circ Rf_* \to f_+ \circ \G_X  
\]
from $D_{\crys}^b(\QCrysC(X))$ to $D_{\lfgu}^b(\CO_{F,Y})$.
\end{proposition}
\begin{proof}
As $\usc \otimes \omega_Y^{-1} \circ Rf_* \cong Rf_+ \circ \usc \otimes \omega_X^{-1}$ by construction, it suffices to show that there is a natural isomorphism of functors $\Gen Rf_+ \to f_+ \Gen$ from $D_{\crys}^b(\QCrysG(X))$ to $D_{\lfgu}^b(\CO_{F,Y})$. For every complex $N^{\bullet}$ of $\gamma$-sheaves, the complex $\Gen N^{\bullet}$ of quasi-coherent unit $\CO_{F,X}$-modules has a two-term resolution by induced modules, namely the short exact sequence 
\[
	0 \longrightarrow \CO_{F,X} \otimes_{\CO_X} N^{\bullet} \xrightarrow{1 - \beta'} \CO_{F,X} \otimes_{\CO_X} N^{\bullet} \longrightarrow \Gen N^{\bullet} \longrightarrow 0
\]
of \cite[Proposition 5.3.3]{EmKis.Fcrys}. Here $\beta'\colon \CO_{F,X} \otimes_{\CO_X} N^{\bullet} \to \CO_{F,X} \otimes_{\CO_X} N^{\bullet}$ denotes the morphism corresponding to $\beta$ via the identification
\[
	\Hom_{\CO_{F,X}}(\CO_{F,X} \otimes_{\CO_X} A,\CO_{F,X} \otimes_{\CO_X} B) \overset{\sim}{\longrightarrow} \Hom_{\CO_X}(A,\oplus_{n=0}^\infty(F_X^r)^*B)
\]
for $\CO_X$-modules $A$ and $B$ described in 1.7.3 of ibid. 

First we verify that the requirements of \autoref{metaresult} are satisfied. Let $I^{\bullet}$ be a bounded below complex of injective $\gamma$-sheaves with $H^i(I^{\bullet}) = 0$ for $i < 0$. Let $\beta\colon I^{\bullet} \to F_X^*I^{\bullet}$ be the morphism of complexes induced by the structural morphisms of the $I^i$. 

The complex $f_+I^{\bullet}$ represents $Rf_+I^{\bullet}$ and, as explained above, we have a short exact sequence
\[
	0 \longrightarrow \CO_{F,X} \otimes_{\CO_X} f_+I^{\bullet} \xrightarrow{1 - f_+\beta'} \CO_{F,X} \otimes_{\CO_X} f_+I^{\bullet} \longrightarrow \Gen f_+I^{\bullet} \longrightarrow 0.
\]
Applying $f_+$ to the two-term resolution of $\Gen I^{\bullet}$ yields a distinguished triangle
\[
	f_+(\CO_{F,X} \otimes_{\CO_X} I^{\bullet}) \xrightarrow{f_+(1 - \beta')} f_+(\CO_{F,X} \otimes_{\CO_X} I^{\bullet}) \longrightarrow f_+(\Gen I^{\bullet}) \longrightarrow f_+(\CO_{F,X} \otimes_{\CO_X} I^{\bullet})[1].
\]
The sheaf $\CO_{F,Y \leftarrow X}$ is locally free as an $\CO_X$-module. It follows that locally $\CO_{F,Y \leftarrow X} \otimes_{\CO_X} I^{\bullet}$ is a direct sum of flasque sheaves and hence flasque. We have
\begin{align*}
	f_+(\CO_{F,X} \otimes_{\CO_X} I^{\bullet}) &= Rf_*(\CO_{F,Y \leftarrow X} \derotimes_{\CO_{F,X}} (\CO_{F,X} \otimes_{\CO_X} I^{\bullet})) \\
	&\overset{\sim}{\longrightarrow} Rf_*(\CO_{F,Y \leftarrow X} \otimes_{\CO_X} I^{\bullet}) \\
	&\overset{\sim}{\longrightarrow} f_*(\CO_{F,Y \leftarrow X} \otimes_{\CO_X} I^{\bullet}),
\end{align*}
see also \cite[Lemma 3.5.1]{EmKis.Fcrys} and its proof. In particular, the complex $f_+(\CO_{F,X} \otimes_{\CO_X} I^{\bullet})$ is represented by the complex whose $i$-th degree equals the sheaf $f_+(\CO_{F,X} \otimes_{\CO_X} I^i)$.

The canonical isomorphism $\CO_{F,X} \otimes_{\CO_X} f_+I^{\bullet} \overset{\sim}{\longrightarrow} f_+(\CO_{F,X} \otimes_{\CO_X} I^{\bullet})$ of the proof of \cite[Theorem 3.5.3]{EmKis.Fcrys} makes the left hand square of the diagram
\begin{align} \label{lowerplustriangle}
	\xymatrix@C40pt{
		\CO_{F,X} \otimes_{\CO_X} f_+I^{\bullet} \ar[r]^-{1 - f_+\beta'} \ar[d]^{\sim} & \CO_{F,X} \otimes_{\CO_X} f_+I^{\bullet} \ar[r] \ar[d]^{\sim} & \Gen f_+I^{\bullet} \ar@{.>}[d]^{\sim} \\
		f_+(\CO_{F,X} \otimes_{\CO_X} I^{\bullet}) \ar[r]^-{f_+(1 - \beta')} & f_+(\CO_{F,X} \otimes_{\CO_X} I^{\bullet}) \ar[r] & f_+(\Gen I^{\bullet})
	}
\end{align}
commutative (\cite[Proposition 3.6.1]{EmKis.Fcrys}). This shows that the cohomology sheaves of $f_+(\Gen I^{\bullet})$ vanish in negative degrees, i.e.\ $f_+ \Gen$ is left $t$-exact for the canonical $t$-structure of the bounded derived category of $\gamma$-sheaves on $X$. Furthermore, for a single injective $\gamma$-sheaf $I$ on $X$, the upper row of the commutative diagram 
\[
	\xymatrix@C40pt{
		\CO_{F,X} \otimes_{\CO_X} f_+I \ar[r]^-{1 - f_+\beta'} \ar[d]^{\sim} & \CO_{F,X} \otimes_{\CO_X} f_+I \ar[r] \ar[d]^{\sim} & \Gen f_+I \ar@{.>}[d]^{\sim} \\
		f_+(\CO_{F,X} \otimes_{\CO_X} I) \ar[r]^-{f_+(1 - \beta')} & f_+(\CO_{F,X} \otimes_{\CO_X} I) \ar[r] & f_+(\Gen I) 
	}
\]
is a short exact sequence when adding $0$ at the ends. Consequently, the cohomology of $f_+ \Gen I$ is concentrated in degree $0$. 

To see that there is an isomorphism of functors $\Gen f_+ \cong H^0(f_+) \Gen$, let $M$ be a $\gamma$-sheaf on $X$. Choose a resolution $I^{\bullet}$ of $M$ by injective $\gamma$-sheaves. The long exact cohomology sequences for the triangles of the diagram \autoref{lowerplustriangle} yield a unique isomorphism 
\[
	\Gen f_+ M \cong H^0(\Gen Rf_+ M) = H^0(\Gen f_+ I^{\bullet}) \overset{\sim}{\longrightarrow} H^0(f_+ \Gen I^{\bullet}) \cong H^0(f_+ \Gen M).
\]
By \autoref{metaresult}, the functor $f_+ \Gen$ is the right derived functor of $H^0(f_+) \Gen$. Furthermore, as $\Gen$ is exact, $\Gen Rf_+$ is the right derived functor of $\Gen f_+$. Thus, there is a natural equivalence $\Gen Rf_+ \cong f_+ \Gen$ of functors from the bounded derived category of $\gamma$-sheaves on $X$ to the bounded derived category of quasi-coherent unit left $\CO_{F,Y}$-modules. It induces an isomorphism of functors between $D_{\crys}^b(\QCrysG(X))$ and $D_{\lfgu}^b(\CO_{F,Y})$ because $D^b(\mu_{\operatorname{u}}(Y)) \overset{\sim}{\longrightarrow} D_{\lfgu}^b(\CO_{F,Y})$ (\cite[Corollary 17.2.5]{EmKis.Fcrys}). 
\end{proof}

\section{Adjunction for morphisms with proper support}

This section is largely independent from the rest of the paper. We establish an extension of the adjunction between $Rf_*$ and $f^!$ for proper morphisms from Grothendieck-Serre duality to the case where $f$ is only proper over the support of the complexes considered. 

We do not assume that the schemes are of characteristic $p$. However, all schemes we consider are assumed to be Noetherian. For a scheme $X$, we let $D_{\qc}^*(X)$ or $D_{\text{c}}^*(X)$ with $* \in \{+,-,\text{b} \}$ denote the derived category of $\CO_X$-modules with quasi-coherent or coherent cohomology. Here $*=+$ or $*=-$ or $*=b$ means that we require that the cohomology sheaves are bounded below or bounded above or bounded in both directions. 

Recall the classical adjunction:
\begin{theorem}[\protect{\cite[VII, 3.4(c)]{HartshorneRD}}] \label{adjunctionclassic}
Let $f\colon X \to Y$ be a proper morphism between Noetherian schemes with $Y$ admitting a dualizing complex. For $\CF^{\bullet}\in D_{\coh}^-(X)$ and $\CG^{\bullet} \in D_{\coh}^+(Y)$ the composition
\begin{align*} 
	\xymatrix{
		Rf_* \RSHom_{\CO_X}^{\bullet}(\CF^{\bullet},f^!\CG^{\bullet}) \ar[r] &  \RSHom_{\CO_Y}^{\bullet}(Rf_* \CF^{\bullet},Rf_*f^! \CG^{\bullet}) \ar[d]^{\tr_f} \\
		& \RSHom_{\CO_Y}^{\bullet}(Rf_* \CF^{\bullet}, \CG^{\bullet})}
\end{align*}
is an isomorphism. Here the first morphism is the canonical one and the second is the trace map. 
\end{theorem}
In \cite{LipmanGrothDual}, Lipman proves this theorem with weaker assumptions. For example, by \cite[Corollary 4.4.2]{LipmanGrothDual}, \autoref{adjunctionclassic} holds even if $\CF^{\bullet}$ and $\CG^{\bullet}$ only have quasi-coherent cohomology.

In this section we relax the properness assumption and show the following:
\begin{theorem} \label{qcohadjunction} 
Let $f\colon X \to Y$ be a separated and finite type morphism of schemes and let $i\colon Z \to Y$ and $i'\colon Z' \to X$ be closed immersions with a proper morphism $f'\colon Z' \to Z$ such that the diagram 
\[
	\xymatrix{
		Z' \ar[r]^-{i'} \ar[d]^-{f'} & X \ar[d]^-f \\
		Z \ar[r]^-i & Y
	}
\]
commutes. Then there is a natural transformation $\tr_f\colon Rf_*R\Gamma_{Z'}f^! \to \id$ such that, for all $\CF^{\bullet} \in D_{\qc}^-(\CO_X)_Z$ and $\CG^{\bullet} \in D_{\qc}^+(\CO_Y)_Z$ (see \autoref{supportdefine}), the composition 
\begin{align*}
	\xymatrix{
		Rf_* \RSHom_{\CO_X}^{\bullet}(\CF^{\bullet},R\Gamma_{Z'}f^!\CG^{\bullet}) \ar[r] &  \RSHom_{\CO_Y}^{\bullet}(Rf_* \CF^{\bullet},Rf_*R\Gamma_{Z'}f^! \CG^{\bullet}) \ar[d]^{\tr_f} \\
		& \RSHom_{\CO_Y}^{\bullet}(Rf_* \CF^{\bullet}, \CG^{\bullet})}
\end{align*}
is an isomorphism. In particular, taking global sections, the functor $Rf_*$ is left adjoint to the functor $R\Gamma_{Z'}f^!$. 
\end{theorem}
Note that the properness of $f'$ is equivalent to the properness of $i \circ f'$.

\subsection{Local cohomology}

Let $X$ be a topological space, $Z$ a closed subset and $\CF$ a sheaf of abelian groups on $X$. The subsheaf $\Gamma_Z(\CF)$ of $\CF$ is given by 
\[
	U \mapsto \{s \in \CF(U) \ | \ \text{The support of s is contained in } Z \cap U\},
\]
see the beginning of Section 1 of \cite{HartshorneLC}. The functor $\Gamma_Z$ is left exact, its right derived functors are called \emph{local cohomology sheaves}. Now we consider the more specific situation where $X$ is a Noetherian scheme. Let $i\colon Z \into X$ be a closed immersion and $j\colon U \into X$ the open immersion of the complement $U = X \backslash Z$:
\[
	Z \overset{i}{\hookrightarrow} X \overset{j}{\hookleftarrow} U.
\]   
Let $\CI$ be a sheaf of ideals defining $Z$. A priori, $\Gamma_Z'(\CF):=\underset{n \in \mathbb N}{\varinjlim} \SHom_{\CO_X}(\CO_X/\CI^n,\CF)$ is a subsheaf of $\Gamma_Z(\CF)$ for an $\CO_X$-module $\CF$. If $\CF$ is quasi-coherent, we have $\Gamma_Z'(\CF)=\Gamma_Z(\CF)$ (\cite[Theorem 2.8]{HartshorneLC}). In what follows, we consider local cohomology in the context of derived categories and we restrict to quasi-coherent cohomology. Therefore we make the following   
\begin{definition} \label{deflocalcohomology}
The local cohomology functor $R\Gamma_Z\colon D_{\qc}(X) \to D_{\qc}(X)$ is the derived functor of the left exact functor 
\[
	\Gamma_Z:= \underset{n \in \mathbb N}{\varinjlim} \SHom_{\CO_X}(\CO_X/\CI^n,\usc),
\]
where $\CI$ is any sheaf of ideals defining $Z$.
\end{definition}
A reference for this point of view is \cite{LocalHomology}. For example, \autoref{deflocalcohomology} is equation (0.1) of ibid.
\begin{definition} \label{supportdefine}
We say that a complex $\CF^{\bullet}$ of $\CO_X$-modules has \emph{support in} or \emph{on $Z$} or that $\CF^{\bullet}$ \emph{is supported in} or \emph{on $Z$} if $j^*\CF^{\bullet} = 0$ in $D(X)$. We write $D(X)_Z$, $D_{\qc}(X)_Z$ etc. for the subcategory of objects of $D(X)$, $D_{\qc}(X)$ etc. whose cohomology is supported on $Z$. 
\end{definition}   
The natural inclusion $\Gamma_Z \to \id$ induces a transformation $R\Gamma_Z \to \id$. As pointed out in the proof of \cite[Lemma (0.4.2)]{LocalHomology}, one has the following triangle:
\begin{proposition}
For every $\CF^{\bullet} \in D_{\qc}(X)$, there is a fundamental distinguished triangle
\[
	R\Gamma_Z \CF^{\bullet} \longrightarrow \CF^{\bullet} \longrightarrow Rj_*j^*\CF^{\bullet} \longrightarrow R\Gamma_Z \CF^{\bullet}[1],
\]
where the second map is the natural one from the adjunction of $Rj_*$ and $j^*$. This triangle restricts to the subcategories $D_{\qc}^+(X)$ and $D_{\qc}^b(X)$ because $j^*$ is exact and $Rj_*\colon D_{\qc}^+(U) \to D_{\qc}^+(X)$ has finite cohomological amplitude. 
\end{proposition}
In particular, $R\Gamma_Z$ only depends on the closed subset $i(Z)$ and not on the scheme structure of $Z$. The fundamental triangle allows another characterization of $D_{\qc}(X)_Z$: 
\begin{corollary}
The subcategory $D_{\qc}(X)_Z$ consists of all complexes $\CF^{\bullet} \in D_{\qc}(X)$ such that the natural map $R\Gamma_Z\CF^{\bullet} \to \CF^{\bullet}$ is an isomorphism.
\end{corollary}
\begin{lemma}
If $\CI$ is an injective quasi-coherent sheaf, then also the quasi-coherent sheaf $\Gamma_Z(\CI)$ is injective.
\end{lemma}
\begin{proof}
It suffices to check the injectivity of $\Gamma_Z(\CI)$ locally. Thus the assertion follows from (\cite[Proposition 2.1.4]{Brodmann.LocCoh}).
\end{proof}
As $R\Gamma_Z \circ R\Gamma_Z \cong R\Gamma_Z$, the image of $R\Gamma_Z$ is exactly the subcategory $D_{\qc}(X)_Z$. Furthermore, the functor $R\Gamma_Z$ is right adjoint to the inclusion $D_{\qc}(X)_Z \into D_{\qc}(X)$. This is a consequence of the following proposition, see the proof of \autoref{qcohadjunction}.
\begin{proposition} \label{Gammaadj}
Let $\CG^{\bullet}$ be a complex in $D_{\qc}(X)_Z$. Then there is a functorial isomorphism
\[
	\RSHom_{\CO_X}^{\bullet}(\CG^{\bullet},R\Gamma_Z\CF^{\bullet}) \cong \RSHom_{\CO_X}^{\bullet}(\CG^{\bullet},\CF^{\bullet})
\]
for every $\CF^{\bullet} \in D_{\qc}(X)$.
\end{proposition}    
\begin{proof}
This is \cite[Lemma (0.4.2)]{LocalHomology}. We even do not have to assume that the cohomology sheaves of $\CF^{\bullet}$ and $\CG^{\bullet}$ are quasi-coherent. 
\end{proof}
Next we verify the compatibility of $R\Gamma_Z$ with the derived functors $Rf_*$, $f^!$ and $\derotimes$.
\begin{lemma} \label{Gammacomm}
Let $f\colon X \to Y$ be a morphism of finite type and $i\colon Z \to Y$ a closed immersion. Let $i'\colon Z' \to X$ denote the projection $Z' := Z \times_Y X \to X$.  
\begin{enumerate}
\item There is a natural isomorphism of functors
\[
	Rf_*R\Gamma_{Z'} \cong R\Gamma_ZRf_*.
\]
\item If $f$ is flat, then there is a natural isomorphism of functors
\[
	f^*R\Gamma_Z \cong R\Gamma_{Z'}f^*.
\]
\end{enumerate}
\end{lemma}
\begin{proof}
First we show that $Rf_*R\Gamma_{Z'}$ is supported on $Z$. Let $u\colon U \into Y$ and $v\colon V \into X$ be the open immersions of the complements of $Z$ and $Z'$ in $Y$ and $X$. Let $f'$ denote the restriction of $f$ to $V$. We obtain a cartesian square
\[
	\xymatrix{
		V \ar[r]^-{v} \ar[d]^-{f'} & X \ar[d]^-f \\
		U \ar[r]^-u & Y.
	}
\]
Hence $u^*Rf_*R\Gamma_{Z'} \cong Rf'_*v^*R\Gamma_{Z'} = 0$. From \autoref{Gammaadj} we know that the canonical morphism $Rf_*R\Gamma_{Z'} \to Rf_*$ factors through $R\Gamma_ZRf_*$. Let $\alpha$ denote the corresponding morphism $Rf_*R\Gamma_{Z'} \to R\Gamma_ZRf_*$. 

Recall that the natural isomorphism $\bc\colon u^*Rf_* \overset{\sim}{\longrightarrow} Rf'_*v^*$ is the adjoint of the composition 
\[
	Rf_* \xrightarrow{Rf_*\operatorname{ad}_v} Rf_*Rv_*v^* \overset{\sim}{\longrightarrow} Ru_*Rf'_*v^*.
\]
Here $\operatorname{ad}_v$ is the unit of the adjunction between $Rv_*$ and $v^*$. The square
\[
	\xymatrix@C40pt{
		Ru_*u^*Rf_* \ar[r]^-{Ru_*\bc} & Ru_*Rf'_*v^* \\
		Rf_* \ar[u]^-{\ad_uRf_*} \ar[r]^-{Rf_*\ad_v} & Rf_*Rv_*v^* \ar[u]^-{\sim}
	}
\]
commutes because $Ru_*\bc \circ \ad_uRf_*$ is the adjoint of $\bc$ and hence equals the original morphism $Rf_* \to Ru_*Rf'_*v^*$. Let $\beta$ be the composition of the natural isomorphism $Rf_*Rv_*v^* \simeq Ru_*Rf'_*v^*$ with the inverse of $Ru_*\bc$. We have just seen that the right square of the	diagram
\[
	\xymatrix{
		Rf_*R\Gamma_{Z'} \ar[d]^{\alpha} \ar[r] & Rf_* \ar@{=}[d] \ar[r] & Rf_*Rv_*v^* \ar[d]_{\sim}^{\beta} \\ 
		R\Gamma_ZRf_* \ar[r] & Rf_* \ar[r] & Ru_*u^*Rf_* 
	}
\]
commutes. The left square commutes by construction. As the lines are distinguished triangles, $\alpha$ is an isomorphism. This shows (a).

For (b) we proceed similarly. The cartesian square above gives rise to the isomorphism 
\[
	\xymatrix{
		f^*R\Gamma_Z \ar[d]^{\sim} \ar[r] & f^* \ar@{=}[d] \ar[r] & f^*Ru_*u^* \ar[d]_{\sim} \\ 
		R\Gamma_{Z'}f^* \ar[r] & f^* \ar[r] & Rv_*v^*f^* 
	}
\]
of distinguished triangles. 
\end{proof}
\begin{remark} \label{underGammaPullback}
With the notation of part (b) of the preceding lemma, for every quasi-coherent sheaf $\CF$, we even have a natural isomorphism
\[
	f^*\Gamma_Z\CF \cong \Gamma_{Z'}f^*\CF,
\] 
see \cite[Lemma 4.3.1]{Brodmann.LocCoh}. 
\end{remark}
\begin{lemma} \label{RGammaTensor}    
Let $\CF^{\bullet}$ and $\CG^{\bullet}$ be complexes in $D_{\qc}^b(X)$. There are natural isomorphisms
\[
	(R\Gamma_Z\CF^{\bullet}) \derotimes_{\CO_X} \CG^{\bullet} \cong \CF^{\bullet} \derotimes_{\CO_X} (R\Gamma_Z\CG^{\bullet}) \cong R\Gamma_Z(\CF^{\bullet} \derotimes_{\CO_X} \CG^{\bullet})
\]
in $D_{\qc}(X)$.
\end{lemma}
\begin{proof}
The natural map $\CF^{\bullet} \derotimes R\Gamma_Z\CG^{\bullet} \to \CF^{\bullet} \derotimes \CG^{\bullet}$ factors through $R\Gamma_Z(\CF^{\bullet} \derotimes R\Gamma_Z\CG^{\bullet})$ because 
\[
	j^*(\CF^{\bullet} \derotimes R\Gamma_Z\CG^{\bullet}) \cong j^*\CF^{\bullet} \derotimes j^*R\Gamma_Z\CG^{\bullet} \cong 0.
\]
Let $\rho$ denote the composition of the natural isomorphism 
\[
	Rj_*j^*(\CF^{\bullet} \derotimes_{\CO_X} \CG^{\bullet}) \cong Rj_*(j^*\CF^{\bullet} \derotimes_{\CO_X} j^*\CG^{\bullet})
\]
and the isomorphism from the projection formula 
\[
	Rj_*(j^*\CF^{\bullet} \derotimes_{\CO_X} j^*\CG^{\bullet}) \cong \CF^{\bullet} \derotimes_{\CO_X} Rj_*j^*\CG^{\bullet}.
\]
We obtain a morphism of distinguished triangles 
\[
	\xymatrix{
		\CF^{\bullet} \derotimes R\Gamma_Z\CG^{\bullet} \ar[d] \ar[r] & \CF^{\bullet} \derotimes \CG^{\bullet} \ar@{=}[d] \ar[r] & \CF^{\bullet} \derotimes Rj_*j^*\CG^{\bullet} \ar[d]_-{\sim}^-{\rho} \\
		R\Gamma_Z(\CF^{\bullet} \derotimes \CG^{\bullet}) \ar[r] & \CF^{\bullet} \derotimes \CG^{\bullet} \ar[r] & Rj_*j^*(\CF^{\bullet} \derotimes \CG^{\bullet}).
	}
\]
Therefore, the left vertical arrow is an isomorphism. Analogously, one shows that 
\[
	R\Gamma_Z\CF^{\bullet} \derotimes \CG^{\bullet} \cong R\Gamma_Z(\CF^{\bullet} \derotimes \CG^{\bullet}).
\]
\end{proof}
Finally, for an open immersion $j\colon X \to \LX$, we study the connection between $R\Gamma_Z$ and $R\Gamma_{\overline{Z}}$, where $\overline{Z}$ is the closure of $Z$ in $\LX$.
\begin{definition}
Let $Z$ and $Z'$ be be closed subsets of a scheme $X$. We let $D_{\qc}(X)_Z^{Z'}$ denote the full subcategory of $D_{\qc}(X)_Z$ of complexes $\CF^{\bullet}$ with $R\Gamma_{Z'} \cong 0$
\end{definition}
\begin{proposition} \label{opensupport}
Let $j\colon X \to \LX$ be an open immersion of schemes, $Z \subseteq X$ a closed subset and $\overline{Z}$ the closure of $Z$ in $\LX$. The functors $Rj_*$ and $j^*$ are inverse equivalences 
\[
	\xymatrix{
     D_{\qc}(X)_Z \ar@<.5ex>[r]^-{Rj_*} & D_{\qc}(\LX)_{\overline{Z}}^{\overline{Z} \backslash X}. \ar@<.5ex>[l]^-{j^*}
   }
\] 
\end{proposition} 		
\begin{proof}
Let $u\colon U \to X$ and $u'\colon U' \to \LX$ denote the open immersions of the complements $U$ of $Z$ in $X$ and $U'$ of $\overline{Z}$ in $\LX$. Let $j'$ be the restriction of $j$ to $U$. We obtain a cartesian square
\[
	\xymatrix{
		U \ar[r]^-u \ar[d]^-{j'} & X \ar[d]^-j \\
		U' \ar[r]^-{u'} & \LX.
	}
\]
The natural isomorphism $u'^*Rj_* \overset{\bc}{\longrightarrow} Rj'_*u^*$ shows that the essential image of $D_{\qc}(X)_Z$ under $Rj_*$ is a subcategory of $D_{\qc}(\LX)_{\overline{Z}}$. The inclusion $j$ factors through the open immersions $\sigma\colon X \to U' \cup X$ and $\tau\colon U' \cup X \to \LX$. In particular, we have a natural isomorphism $Rj_* \cong R\tau_* R\sigma_*$. Since the composition
\[
	R\tau_* R\sigma_* \xrightarrow{\id \to R\tau_* \tau^*} R\tau_* \tau^* R\tau_* R\sigma_* \xrightarrow{\tau^* R\tau_* \to \id} R\tau_* R\sigma_*
\]
is the identity and the second morphism is an isomorphism, the first map is an isomorphism too. As $\tau$ is the open immersion of the complement of $\overline{Z} \backslash X$ into $\LX$, it follows from the distinguished triangle 
\[
	R\Gamma_{\overline{Z} \backslash X} \longrightarrow \id \longrightarrow R\tau_* \tau^* \longrightarrow R\Gamma_{\overline{Z} \backslash X}[1]
\]
that $R\Gamma_{\overline{Z} \backslash X} Rj_* j^* = 0$. 

The adjunction morphism $j^*Rj_* \to \id$ is always an isomorphism. It remains to show that the natural map $\id \to Rj_*j^*$ is an isomorphism. For every $\CF^{\bullet}$ in $D_{\qc}(\LX)_{\overline{Z}}^{\overline{Z} \backslash X}$, we have $\CF^{\bullet} \cong R\Gamma_{\overline{Z}} \CF^{\bullet}$ and $R\Gamma_{\overline{Z} \backslash X} \CF^{\bullet} \cong 0$. It follows that 
\begin{align*} 
	R\Gamma_{\LX \backslash X}\CF^{\bullet} &\cong R\Gamma_{\LX \backslash X}R\Gamma_{\overline{Z}}\CF^{\bullet} \\
	&\cong R\Gamma_{\overline{Z} \backslash X}\CF^{\bullet} \\
	&\cong 0.
\end{align*}
Thus the second morphism in the fundamental triangle
\[
	R\Gamma_{\LX \backslash X}\CF^{\bullet} \longrightarrow \CF^{\bullet} \longrightarrow Rj_* j^*\CF^{\bullet} \longrightarrow R\Gamma_{\LX \backslash X}\CF^{\bullet}[1]
\]
is an isomorphism.
\end{proof}
\begin{remark}
In the standard reference \cite[Corollary II.5.11]{HartshorneRD}, Hartshorne proves that for a morphism $f\colon X \to Y$ of schemes, the functor $Lf^*$ from $D_c^-(Y)$ to $D_c^-(X)$ is left adjoint to the functor $Rf_*$ from $D^+(X)$ to $D^+(Y)$. One the one hand, we can relax the coherence assumption because in the case of an open immersion, which is a flat morphism, $f^*$ is exact. On the other hand, Proposition (3.2.1) of the more recent reference \cite{LipmanGrothDual} shows this adjunction generally for ringed spaces and without any boundedness or (quasi-)coherence assumptions on the complexes.  
\end{remark}
\begin{corollary} \label{Gammajung}
If $Z$ is a closed subset of a scheme $X$ and $j\colon X \to \LX$ is an open immersion such that the image of $Z$ in $\LX$ is closed, then there is a natural isomorphism of functors
\[
	\epsilon\colon R\Gamma_Z \overset{\sim}{\longrightarrow} Rj_*R\Gamma_Zj^*.
\]
\end{corollary}
\begin{proof}
We define $\epsilon$ as the composition of the natural map $R\Gamma_Z \to Rj_*j^*R\Gamma_Z$, which is an isomorphism by \autoref{opensupport}, and the natural isomorphism $Rj_*j^*R\Gamma_Z \overset{\sim}{\longrightarrow} Rj_*R\Gamma_Zj^*$.
\end{proof} 
For example, the condition that $Z$ is also closed in $\LX$ is satisfied if $j\colon X \to \LX$ is an open immersion of $Y$-schemes and $i\colon Z \to X$ is a closed immersion of $Y$-schemes over some base scheme $Y$ such that the structural morphisms $Z \to Y$ and $\LX \to Y$ are proper, see \autoref{ExerciseHartshorne}. When constructing the generalized trace map, we will be exactly in this situation. 
\begin{lemma} \label{opennattrans}
Let $j\colon X \to \LX$ be an open immersion. Let $Z \subset X$ be a closed subset such that $j(Z)$ is closed in $\overline{Z}$. Then for $\CF^{\bullet} \in D_{\qc}^-(X)_Z$ and $\CG^{\bullet} \in D_{\qc}^+(X)$, the natural transformation
\[
	\tau\colon Rj_* \RSHom_{\CO_X}^{\bullet}(\CF^{\bullet},\CG^{\bullet}) \to \RSHom_{\CO_{\LX}}^{\bullet}(Rj_*\CF^{\bullet},Rj_*\CG^{\bullet})
\]
is a functorial isomorphism.  
\end{lemma} 
\begin{proof}
Consider the following diagram
\[
	\xymatrix{
		Rj_*\RSHom_{\CO_X}^{\bullet}(\CF^{\bullet},\CG^{\bullet}) \ar[r]^-{\sim} \ar[d]^{\tau} & Rj_*\RSHom_{\CO_X}^{\bullet}(j^*Rj_*\CF^{\bullet},\CG^{\bullet}) \ar[d]^{\tau} \\
		\RSHom_{\CO_{\LX}}^{\bullet}(Rj_*\CF^{\bullet},Rj_*\CG^{\bullet}) \ar[r]^-{\sim} \ar@{=}[dr] & \RSHom_{\CO_{\LX}}^{\bullet}(Rj_*j^*Rj_*\CF^{\bullet},Rj_*\CG^{\bullet}) \ar[d] \\
		& \RSHom_{\CO_{\LX}}^{\bullet}(Rj_*\CF^{\bullet},Rj_*\CG^{\bullet}),
	}
\]
where the horizontal arrows are induced by the counit $j^*Rj_* \to \id$ of adjunction -- these maps are isomorphisms by \autoref{Gammajung} since $\CF^{\bullet}$ is supported in $Z$ -- and the arrow to the bottom right corner stems from the unit $\id \to Rj_*j^*$ of adjunction. The upper square commutes because of the functoriality of $\tau$. The triangle on the bottom commutes because the composition of the unit and counit of an adjunction in the manner of the diagram is canonically isomorphic to the identity. Hence the whole diagram is commutative. Finally, the composition of the two vertical arrows on the right is an isomorphism (\cite[Proposition (3.2.3)]{LipmanGrothDual}. It follows that the vertical arrow on the left is an isomorphism.    
\end{proof}

\subsection{Adjunction for quasi-coherent sheaves} 

As pointed out in the introduction of this section, the adjunction between $Rf_*$ and $f^!$ for a proper morphism $f\colon X \to Y$ is based on the \emph{trace map}, which is a natural transformation of functors
\[
	\tr_f\colon Rf_*f^! \to \id.
\]
The classical way to construct the trace is to define it for residual complexes and then ``pull ourselves up by our bootstraps'': As $Y$ is regular, the structure sheaf $\CO_Y$ is a dualizing sheaf, and hence in particular a pointwise dualizing complex. Its Cousin complex $K^{\bullet} := E^{\bullet}(\CO_Y)$, see \cite[IV.2]{HartshorneRD}, is an injective resolution of $\CO_X$ and an example for a \emph{residual complex}. We recall the basic facts from chapter 3.2 of \cite{conrad_grothendieck_2000}.
\begin{definition}
For every $x \in X$, let $i_x$ denote the canonical morphism $\Spec(\CO_{X,x}) \to X$. We set $\CJ(x) = i_{x*}J(x)$, where $J(x)$ is an injective hull of the residue field $k(x)$ considered as a module over $\CO_{X,x}$. A \emph{residual complex} $K^{\bullet}$ on $X$ is a complex in $D_c^+(X)$ of quasi-coherent injectives such that there is an isomorphism of $\CO_X$-modules
\[
	\bigoplus_{n \in \ZZ}K^n \cong \bigoplus_{x \in X}\CJ(x).
\]
\end{definition} 
One can show that, given a residual complex $K^{\bullet}$, there is a unique function $d_{K^{\bullet}}\colon X \to \mathbb N$ such that 
\[
	K^n \cong \bigoplus_{d_{K^{\bullet}}(x)=n}\CJ(x).
\]	
Restricting to residual complexes one can construct a functor $g^{\Delta}\colon \Res(Y) \to \Res(X)$ for every morphism $g\colon X \to Y$ of finite type by gluing the functors $g^{\flat}$ for finite $g$ and $g^{\sharp}$ for separated and smooth $g$. This gives rise to the \emph{twisted} or \emph{exceptional inverse image functor}:
\begin{definition}
Let $g\colon X \to Y$ be a morphism of finite type. We define the functor $g^!\colon D_{\coh}^+(Y) \to D_{\coh}^+(X)$ by
\[
	g^! = D_{g^{\Delta}K^{\bullet}} \circ Lg^* \circ D_{K^{\bullet}},
\]
where $D$ is the duality.
\end{definition}
For proper $f$, we can define a map of complexes $\tr_f(K^{\bullet})\colon f_*f^{\Delta}K^{\bullet} \to K^{\bullet}$ (\cite[VII, Theorem 2.1]{HartshorneRD}), where $f^{\Delta}$ is the functor $f^!$ for residual complexes, see \cite[VI.3.]{HartshorneRD}. With this map in hand one defines the natural transformation $\tr_f\colon Rf_*f^! \to \id$ in the category $D_{\coh}^+(Y)$ as the unique map making the diagram
\begin{align*} 
	\xymatrix{
		Rf_*f^! \ar@{=}[r] \ar@{.>}[dd]^{\tr_f}& Rf_*\RSHom_{\CO_X}^{\bullet}(Lf^* \circ D_{K^{\bullet}}(\usc),f^{\Delta}K^{\bullet}) \ar[d]^{\sim}\\
		& \RSHom_{\CO_Y}^{\bullet}(D_{K^{\bullet}}(\usc),f_*f^{\Delta}K^{\bullet}) \ar[d]^{\tr_f(K^{\bullet})} \\
		\id \ar[r]^-{\sim} & \RSHom_{\CO_Y}^{\bullet}(D_{K^{\bullet}}(\usc),K^{\bullet})
	}
\end{align*}
commutative. Here the first vertical isomorphism on the right is the natural isomorphism from the adjunction of $Rf_*$ and $Lf^*$. Note that $f^{\Delta}(K^{\bullet})$ is injective, hence $f_*f^{\Delta}(K^{\bullet})$ computes $Rf_*f^{\Delta}(K^{\bullet})$.

Instead of constructing the twisted inverse image functor $f^!$ by pasting it from special situations such as smooth and proper maps, Lipman uses a more abstract method, the Special Adjoint Functor Theorem, to obtain a right adjoint of $Rf_*$ under weak assumptions on the morphism $f$. Then he extends this result to a ``sheafified duality'', i.e.\ for $\CF^{\bullet} \in D_{\qc}(X)$, $\CG^{\bullet} \in D_{\qc}^+(Y)$ and quasi-proper $f$, a natural isomorphism
\[
	Rf_*\RSHom_{\CO_X}^{\bullet}(\CF^{\bullet},f^!\CG^{\bullet}) \to \RSHom_{\CO_Y}^{\bullet}(Rf_*\CF^{\bullet},\CG^{\bullet}). 
\]
The compatibility of the approaches of \cite{HartshorneRD} and \cite{LipmanGrothDual} is involved, as pointed out in the introduction of \cite{LipmanGrothDual}.
  
Let us recall some results of the trace for proper morphisms.
\begin{definition}[\protect{{\cite[VI. 5.]{HartshorneRD}}}]
A morphism $f\colon X \to Y$ of schemes is called \emph{residually stable} if it is flat, integral and the fibers of $f$ are Gorenstein.
\end{definition}
\begin{lemma}[\protect{{\cite[Corollary 4.4.3]{LipmanGrothDual}}}] \label{flatcommute}
Let $f\colon X \to Y$ be proper and let $g\colon Y' \to Y$ be flat. Let $f'$ and $g'$ be the projections of $X \times_Y Y'$ such that the square
\[
	\xymatrix{
		Y' \times_Y X \ar[d]^-{f'} \ar[r]^-{g'} & X \ar[d]^-f \\
		Y' \ar[r]^-g & Y
	}
\]
is cartesian. The morphism $\beta\colon g'^*f^! \overset{\sim}{\longrightarrow} f'^!g^*$, defined as the adjoint of the composition
\[
	Rf'_*g'^*f^! \xrightarrow{\bc^{-1} f^!} g^*Rf_*f^! \xrightarrow{g^*\tr_f} g^*,
\]
is an isomorphism.
\end{lemma}
Let us recall two compatibilities of the trace, which usually are known as ``TRA 1'' and ``TRA 4''.
\begin{lemma} \label{proptrace}
Let $f\colon X \to Y$ be a proper morphism of schemes.
\begin{enumerate}
	\item (TRA 1) If $g\colon Y \to Z$ is another proper morphism, then there is a commutative diagram
	\[
		\xymatrix{
			R(gf)_*(gf)^! \ar[r]^-{\tr_{g \circ f}} \ar[d]^{\sim} & \id \\
			Rg_*Rf_*f^!g^! \ar[r]^-{\tr_f} & Rg_*g^! \ar[u]_{\tr_g}
		}
	\]
	where the first vertical arrow is the natural isomorphism. 
	\item (TRA 4) For a flat morphism $g\colon Y' \to Y$, there is a commutative diagram
	\[
		\xymatrix@C45pt{
			g^*Rf_*f^! \ar[r]^-{g^* \tr_f} \ar[d]^{\bc}_{\sim} & g^* \\
			Rf'_*g'^*f^! \ar[r]^-{Rf'_*\beta}_-{\sim} & Rf'_*f'^!g^*, \ar[u] _{\tr_{f'}g^*}
		}
	\]
where $g'$ and $f'$ are the two projections of $X \times_Y Y'$.
\end{enumerate}
\end{lemma}
\begin{proof}
(a) is \cite[Corollary VII.3.4]{HartshorneRD}. The diagram in (b) commutes by construction of $\beta$: The composition $\tr_{f'}g^* \circ Rf'_*\beta$ is the adjoint of the adjoint of the composition $u^* \tr_f \circ \bc^{-1}$, see \autoref{flatcommute}. 
\end{proof}
\begin{remark}
Part (b) of the preceding lemma holds under the milder assumption that $f$ is of finite Tor-dimension, see \cite[Corollary 4.4.3]{LipmanGrothDual}. In this more general case one considers the left derived functors $Lf^*$ and $Lf'^*$. However, we will need the compatibility of the trace with pullback only for flat morphisms.
\end{remark} 
From now on we do not assume that $f$ is proper. We are interested in the case where $f\colon X \to Y$ is a separated morphism of finite type of smooth Noetherian schemes and $i\colon Z \to Y$ and $i'\colon Z' \to X$ are closed immersions with a proper morphism $f'\colon Z' \to Z$ such that $f \circ i' = i \circ f'$. The compactification theorem of Nagata (\cite{Nagata}, see also \cite{Luet} for a more recent proof) states that there exists a factorization of $f$ into an open immersion $j\colon X \to \LX$ and a proper morphism $\overline{f} \colon \LX \to Y$.
\begin{lemma} \label{ExerciseHartshorne}
Let $f\colon X \to Y$ be a morphism of schemes that factors through an open immersion $j\colon X \to \LX$ followed by a proper morphism $\overline{f}\colon \LX \to Y$. Then for every closed immersion $i\colon Z \to X$ such that $f \circ i$ is proper, the composition $j \circ i$ is also a closed immersion. 
\end{lemma} 
\begin{proof}
We have to show that $j(i(Z))$ is closed in $\LX$ (which is a special case of the first part of exercise II.4.4 of \cite{Hartshorne}). By assumption the composition $\overline{f} \circ j \circ i = f \circ i$ is proper and $\overline{f}$ is proper, in particular $\overline{f}$ is separated. Hence by \cite[Corollary II.4.8]{Hartshorne}, $j \circ i$ is proper, which implies that the image $j(i(Z))$ is closed.
\end{proof}
The following generalization of the trace map stems from \cite{Ruelling}. Our construction is similar to the morphism $\operatorname{Tr}_f$ from Corollary 1.7.6 of ibid. 
\begin{definition} \label{tracedef}
For a morphism $f\colon X \to Y$ of finite type and closed immersions $i\colon Z \to Y$ and $i'\colon Z' \to X$ with a proper morphism $f'\colon Z' \to Z$ such that $f \circ i' = i\circ f'$, choose a compactification, i.e.\ an open immersion $j\colon X \to \LX$ and a proper morphism $\overline{f}\colon \LX \to X$ with $\overline{f} \circ j = f$. We obtain the following commutative diagram:
\[
	\xymatrix{
		Z' \ar@^{(->}[r]^-{i'} \ar[d]^-{f'} & X \ar[r]^-j \ar[d]^-f & \LX \ar[dl]^{\overline{f}} \\
		Z \ar@^{(->}[r]^-i & Y &
	}
\]
We define the trace of $f$ as the morphism of functors 
\[
	\tr_{f,Z} = \tr_f\colon Rf_*R\Gamma_{Z'}f^! \to \id 
\]
on $D_{\qc}^+(\CO_Y)$ given by the composition 
\begin{align*} \label{trace}
	Rf_*R\Gamma_{Z'}f^! \overset{\sim}{\longrightarrow} R\overline{f}_*Rj_*R\Gamma_{Z'}j^*\overline{f}^! \xrightarrow{R\overline{f}_*\epsilon^{-1}\overline{f}^!} R\overline{f}_*R\Gamma_{Z'}\overline{f}^! \xrightarrow{R\Gamma_{Z'} \to \id}R\overline{f}_*\overline{f}^! \xrightarrow{tr_{\overline{f}}} \id,
\end{align*}
where $\epsilon$ is the isomorphism of \autoref{Gammajung} and the last morphism is the classical Grothendieck-Serre trace for the proper map $\overline{f}$. 
\end{definition} 
Because $Rf_*R\Gamma_{Z'}f^! \cong R\Gamma_ZRf_*f^!$ (\autoref{Gammacomm}), the complex $Rf_*R\Gamma_{Z'}f^!$ is supported on $Z$. By \autoref{Gammaadj}, $\tr_f$ factors through $R\Gamma_Z$, i.e.\ there is a commutative diagram
\[
	\xymatrix{
		Rf_*R\Gamma_{Z'}f^! \ar[rr]^-{\tr_f} \ar[dr]_-{\widetilde{\tr}_f} & & \id, \\
		& R\Gamma_Z \ar[ur] &
	}
\]
where $\widetilde{\tr}_f$ is induced by $\tr_f$ and the map $R\Gamma_Z \to \id$ is the natural one. We will not distinguish between $\widetilde{\tr}_f$ and $\tr_f$. For a residual complex $E^{\bullet}$, the trace is a morphism of complexes because $f^{\Delta}E^{\bullet}$ and $\overline{f}^{\Delta}E^{\bullet}$ are residual complexes and $\Gamma_Z$ preserves injectives.  

Of course we have to show that $\tr_f$ is well-defined, i.e.\ it does not depend on the choice of a compactification. The next lemma prepares the proof of this independence. 
\begin{lemma} \label{opentrace}
Let $f\colon X \to Y$ be an open immersion. Let $Z \subseteq X$ be a closed subset such that $f(Z)$ is closed in $Y$. Then for every compactification $X \xrightarrow{j} \LX \xrightarrow{\overline{f}} Y$, the map $\tr_f$ equals the inverse $Rf_*R\Gamma_Zf^* \cong R\Gamma_Z$ of the isomorphism of \autoref{Gammajung} followed by the natural morphism $R\Gamma_Z \to \id$.  
\end{lemma}
\begin{proof}
Let $\alpha\colon R\Gamma_Z \to \id$ denote the canonical morphism of functors. The claim of the lemma is the commutativity of the diagram 
\[
	\xymatrix{
		Rf_* R\Gamma_Z f^* \ar[r]_-{\sim} \ar[d]_{\sim}^{\epsilon^{-1}} & R\overline{f}_* Rj_* R\Gamma_Z j^* \overline{f}^! \ar[d]_{\sim}^{R\overline{f}_* \epsilon^{-1} \overline{f}^!} \\
		R\Gamma_Z \ar[d]^{\alpha} & R\overline{f}_* R\Gamma_Z \overline{f}^! \ar[d]^{R\overline{f}_* \alpha \overline{f}^*} \\
		\id & R\overline{f}_* \overline{f}^!, \ar[l]_-{\tr_{\overline{f}}}
	}
\]
where $\epsilon$ is the isomorphism of \autoref{Gammajung}. Let $Z'$ be the closed subset $\overline{f}^{-1}(Z)$, which contains $Z$. Let $\phi$ denote the composition 
\[
	R\overline{f}_* R\Gamma_Z \overline{f}^! \to R\overline{f}_*R\Gamma_{Z'}\overline{f}^!\overset{\sim}{\longrightarrow} R\Gamma_Z R\overline{f}_*\overline{f}^! \xrightarrow{\tr_{\overline{f}}} R\Gamma_Z
\]
of canonical transformations obtained from the natural transformation $R\Gamma_Z \to R\Gamma_{Z'}$, the isomorphism of \autoref{Gammacomm} and the trace. It fits into the commutative diagram
\[
	\xymatrix{
		R\Gamma_Z \ar[d]^-{\alpha} & R\overline{f}_* R\Gamma_Z \overline{f}^! \ar[d]^-{R\overline{f}_* \alpha \overline{f}^!} \ar[l]_-{\phi} \\
		\id & R\overline{f}_* \overline{f}^!. \ar[l]_-{\tr_{\overline{f}}}
	}
\]
Hence it suffices to show that
\[
	\xymatrix{
		Rf_* R\Gamma_Z f^* \ar[r]^-{\sim} \ar[d]_{\sim}^{\epsilon^{-1}} & R\overline{f}_* Rj_* R\Gamma_Z j^* \overline{f}^! \ar[d]_{\sim}^{R\overline{f}_* \epsilon^{-1} \overline{f}^!} \\
		R\Gamma_Z & R\overline{f}_* R\Gamma_Z \overline{f}^! \ar[l]_-{\phi} 
	}
\]
commutes. This diagram can be extracted from the following bigger diagram:
\[
	\xymatrix{
		Rf_* R\Gamma_Z f^* \ar[r]^-{\sim} \ar[d]_{\sim} & R\overline{f}_*Rj_* R\Gamma_Z j^*\overline{f}^! \ar[d]_{\sim} \ar[dr]^{\sim} & \\
		R\Gamma_Z Rf_* f^* \ar[r]^-{\sim} & R\Gamma_Z R\overline{f}_* Rj_* j^* \overline{f}^! & R\overline{f}_* R\Gamma_{Z'} Rj_*j^*\overline{f}^! \ar[l]_-{\sim} \\
		R\Gamma_Z \ar[u]_{\ad_f}^{\sim} & R\Gamma_Z R\overline{f}_* \overline{f}^! \ar[u]_{\ad_j}^{\sim} \ar[l]_-{\tr} & R\overline{f}_* R\Gamma_{Z'} \overline{f}^!. \ar[u]_{\ad_j}^{\sim} \ar[l]_-{\sim} 
	}
\]
Here $\ad_f$ and $\ad_j$ denote the units of adjunction as in the proof of \autoref{Gammacomm}. The only part whose commutativity is not obvious is the bottom left square. For this it is enough to show that the diagram 
\begin{align} \label{opendiagram}
	\xymatrix@C50pt{
		& R\overline{f}_* \overline{f}^! \ar[r]^-{\tr} \ar[d]^{\ad_f} \ar[ddl]_{\ad_j} & \id \ar[d]^{\ad_f} \\
		& Rf_* f^* R\overline{f}_* \overline{f}^! \ar[r]^-{Rf_*f^*\tr_{\overline{f}}} \ar[d]^{\bc} & Rf_* f^* \ar@{=}[dd] \\
		R\overline{f}_* Rj_* j^* \overline{f}^! \ar[d]^-{\sim} & Rf_* Rf'_* j^* \overline{f}^! \ar[l]_-{\sim} \ar[d]^{\sim} & \\
		R\overline{f}_* Rj_* f'^* f^* & Rf_* Rf'_* f'^* f^* \ar[l]_-{\sim} \ar[r]^-{Rf_*\tr_{f'}f^*} & Rf_*f^*
	}
\end{align}
commutes. Here the morphism $\bc$ is the base change morphism with respect to the cartesian square
\[
	\xymatrix{
		X' \ar[r]^j \ar[d]^-{f'} & \overline{X} \ar[d]^-{\overline{f}} \\
		X \ar[r]^f & Y.
	}
\]
The commutativity of the upper left triangle of the diagram \autoref{opendiagram} was part of the proof of \autoref{Gammacomm}. The upper square of this diagram commutes by naturality of $\ad_f$ and the commutativity of the square below follows from \autoref{flatcommute}. Finally, the bottom left square commutes by naturality of the transformation $Rf_* Rf'_* \to R\overline{f}_* Rj_*$.    
\end{proof} 
\begin{lemma} 
The map $\tr_f$ is well-defined, i.e.\ it is independent of the choice of the compactification $j\colon X \into \LX$. 
\end{lemma}
\begin{proof}
Let $j_1\colon X \to \LX_1$ and $j_2\colon X \to \LX_2$ be two open immersions with proper morphisms $f_1\colon \LX_1 \to Y$ and $f_2\colon \LX_2 \to Y$ such that $f = f_1 \circ j_1 = f_2 \circ j_2$. By considering $\LX_1 \times_Y \LX_2$ we can reduce to the case that there is a proper morphism $g\colon \LX_1 \to \LX_2$ such that $g \circ j_1 = j_2$, i.e.\ the diagram 
\[
	\xymatrix{
		& X \ar[dl]_{j_1} \ar[dr]^{j_2} \ar[dd]_/-15pt/f & \\
		\LX_1\ \ar[dr]_{f_1} \ar[rr]_/15pt/g & & \LX_2 \ar[dl]^{f_2} \\
		& Y & 
	}
\]
commutes. That $\tr_f$ is well-defined means exactly that the following diagram of functors is commutative:
\[
	\xymatrix{
		& Rf_*R\Gamma_Zf^! \ar[dl]_-{\sim} \ar[d]_-{\sim} \ar[dr]^-{\sim} & \\
		R{f_1}_*R{j_1}_*R\Gamma_Zj_1^!f_1^!  \ar[d] & R{f_2}_*Rg_*R{j_1}_*R\Gamma_Zj_1^!g^!f_2^! \ar[l]_-{\sim} \ar[r]^-{\sim} \ar[d] & R{f_2}_*{j_2}_* R\Gamma_Z j_2^!f_2^! \ar[d] \\
		R{f_1}_*R\Gamma_Zf_1^! \ar[d]	& R{f_2}_*Rg_*R\Gamma_Zg^!f_2^! \ar[l]_-{\sim} \ar[d] & R{f_2}_*R\Gamma_Zf_2^! \ar[d] \\
		R{f_1}_*f_1^! \ar[dr]_-{\tr_{f_1}} & R{f_2}_*Rg_*g^!f_2^! \ar[l]_-{\sim} \ar[r]^-{{Rf_2}_* \tr_g f_2^!} \ar[d]^{\tr_{f_2 \circ g}} & R{f_2}_*Rf_2^! \ar[dl]^-{\tr_{f_2}} \\
		& \id. & \\
	}
\]	
Here the six vertical arrows in the middle are the natural maps occurring in the definition of $\tr_f$. The only part of which the commutativity is not obvious is the bigger rectangle on the right hand side, which follows from \autoref{opentrace} after canceling $R{f_2}_*$ and $f_2^!$ from the edges of the terms.    
\end{proof}
Note that the independence of $tr_f$ of the chosen compactification implies that $\tr_f$ equals the classical trace map whenever $f$ is proper.
\begin{proposition} \label{traceres}
The map $\tr_f$ is compatible with residually stable base change: For a residually stable morphism $g\colon S \to Y$, let $f'$ and $g'$ be the projections of $S \times_Y X$. Furthermore, let $Z_S$ and $Z_S'$ be the preimages of $Z$ and $Z'$ in $S$ and $S \times_Y X$. Then the diagram
\[	
	\xymatrix@C55pt{
		g^*Rf_*R\Gamma_{Z'}f^! \ar[d]_-{\bc}^-{\sim} \ar[r]^-{g^* \tr_f} & g^* \\
		Rf'_*g'^*R\Gamma_{Z'}f^! \ar[d]^-{\sim} & \\
		Rf'_*R\Gamma_{Z_S'}g'^*f^! \ar[r]^-{Rf'_*R\Gamma_{Z_S'}\beta} & Rf'_*R\Gamma_{Z_S'}f'^!g^* \ar[uu]_-{\tr_{f'}g^*} 
	}
\]
commutes. Here $\beta$ is the isomorphism of \autoref{flatcommute}.  
\end{proposition}
\begin{proof}
First we treat the case of an open immersion $u\colon U \to Y$. Let $h\colon S \to Y$ be a residually stable morphism and let $u'$ and $h'$ denote the projections of $S \times_Y U$. Again, we let $\ad_u$, $\ad_{u'}$, $\ad_{h'}$ and $\ad_{h \circ u'}$ denote the units of adjunction. The natural isomorphisms $Ru_*Rh'_* \cong Rh_*Ru'_*$ and $h'^*u^* \cong u'^*h^*$ are compatible with the adjunction of $(u \circ h')^*$ and $R(u \circ h')_*$, i.e.\ the diagram
\[
	\xymatrix{
		\id \ar[r]^-{\ad_{h \circ u'}} \ar[d]_-{\ad_u} & Rh_*Ru'_*u'^*h^* \ar[r]^-{\sim} & Rh_*Ru'_*h'^*u^* \ar@{=}[d] \\
		Ru_*u^* \ar[r]^-{\ad_{h'}} & Ru_*Rh'_*h'^*u^* \ar[r]^-{\sim} & Rh_*Ru'_*h'^*u^* 
	}
\]
of natural maps commutes. Passing to the adjoint maps we see that the square
\[
	\xymatrix{
	h^* \ar[r]^-{\ad_{u'}} \ar[d]_-{\ad_u} & Ru'_*u'^*h^* \\
	h^*Ru_*u^* \ar[r]^-{\bc} & Ru'_*h'^*u^* \ar[u]_-{\sim}
	}
\]	
is commutative. Applying the derived local cohomology functor and taking the inverse of the now invertible units of adjunction (\autoref{Gammajung}), we obtain the commutative diagram
\[
	\xymatrix{
		h^*Ru_*R\Gamma_{Z'}u^* \ar[r]^-{h^* \tr_u} \ar[d]_-{\bc} \ar@{.>}[dr]^-{d}& h^* \\
		Ru'_*h'^*R\Gamma_{Z'}u^* \ar[r]^-{\sim} & Ru'_*R\Gamma_{Z_U'}u'^*h^*, \ar[u]_-{\tr_{u'} h^*} 
	}
\]
where $d$ denote the composition $h^*Ru_*R\Gamma_{Z'}u^* \overset{\bc}{\longrightarrow} Ru'_*h'^*R\Gamma_{Z'}u^* \overset{\sim}{\longrightarrow} Ru'_*R\Gamma_{Z_U'}u'^*h^*$.

Now we choose a compactification $X \xrightarrow{j} \LX \xrightarrow{\overline{f}} Y$ of $f$. Then $S \times_Y X \xrightarrow{j'} S \times_Y \LX \xrightarrow{\overline{f'}} S$ is a compactification of $f'$ where $j' := \id \times j$ and $\overline{f'}$ is the projection. Let $\overline{g'}\colon S \times_Y \LX \to \LX$ denote the projection onto $\LX$. The three squares in the commutative diagram
\[
	\xymatrix{
		& S \times_Y \LX \ar[rr]^-{\overline{g'}} \ar@/^4mm/[dddl]^{\overline{f'}} & & \LX \ar@/^4mm/[dddl]^{\overline{f}} \\
		S \times_Y X \ar[ur]^{j'} \ar[rr]^/5mm/{g'} \ar[dd]_{f'} & & X \ar[ur]^j \ar[dd]_f & \\
		& & & \\
		S \ar[rr]^-g & & Y &
	}
\]	
are cartesian. It suffices to show the commutativity of 
\[
	\xymatrix{
		g^*\overline{f}_*j_*\Gamma_{Z'}j^*\overline{f}^! \ar[r]^-{\bc} \ar[d]_{\sim}^{\tr_j} & \overline{f'}_*\overline{g'}^*j_*\Gamma_{Z'}j^*\overline{f}^! \ar[r]^-{d} \ar[d]_{\sim}^{\tr_j} & \overline{f'}_*j'_*\Gamma_{Z_S'}j'^*\overline{g'}^*\overline{f}^! \ar[r]^-{\beta} \ar[d]_{\sim}^{\tr_{j'}} & \overline{f'}_*j'_*\Gamma_{Z_S'}j'^*\overline{f'}^!g^* \ar[d]_{\sim}^{\tr_{j'}} \\
		g^*\overline{f}_*\Gamma_{Z'}\overline{f}^! \ar[r]^-{\bc} \ar[d]^{\tr_{\overline{f}}} & \overline{f'}_*\overline{g'}^*\Gamma_{Z'} \overline{f}^! \ar[r]^-{\sim} & \overline{f'}_*\Gamma_{Z_S'}\overline{g'}^*\overline{f}^! \ar[r]^-{\beta} & \overline{f'}_* \Gamma_{Z_S'} \overline{f'}^! g^* \ar[d]^{\tr_{\overline{f'}}} \\
		g^* \ar[rrr]^-{\id} & & & g^*,
	}
\]
where we left out the $R$ indicating derived functors to streamline the notation. The first and the third upper square commute because of the naturality of $\tr_j$ and $\tr_{j'}$. The commutativity of the upper square in the middle is the case of an open immersion, which we have already seen. Finally, the commutativity of the bottom rectangle is the case of a proper morphism (\autoref{proptrace}). 
\end{proof}
\begin{proposition} \label{tracecomp}
Let $f\colon X \to Y$ and $g\colon Y \to S$ be separated and finite type morphisms of schemes. Assume that $i\colon Z \to S$, $i'\colon Z' \to Y$ and $i''\colon Z'' \to X$ are closed immersions with proper morphisms $f'\colon Z'' \to Z'$ and $g'\colon Z' \to Z$ making the diagram
\[
	\xymatrix{
		Z'' \ar[r]^-{i''} \ar[d]^-{f'} & X \ar[d]^-f \\
		Z' \ar[r]^-{i'} \ar[d]^-{g'} & Y \ar[d]^-g \\
		Z \ar[r]^-i & S
	}
\]
commutative. Then there is a commutative diagram:  
\[
	\xymatrix@C40pt@R30pt{
		R(g \circ f)_*R\Gamma_{Z''}(g \circ f)^! \ar[d]^{\sim} \ar[drr]^-{\tr_{g \circ f}} & & \\ 
		Rg_*Rf_*R\Gamma_{Z''}f^!g^! \ar[r]_-{Rg_* \tr_f g^!} & Rg_*R\Gamma_{Z'}g^! \ar[r]_-{\tr_g} & \id.
	}
\]	
\end{proposition}
\begin{proof}
Choose a compactification $Y \xrightarrow{j_Y} \overline{Y} \xrightarrow{\overline{g}} S$ of $g$, then choose a compactification $X \xrightarrow{j_X} \LX \xrightarrow{f'} \overline{Y}$ of the composition $j_Y \circ f$. The morphisms $f$ and $j_X$ induce a morphism $h\colon X \to Y \times_{\overline{Y}}\LX$. The projection $\pr_Y\colon Y \times_{\overline{Y}} \LX \to Y$ is proper because it is the base change of the proper morphism $f'$. The projection $\pr_{\LX}\colon X \times_{\overline{Y}}\LX \to \LX$ is a base change of $j_Y$ and hence an open immersion. We obtain the following commutative diagram:
\[
	\xymatrix{
		X \ar[d]_-f \ar[r]^-{h} & Y \times_{\overline Y} \LX \ar[dl]^-{\pr_{Y}} \ar[r]^-{\pr_{\LX}} & \LX \ar[dl]^-{f'} \\
		Y \ar[d]_-g \ar[r]_-{j_Y} & \overline{Y} \ar[dl]^-{\overline{g}} & \\
		S.
	}
\]
Because $\pr_{\LX} \circ h$ equals the open immersion $j_X$, it follows that $h$ is an open immersion too. Applying the compatibility of the classical trace with compositions to the proper morphisms $f'$ and $\overline{g}$ and using \autoref{traceres} for $f'$ and the open immersion $j_Y$, we see that the following diagram commutes: 
\[
	\xymatrix@C3pt{
	(g \circ f)_*\Gamma_{Z''}(g \circ f)^! \ar[rr]^-{\sim} \ar[d]^{\sim} & &  (\overline{g} \circ f' \circ \pr_{\LX} \circ h)_* \Gamma_{Z''} (\overline{g} \circ f' \circ \pr_{\LX} \circ h)^! \ar[d]_{\sim} \\
	\overline{g}_*j_{Y*}\pr_{Y*}h_*\Gamma_{Z''}h^*\pr_Y^!j_Y^*\overline{g}^! \ar[rr]^-{\sim} \ar[d]_-{\epsilon_h^{-1}}^{\sim} & & \overline{g}_*f'_*\pr_{\LX *}h_*\Gamma_{Z''}h^*\pr_{\LX}^*f'^!\overline{g}^! \ar[d]_{\sim}^{\epsilon_h^{-1}} \\
	\overline{g}_*j_{Y*}\pr_{Y*}\Gamma_{Z''}\pr_Y^!j_Y^*\overline{g}^! \ar[r]^-{\sim}_-{\psi} \ar[d]^{\tr_{\pr_Y}} & \overline{g}_*j_{Y*}j_Y^*f'_* \Gamma_{Z''}f'^! \overline{g}^! \ar[d]^{\tr_{f'}} \ar[dr]_-{\epsilon_{j_Y}^{-1}}^-{\sim} & \overline{g}_*f'_*\pr_{\LX *}\Gamma_{Z''}\pr_{\LX}^*f'^!\overline{g}^! \ar[l]_-{\sim} \ar[d]_{\sim}^{\epsilon_{\pr_{\LX}}^{-1}} \\
	\overline{g}_*j_{Y*}\Gamma_{Z'}j_Y^*\overline{g}^! \ar[r]^-{\sim} \ar[dr]_-{\epsilon_{j_Y}^{-1}}^-{\sim} & \overline{g}_*j_{Y*}j_Y^*\Gamma_{Z'}\overline{g}^! \ar[d]_{\sim}^{\epsilon_{j_Y}^{-1}} & \overline{g}_*f'_*\Gamma_{Z''}f'^!\overline{g}^! \ar[dl]_-{\tr_{f'}} \ar[ddl]^-{\tr_{(\overline{g} \circ f')}} \\
	& \overline{g}_*\Gamma_{Z'}\overline{g}^! \ar[d]^{\tr_{\overline{g}}} & \\
	& \id. &
	}	
\]
Again we left out the $R$ for derived functors to streamline the notation. Here, for an open immersion $s$, $\epsilon_s$ denotes the isomorphism of \autoref{Gammajung} or simple modifications, such as a canonical isomorphism $R\Gamma_{Z'} \overset{\epsilon_{j_Y}}{\longrightarrow} Rj_{Y*}j^*R\Gamma_{Z'}$, and $\psi$ is obtained from the composition
\[
	\pr_{Y*}\pr_Y^!j_Y^* \xrightarrow{\beta^{-1}} \pr_{Y*}\pr_{\LX}^*f'^! \xrightarrow{\bc^{-1}} j_Y^*f'_*f'^!. 
\]
The morphism $\tr_g \circ (Rg_* \tr_f g^!)$ is obtained by following the vertical and diagonal arrows on the left, while $\tr_{g \circ f}$ is the composition of the outer right morphisms. 
\end{proof}
\begin{proof}[Proof of \autoref{qcohadjunction}]
Consider the commutative diagram
\[
	\xymatrix{
		Rf_*\RSHom_{\CO_X}^{\bullet}(\CF^{\bullet},R\Gamma_{Z'}f^!\CG^{\bullet}) \ar[d] \ar[dr] & \\
		R\overline{f}_*\RSHom_{\CO_{\LX}}^{\bullet}(Rj_*\CF^{\bullet},Rj_*R\Gamma_{Z'}j^*\overline{f}^!\CG^{\bullet}) \ar[r] \ar[d]^{\epsilon^{-1}} & \RSHom_{\CO_Y}^{\bullet}(Rf_*\CF^{\bullet},R\overline{f}_*Rj_*R\Gamma_{Z'}j^*\overline{f}^!\CG^{\bullet}) \ar[d] \\
		R\overline{f}_*\RSHom_{\CO_{\LX}}^{\bullet}(Rj_*\CF^{\bullet},R\Gamma_{Z'}\overline{f}^!\CG^{\bullet}) \ar[r] \ar[d] & \RSHom_{\CO_Y}^{\bullet}(Rf_*\CF^{\bullet},R\overline{f}_*R\Gamma_{Z'}\overline{f}^!\CG^{\bullet}) \ar[d] \\
		R\overline{f}_*\RSHom_{\CO_{\LX}}^{\bullet}(Rj_*\CF^{\bullet},\overline{f}^!\CG^{\bullet}) \ar[dr] \ar[r] & \RSHom_{\CO_Y}^{\bullet}(Rf_*\CF^{\bullet},R\overline{f}_*\overline{f}^!\CG^{\bullet}) \ar[d] \\
		& \RSHom_{\CO_Y}^{\bullet}(Rf_*\CF^{\bullet},\CG^{\bullet})
	}
\]
of natural morphisms. The vertical arrows on the left are isomorphisms by \autoref{opennattrans}, \autoref{Gammajung} and \autoref{Gammaadj}. The diagonal morphism to the lower right corner is the well-known isomorphism from the adjunction for the proper morphism $\overline{f}$. Hence the composition of the first diagonal morphism and the vertical morphisms on the right is an isomorphism.  

Finally, for the adjunction of $Rf_*$ and $R\Gamma_{Z'}f^!$, we apply the degree zero cohomology of the right derived functor of global sections $H^0R\Gamma$ to both sides of the just proven isomorphism 
\[
	\RSHom_{\CO_Y}^{\bullet}(Rf_*\CF^{\bullet},\CG^{\bullet}) \overset{\sim}{\longrightarrow} Rf_*\RSHom_{\CO_X}^{\bullet}(\CF^{\bullet},R\Gamma_{Z'}f^!\CG^{\bullet}).
\]
Then we use the natural isomorphisms 
\begin{align*}
	H^0R\Gamma\RSHom_{\CO_Y}^{\bullet}(Rf_*\CF^{\bullet},\CG^{\bullet}) &\overset{\sim}{\longrightarrow} H^0\RHom_{\CO_Y}^{\bullet}(Rf_*\CF^{\bullet},\CG^{\bullet}) \\
	&\overset{\sim}{\longrightarrow} \Hom_{D(\CO_Y)}(Rf_*\CF^{\bullet},\CG^{\bullet})
\end{align*}
of Proposition II.5.3 and Theorem I.6.4 of \cite{HartshorneRD} and similarly for 
\[
	H^0R\Gamma Rf_*\RSHom_{\CO_X}^{\bullet}(\CF^{\bullet},R\Gamma_{Z'}f^!\CG^{\bullet}), 
\]
where we additionally use the isomorphism $R\Gamma(X,\usc) \overset{\sim}{\longrightarrow} R\Gamma(Y,Rf_*(\usc))$ of Proposition II.5.2 of ibid.
\end{proof}
We conclude this section with a statement which, under certain hypothesis, allows us to recover the trace $\tr_f$ by its application to the structure sheaf $\CO_Y$. 
\begin{definition} \label{EssentiallyPerfect}
For a separated morphism $f\colon X \to Y$ of finite type, a compactification $X \xrightarrow{j} \LX \xrightarrow{\overline{f}} Y$ and $\CF^{\bullet} \in D_{\qc}^+(Y)$, let 
\[
	\chi_{\CF^{\bullet}}^f\colon f^!\CO_Y \derotimes_{\CO_X} Lf^*\CF^{\bullet} \to f^!\CF^{\bullet}
\]
be the morphism $j^*\phi$, where $\phi\colon \overline{f}^!\CO_Y \derotimes_{\CO_X} L\overline{f}^*\CF^{\bullet} \to \overline{f}^!\CF^{\bullet}$ is the adjoint of the composition
\[
	R\overline{f}_*(\overline{f}^!\CO_Y \derotimes_{\CO_X} L\overline{f}^*\CF^{\bullet}) \xrightarrow{\rho} R\overline{f}_*(\overline{f}^!\CO_Y) \derotimes_{\CO_Y} \CF^{\bullet} \xrightarrow{\tr_{\overline{f}}} \CO_Y \otimes_{\CO_Y} \CF^{\bullet}.
\]
Here $\rho$ denotes the isomorphism of the projection formula. The morphism $\chi_{\CF^{\bullet}}^f$ is independent of the choice of the compactification \cite[Proposition 5.8]{Na.CompEFT}. If $\chi^f$ is an isomorphism of functors, then the morphism $f$ is called \emph{essentially perfect}. 
\end{definition}
Theorem 5.9 of \cite{Na.CompEFT} Nayak gives various characterizations of essentially perfect morphisms. For example, smooth morphisms between smooth schemes are essentially perfect.
\begin{proposition} \label{Nayak}
Let $f$ be an essentially perfect map. Under the assumptions and with the notation of \autoref{qcohadjunction}, there is a commutative diagram
\[
	\xymatrix{
		Rf_*R\Gamma_{Z'}(f^!\CO_Y \derotimes_{\CO_X} Lf^*\CF^{\bullet}) \ar[r]^-{\chi_{\CF^{\bullet}}^f} \ar[d]^-{\rho} & Rf_*R\Gamma_{Z'}f^!\CF^{\bullet} \ar[d]^-{\tr_f} \\
		Rf_*R\Gamma_{Z'}f^!\CO_Y \derotimes_{\CO_Y} \CF^{\bullet} \ar[r]^-{\tr_f \otimes \id} & \CF^{\bullet}
	}
\]
for every $\CF^{\bullet} \in D_{\qc}^+(Y)$. Here $\rho$ denotes the isomorphism of the projection formula and \autoref{RGammaTensor}.
\end{proposition}
\begin{proof}
We have to verify that the diagram
\begin{align} \label{NayakDiagram}
	\xymatrix{
		Rf_*R\Gamma_{Z'}(f^!\CO_Y \derotimes_{\CO_X} Lf^*\CF^{\bullet}) \ar[r]^-{\chi_{\CF^{\bullet}}^f} \ar[d]^-{\sim} & Rf_*R\Gamma_{Z'}f^!\CF^{\bullet} \ar[d]_-{\sim} \\
		R\overline{f}_*Rj_*R\Gamma_{Z'}(j^*\overline{f}^!\CO_Y \derotimes j^*L\overline{f}^*\CF^{\bullet}) \ar[d]^-{\rho} & R\overline{f}_*Rj_*R\Gamma_{Z'}j^*\overline{f}^!\CF^{\bullet} \ar[dd]_-{\sim}^{\epsilon^{-1}} \\
		R\overline{f}_*(Rj_*R\Gamma_{Z'}j^*\overline{f}^!\CO_Y \derotimes L\overline{f}^*\CF^{\bullet}) \ar[d]_-{\epsilon^{-1}}^-{\sim} & \\
		R\overline{f}_*(R\Gamma_{Z'}\overline{f}^!\CO_Y \derotimes L\overline{f}^*\CF^{\bullet}) \ar[d]^-{\rho} \ar[r]^-{\phi} & R\overline{f}_*R\Gamma_{Z'}\overline{f}^!\CF^{\bullet} \ar[dd]^-{\tr_{\overline{f}}} \\
		Rf_*R\Gamma_{Z'}f^!\CO_Y \derotimes_{\CO_Y} \CF^{\bullet} \ar[d]^{\tr_f \otimes \id} & \\
		\CO_Y \otimes \CF^{\bullet} \ar[r]^-{\sim} & \CF^{\bullet}
	}
\end{align}
commutes. The upper rectangle commutes because the projection formula is compatible with the unit $\id \to Rj_*j^*$ of adjunction, which we denote by $\ad_j$. More precisely, it follows from the commutativity of the diagram
\[
	\xymatrix@C40pt{
	\overline{f}^!\CO_Y \derotimes L\overline{f}^* \CF^{\bullet} \ar[r] \ar[d] & Rj_*j^*(\overline{f}^!\CO_Y \derotimes L\overline{f}^* \CF^{\bullet}) \ar[d] \\
	(Rj_*j^*\overline{f}^!\CO_Y) \derotimes L\overline{f}^* \CF^{\bullet} \ar[r] \ar[d]^-{\proj} & Rj_*j^*((Rj_*j^*\overline{f}^!\CO_Y) \derotimes L\overline{f}^* \CF^{\bullet}) \ar[d]_-{\sim} \\
	Rj_*(j^*\overline{f}^!\CO_Y \derotimes j^*L\overline{f}^* \CF^{\bullet}) & Rj_*(j^*(Rj_*j^*\overline{f}^!\CO_Y) \derotimes j^*L\overline{f}^* \CF^{\bullet}), \ar[l]_-{\widetilde{\ad}_j}
	}
\]
where the maps of the upper square stem from $\ad_j$ -- this square commutes by the naturality of the unit of adjunction --, where $\proj$ is the isomorphism from the projection formula and where the lower horizontal arrow is obtained from the counit of adjunction $\widetilde{\ad}_j\colon j^*Rj_* \to \id$. The lower rectangle commutes by construction of $\proj$. The composition
\[
	j^* \xrightarrow{j^* \ad_j} j^*Rj_*j^* \xrightarrow{\widetilde{\ad}_j j^*} j^*
\]
is the identity. Therefore, the composition of the vertical arrows on the right hand side and the lower horizontal arrow equals $Rj_*$ applied to the natural isomorphism 
\[
	j^*(\overline{f}^!\CO_Y \derotimes L\overline{f}^* \CF^{\bullet}) \to j^*\overline{f}^!\CO_Y \derotimes j^*L\overline{f}^* \CF^{\bullet}.
\]
The bottom rectangle of the diagram \autoref{NayakDiagram} commutes by construction of $\chi_{\CF^{\bullet}}^{\overline{f}}$.  
\end{proof}

\section{Locally finitely generated unit modules on singular schemes}

For a proper map $f\colon X \to Y$ of smooth $k$-schemes, Emerton and Kisin proved that there is a natural isomorphism
\[
	\RSHom_{\CO_{F,Y}}^{\bullet}(f_+\CM^{\bullet},\CN^{\bullet}) \overset{\sim}{\longrightarrow} Rf_*\RSHom_{\CO_{F,X}}^{\bullet}(\CM^{\bullet},f^!\CN^{\bullet})
\]
for $\CM^{\bullet} \in D_{\qc}^b(\CO_{F,X})$ and $\CN^{\bullet} \in D_{\qc}^b(\CO_{F,Y})$ (\cite[Theorem 4.4.1]{EmKis.Fcrys}) by constructing a trace map acting as the counit of adjunction. We generalize this trace map to separated and finite type morphisms $f\colon X \to Y$ between smooth $k$-schemes sitting in a commutative diagram
\[
	\xymatrix{
		Z' \ar[r]^{i'} \ar[d]^-{f'} & X \ar[d]^-f \\
		Z \ar[r]^i & Y,
	}
\]
where $i$ and $i'$ are closed immersions and $f'$ is proper. This generalized trace map induces an adjunction between $f_+$ and $R\Gamma_{Z'}f^!$ considered as functors between the derived categories $D_{\lfgu}^b(\CO_{F,X})_{Z'}$ and $D_{\lfgu}^b(\CO_{F,Y})_Z$ of complexes whose cohomology sheaves are supported in $Z'$ or $Z$. 

\subsection{Generalization of Emerton-Kisin's adjunction}

\begin{proposition}  
Let $f\colon X \to Y$ be a separated and finite type morphism of smooth $k$-schemes and let $i\colon Z \to Y$ and $i'\colon Z' \to X$ be closed immersions with a proper morphism $f'\colon Z' \to Z$ such that $f \circ i' = i \circ f'$. 
\begin{enumerate}
\item There is a natural morphism
\[
	\tr_{F,f}\colon f_+R\Gamma_{Z'}\CO_{F,X}[d_{X/Y}] \to \CO_{F,Y}
\]
of $(\CO_{F,Y},\CO_{F,Y})$-bimodules which, as a morphism of left $\CO_{F,Y}$-modules, is the trace
\[
	\CO_{F,Y} \otimes_{\CO_Y} Rf_*R\Gamma_{Z'}\omega_{X/Y}[d_{X/Y}] \to \CO_{F,Y}
\]
of \autoref{tracedef}. 
\item For every $\CM^{\bullet} \in D_{\qc}^b(\CO_{F,Y})$, the trace map $\tr_{F,f}$ induces a morphism 
\[
	\tr_{F,f}(\CM^{\bullet})\colon f_+R\Gamma_{Z'}f^!\CM^{\bullet} \to \CM^{\bullet}
\]
in $D_{\qc}^b(\CO_{F,Y})$. 
\end{enumerate}
\end{proposition}
\begin{proof}
(a) This is an analogue of \cite[Proposition 4.4.9 (i)]{EmKis.Fcrys}. A careful reading of the proof shows that we can adopt it. Consider the relative Frobenius diagram (diagram \autoref{relativeFrobenius} on page \pageref{relativeFrobenius}):
\[
	\xymatrix{
		X \ar[r]^-{F_{X/Y}} \ar[dr]_f & X' \ar[r]^-{F_Y'} \ar[d]^-{f'} & X \ar[d]^-f \\
		& Y \ar[r]^-{F_Y} & Y.
	}
\]
Since $X$ and $Y$ are assumed to be smooth $k$-schemes, we still have flatness of the Frobenius $F_Y$ and therefore of $F_Y'$ because flatness is stable under base change. Note that $F_{X/Y}$ is finite (\cite[A.2]{EmKis.Fcrys}). First Emerton and Kisin explain how the relative Cartier operator 
\[
	C_{X/Y}\colon F_{X/Y*}\omega_{X/Y} \to F_Y'^*\omega_{X/Y} 
\]
is realized for the residual complex $f^{\Delta}E^{\bullet}$. Here $E^{\bullet}$ denotes the Cousin complex $E^{\bullet}(\CO_X)$. For our result we replace $f^{\Delta}E^{\bullet}$ by the subcomplex $\Gamma_{Z'}f^{\Delta}E^{\bullet}$ of flasque sheaves which computes $R\Gamma_{Z'}\omega_{X/Y}$. We obtain the \emph{relative Cartier operator with support on $Z'$}:
\[
	C_{X/Y}^Z\colon F_{X/Y*}\Gamma_{Z'}f^{\Delta}E^{\bullet} \to F_Y'^*\Gamma_{Z'}f^{\Delta}E^{\bullet}. 
\]   
By Proposition-Definition 1.10.1 of \cite{EmKis.Fcrys}, $f^{-1}\CO_{F,Y} \otimes_{f^{-1}\CO_Y} \Gamma_{Z'}f^{\Delta}E^{\bullet}$ is equipped with a $(f^{-1}\CO_{F,Y},\CO_{F,X})$-bimodule structure or, after restricting scalars via the natural map $f^{-1}\CO_Y[F] \to \CO_X[F]$, with a $(f^{-1}\CO_{F,Y},f^{-1}\CO_{F,Y})$-bimodule structure. Finally this endows $f_*(f^{-1}\CO_{F,Y} \otimes_{f^{-1}\CO_Y} \Gamma_{Z'}f^{\Delta}E^{\bullet})$ with the structure of a $(\CO_{F,Y},\CO_{F,Y})$-bimodule, the one from the definition of $f_+R\Gamma_{Z'}\CO_{F,X}$.  

But there is another way to look at this bimodule: The map $C_{X/Y}^Z$ gives rise to a morphism 
\[
	\sigma\colon f_*\Gamma_{Z'}f^{\Delta}E^{\bullet} \to F_Y^*f_*\Gamma_{Z'}f^{\Delta}E^{\bullet}
\]
by the composition 
\[
	f_*\Gamma_{Z'}f^{\Delta}E^{\bullet} \overset{\sim}{\longrightarrow} f'_*{F_{X/Y}}_*\Gamma_{Z'}f^{\Delta}E^{\bullet} \xrightarrow{C_{X/Y}^Z}f'_*F_Y'^*\Gamma_{Z'}f^{\Delta}E^{\bullet} \overset{\bc^{-1}}{\longrightarrow} F_Y^*f_*\Gamma_{Z'}f^{\Delta}E^{\bullet},
\]
where the first isomorphism is deduced from $f=f' \circ F_{X/Y}$ and the last isomorphism is flat base change. Now Proposition-Definition 1.10.1 of ibid.\ in the special case of the morphism $\id_Y$ yields a $(\CO_{F,Y},\CO_{F,Y})$-bimodule structure on $\CO_{F,Y} \otimes_{\CO_Y} f_*\Gamma_{Z'}f^{\Delta}E^{\bullet}$. The isomorphism 
\[
	f_*(f^{-1}\CO_{F,Y} \otimes_{f^{-1}\CO_Y} \Gamma_{Z'}f^{\Delta}E^{\bullet}) \cong \CO_{F,Y} \otimes_{\CO_Y} f_*\Gamma_{Z'}f^{\Delta}E^{\bullet}
\]
stemming from the projection formula is compatible with the constructed bimodule structure for both complexes by Lemma 1.10.6 of ibid. Hence it suffices to show that $\tr_{F,f}$ induces a morphism between the $(\CO_{F,Y},\CO_{F,Y})$-bimodule $\CO_{F,Y} \otimes_{\CO_Y} f_*\Gamma_{Z'}f^{\Delta}E^{\bullet}$ and $E^{\bullet}$ equipped with the structure of a  $(\CO_{F,Y},\CO_{F,Y})$-bimodule via the canonical isomorphism $E^{\bullet} \overset{\sim}{\longrightarrow} F_Y^*E^{\bullet}$ induced from the Frobenius $\CO_Y \to F_Y^* \CO_Y$. Lemma 1.10.2 of ibid.\ applied to the identity morphism on $Y$ reduces to the commutativity of the diagram  
\[
	\xymatrix@C50pt{
		f_*\Gamma_{Z'}f^{\Delta}E^{\bullet} \ar[r]^-{\tr_{F,f}} \ar[d]^{\sigma} & E^{\bullet} \ar[d]^{\sim} \\
		F_Y^*f_*\Gamma_{Z'}f^{\Delta}E^{\bullet} \ar[r]^-{F_Y^*\tr_{F,f}} & F_Y^*E^{\bullet}
	}
\]
of complexes. For this we have to see that the following bigger diagram commutes: 
\[
	\xymatrix@C40pt{
		f_*\Gamma_{Z'}f^{\Delta}E^{\bullet} \ar[rr]^-{\tr_f} \ar[d]^{\sim} & & E^{\bullet} \ar@{=}[d] \\
		f'_*F_{X/Y*}\Gamma_{Z'}F_{X/Y}^{\Delta}f'^{\Delta}E^{\bullet} \ar[r]^-{\tr_{F_{X/Y}}} \ar[d]^{\sim} & f'_*\Gamma_{Z'}f'^{\Delta}E^{\bullet} \ar[r]^-{\tr_{f'}}  \ar[d]^{\sim} & E^{\bullet}\ar[d]^{\sim} \\
		f'_*F_{X/Y*}\Gamma_{Z'}F_{X/Y}^{\Delta}f'^{\Delta}F_Y^*E^{\bullet} \ar[r]^-{\tr_{F_{X/Y}}} \ar[d]_{\beta}^{\sim} & f'_*\Gamma_{Z'}f'^{\Delta}F_Y^*E^{\bullet} \ar[r]^-{\tr_{f'}} \ar[d]_{\beta}^{\sim} & F_Y^*E^{\bullet} \ar@{=}[ddd] \\
		f'_*F_{X/Y*}\Gamma_{Z'}F_{X/Y}^{\Delta}F_Y'^*f^{\Delta}E^{\bullet} \ar[r]^-{\tr_{F_{X/Y}}} & f'_*\Gamma_{Z'}F_Y'^*f^{\Delta}E^{\bullet} \ar[d]^{\sim} & \\
		& f'_*F_Y'^*\Gamma_{Z'}f^{\Delta}E^{\bullet} \ar[d]_{\bc^{-1}}^{\sim} & \\
		& F_Y^*f_*\Gamma_{Z'}f^{\Delta}E^{\bullet} \ar[r]^-{\tr_f} & F_Y^*E^{\bullet}.
	}
\]
The commutativity of the top rectangle follows from \autoref{tracecomp}. The two squares below it and the lower left square commute by functoriality of the trace maps $\tr_{F_{X/Y}}$ and $\tr_{f'}$. The commutativity of the bottom right rectangle is the compatibility of the trace with base change by the residually stable map $F_Y$ (\autoref{traceres}). \footnote{Here we have to show that we have commutative diagrams of complexes, which is not exactly the claim of \autoref{tracecomp} and \autoref{traceres}. However, one can check that the diagrams in the proofs of these propositions commute as diagrams of complexes when applied to residual complexes. For this one should keep in mind that for a complex $K^{\bullet}$, $\Gamma_{Z'}K^{\bullet}$ is a subcomplex of $K^{\bullet}$ and that there is a natural isomorphism of functors $\Gamma_{Z'}F_Y'^* \overset{\sim}{\longrightarrow} F_Y'^*\Gamma_{Z'}$  between abelian categories because $F_Y'$ is flat, see \autoref{underGammaPullback}.}

(b) Once we know that $\tr_f$ is a morphism in $D_{\qc}^b(\CO_{F,Y})$, we can define $\tr_f(\CM^{\bullet})$ as the following composition:
\begin{align*}
	f_+R\Gamma_{Z'}f^!\CM^{\bullet} &\overset{\sim}{\longrightarrow} f_+(\CO_{F,X} \otimes_{\CO_{F,X}} R\Gamma_{Z'}f^!\CM^{\bullet}) \\
	&\longrightarrow f_+(R\Gamma_{Z'}\CO_{F,X} \derotimes_{\CO_{F,X}} f^!\CM^{\bullet}) \\
	&\longrightarrow f_+R\Gamma_{Z'}\CO_{F,X}[d_{X/Y}] \derotimes_{\CO_{F,Y}} \CM^{\bullet} \\
	&\xrightarrow{\tr_f \otimes \id} \CO_{F,Y} \otimes_{\CO_{F,Y}} \CM^{\bullet} \\
	&\overset{\sim}{\longrightarrow} \CM^{\bullet}.
\end{align*}
Here the second morphism is the one of \autoref{RGammaTensor} and the third morphism is the one of \cite[Lemma 4.4.7]{EmKis.Fcrys}.
\end{proof}
\begin{lemma} \label{traceforward}
We keep the notation of the preceding proposition. For an open immersion $j\colon U \to Y$, let $f'$ and $j'$ denote the projections of $U' = U \times_Y X$. Assume that $Z$ and $Z'$ are the closures of the locally closed subsets $Z_U= Z \cap U$ and $Z'_{U'} = Z' \cap U'$ in $Y$ and in $X$.
\begin{enumerate}
\item There is a functorial isomorphism $e_{j,f}\colon f_+R\Gamma_Zf^!j_+ \overset{\sim}{\longrightarrow} j_+f'_+R\Gamma_{Z'}f'^!$ such that the diagram
\[
	\xymatrix{
		f_+R\Gamma_Zf^!j_+ \ar[r]^-{e_{j,f}} \ar[dr]_-{\tr_f j_+} & j_+f'_+R\Gamma_{Z'}f'^! \ar[d]^{j_+\tr_{f'}} \\
		& j_+
	}
\]
commutes.
\item Let $\ctr_f$ denote the unit $\id \to R\Gamma_Zf^!f_+$ of the adjunction. Then there is a functorial isomorphism $e'_{j,f}\colon R\Gamma_Zf^!f_+j'_+ \overset{\sim}{\longrightarrow} j'_+R\Gamma_{Z'}f'^!f'_+$ such that the diagram
\[
	\xymatrix{
		R\Gamma_Zf^!f_+j'_+ \ar[r]^-{e_{j,f}} & j'_+R\Gamma_{Z'}f'^!f'_+ \\
		& j'_+ \ar[ul]^-{\ctr_f j'_+} \ar[u]_{j'_+\ctr_{f'}} 
	}
\]
commutes.
\end{enumerate}
\end{lemma}
\begin{proof}
Let $Z_U$ and $Z'_{U'}$ denote the closed subsets $U \cap Z$ and $U' \cap Z'$ of $U$ and $U'$. From \autoref{opensupport} we know that the functors $j_+$ and $j'_+$ are equivalences 
\[
	D_{\lfgu}^b(\CO_{F,U})_{Z_U} \overset{\sim}{\longrightarrow} D_{\lfgu}^b(\CO_{F,Y})_Z^{Z \backslash U} \text{ and } D_{\lfgu}^b(\CO_{F,U'})_{Z'_{U'}} \overset{\sim}{\longrightarrow} D_{\lfgu}^b(\CO_{F,X})_{Z'}^{Z' \backslash U'}.
\]
Furthermore, there are natural isomorphisms
\[
	f_+j'_+ \overset{\sim}{\longrightarrow} j_+f'_+ \text{ and } R\Gamma_{Z'}f^!j_+ \overset{\sim}{\longrightarrow} j'_+R\Gamma_{Z'_{U'}}f'^!,
\]
where the second one is obtained from the composition  
\[
	j'^!R\Gamma_{Z'}f^! \overset{\sim}{\longrightarrow} R\Gamma_{Z'_{U'}}j'^!f^! \overset{\sim}{\longrightarrow} R\Gamma_{Z'_{U'}}f'^!j^!
\]
of natural isomorphisms. Moreover, together with the canonical isomorphism $f'_+j'^! \cong j^!f_+$ of \cite[Proposition 3.8]{EmKis.Fcrys}, this composition yields a canonical isomorphism
\[
	\tilde{e}_{j,f}\colon f'_+R\Gamma_{Z'}f'^!j^! \to j^!f_+R\Gamma_Zf^!.
\]
For $\CM^{\bullet} \in D_{\qc}^b(\CO_{F,Y})$, the diagram
\[
	\xymatrix{
		f'_+R\Gamma_{Z'}f'^!j^!\CM^{\bullet} \ar[r]^-{\tilde{e}_{j,f}} \ar[d]^-{\sim} & j^!f_+R\Gamma_Zf^!\CM^{\bullet} \ar[d]^-{\sim} \\
		f'_+(\CO_{F,U'} \otimes_{\CO_{F,U'}} R\Gamma_{Z'}f'^!j^!\CM^{\bullet}) \ar[r]^-{\sim} \ar[d]^-{\sim} & j^!f_+(\CO_{F,X} \otimes_{\CO_{F,X}} R\Gamma_Zf^!\CM^{\bullet}) \ar[d]^-{\sim} \\
		f'_+R\Gamma_{Z'}f'^!\CO_{F,U'} \derotimes_{\CO_{F,U}} j^!\CM^{\bullet} \ar[r]^-{\tilde{e}_{j,f}} \ar[d]_-{\tr_{f'}} & j^!f_+R\Gamma_Zf^!\CO_{F,X} \derotimes_{\CO_{F,U}} j^!\CM^{\bullet} \ar[dl]^-{\tr_f} \\
		j^!\CM^{\bullet} &
	}
\]
of natural isomorphisms and the trace commutes: While the first square commutes simply by functoriality, the commutativity of the second square follows from \cite[Lemma 4.4.7 (ii)]{EmKis.Fcrys}. For the triangle on the bottom we apply \autoref{traceres}. In summary the diagram
\[
	\xymatrix{
		f'_+R\Gamma_{Z'}f'^!j^! \ar[r]^-{\tilde{e}_{j,f}} \ar[d]_-{\tr_{f'}j^!} & j^!f_+R\Gamma_Zf^! \ar[dl]^-{j^! \tr_f} \\
			j^! &
	}
\]
is commutative. Since $j^!$ and $j'^!$ are quasi-inverses of $j_+$ and $j'_+$ and $f_+$ and $R\Gamma_{Z'}f^!$ restrict to the functors $f'_+$ and $R\Gamma_{Z'_{U'}}f'^!$ between $D_{\lfgu}^b(\CO_{F,U})_{Z_U}$ and $D_{\lfgu}^b(\CO_{F,U'})_{Z'_{U'}}$ with respect to the equivalences $j^!$ and $j'^!$, the claims of the lemma are formal consequences.  
\end{proof}
\begin{theorem} \label{Theorem}
Let $f\colon X \to Y$ be a separated and finite type morphism of smooth schemes and let $i\colon Z \to Y$ and $i'\colon Z' \to X$ be closed immersions with a morphism $f'\colon Z' \to Z$ such that the diagram
\[
	\xymatrix{
		Z' \ar[r]^-{i'} \ar[d]^-{f'} & X \ar[d]^-f \\
		Z \ar[r]^-i & Y
	}
\]
commutes. Then, for any $\CM^{\bullet} \in D_{\qc}^b(\CO_{F,X})_{Z'}$ and any $\CN^{\bullet} \in D_{\qc}^b(\CO_{F,Y})_Z$, there is a natural isomorphism
\[
	\RSHom_{\CO_{F,Y}}^{\bullet}(f_+\CM^{\bullet},\CN^{\bullet}) \overset{\sim}{\longrightarrow} Rf_*\RSHom_{\CO_{F,X}}^{\bullet}(\CM^{\bullet},R\Gamma_{Z'}f^!\CN^{\bullet}).
\]
In particular, $f_+\colon D_{\qc}^b(\CO_{F,X})_{Z'} \to D_{\qc}^b(\CO_{F,Y})_Z$ is left adjoint to $R\Gamma_{Z'}f^!$.
\end{theorem}
\begin{proof}
The morphism $f$ factors through the graph morphism $X \times_k Y$, which is a closed immersion, followed by the projection $X \times_k Y \to Y$, which is smooth. Therefore, we may assume that $f$ is an essentially perfect morphism. We show that the natural transformation $\tau$ given by the composition 
\[
	\xymatrix{
		Rf_*\RSHom_{\CO_{F,X}}^{\bullet}(\CM^{\bullet},R\Gamma_{Z'}f^!\CN^{\bullet}) \ar[r] & \RSHom_{\CO_{F,Y}}^{\bullet}(f_+\CM^{\bullet},f_+R\Gamma_{Z'}f^!\CN^{\bullet}) \ar[d]^{\tr_{F,f}} \\
		& \RSHom_{\CO_{F,Y}}^{\bullet}(f_+\CM^{\bullet},\CN^{\bullet})
		}
\]
is an isomorphism in $D^+(X,\ZZ/p\ZZ)$. Here the horizontal arrow is the natural morphism of \cite[Proposition 4.4.2]{EmKis.Fcrys}. Let $\CO_{F,f}$ denote the $(f^{-1}\CO_{F,Y},\CO_{F,X})$-bimodule $\CO_{F,Y \leftarrow X}$ and let $\omega_f$ denote the $\CO_X$-module $\omega_{X/Y}$. We set $d = d_{X/Y}$. First we replace $\CM^{\bullet}$ by a bounded above complex of quasi-coherent induced left $\CO_{F,X}$-modules, i.e.\ left $\CO_{F,X}$-modules of the form $\CO_{F,X} \otimes_{\CO_X} M$ with quasi-coherent $\CO_X$-modules $M$, see Definition 1.7 and Lemma 1.7.1 of \cite{EmKis.Fcrys}. Now by the Lemma on Way-out Functors (\cite[Proposition I.7.1]{HartshorneRD}), we reduce to the case of a single sheaf $\CM^{\bullet} = \CO_{F,X} \otimes_{\CO_X} M$. For such an induced module we have an isomorphism
\begin{align} \label{inducedPushforward}
	f_+\CM \overset{\sim}{\longrightarrow} \CO_{F,Y} \otimes_{\CO_Y} Rf_*(\omega_{X/Y} \otimes_{\CO_X} M),
\end{align}
which is based on the projection formula, see the proof of \cite[Theorem 3.5.3]{EmKis.Fcrys}. Note that in this proof $f^!$ always denotes Emerton-Kisin's pull-back of left $\CO_{F,Y}$-modules, sometimes considered as an $\CO_Y$-module. It is connected to the functor $f^!$ for quasi-coherent sheaves by the canonical isomorphisms 
\[
	f^!\CN^{\bullet} \overset{\sim}{\longrightarrow} Lf^*\CN^{\bullet}[d]\overset{\sim}{\longrightarrow} \omega_{X/Y}^{-1} \otimes_{\CO_X} \text{`}{f^!}\text{'} \CN^{\bullet}
\] 
in $D_{\qc}(X)$, where `$f^!$' denotes the classical $f^!$ as in section 3. We will show that there is a commutative diagram 
\begin{align*}
	\xymatrix{
		Rf_*\RSHom_{\CO_X}^{\bullet}(M,R\Gamma_{Z'}f^!\CN^{\bullet}) \ar[d]^{t} \ar[r]^-{\sim} & Rf_*\RSHom_{\CO_{F,X}}^{\bullet}(\CM,R\Gamma_{Z'}f^!\CN^{\bullet}) \ar[d]^{\tau} \\
		\RSHom_{\CO_Y}^{\bullet}(Rf_*(\omega_{X/Y} \otimes_{\CO_X} M), \CN^{\bullet}) \ar[r]^-{\sim} & \RSHom_{\CO_{F,Y}}^{\bullet}(f_+\CM,\CN^{\bullet})
	}
\end{align*}
with an isomorphism $t$ and where the horizontal arrows are the natural isomorphisms induced by the isomorphism 
\[
	\Hom_{\CO_X}(M, \usc) \overset{\sim}{\longrightarrow} \Hom_{\CO_{F,X}}(\CO_{F,X} \otimes_{\CO_X} M, \usc)
\]
of \cite[1.7.2]{EmKis.Fcrys} and \autoref{inducedPushforward}. For this we consider the bigger diagram of natural maps on page \pageref{bigdiagram}. Let $t$ be the composition of the left vertical arrows. It is an isomorphism by \autoref{Nayak} and \autoref{qcohadjunction}. Recall that $\CO_{F,Y \leftarrow X}$ is locally free as a right $\CO_X$-module and that $\CO_{F,Y \leftarrow X} \derotimes_{\CO_{F,X}} \CM \cong \CO_{F,Y \leftarrow X} \otimes_{\CO_X} M$, which is computed in the proof of Lemma 3.5.1 of \cite{EmKis.Fcrys}. In particular, induced modules are acyclic for the functor $\CO_{F,Y \leftarrow X} \otimes_{\CO_{F,X}} \usc$. For the first square, we consider the diagram without the outer $Rf_*$, resolve $M$ by a complex $P^{\bullet}$ of locally free $\CO_X$-modules and $R\Gamma_{Z'}f^!\CN^{\bullet}$ by a complex $\CJ^{\bullet}$ of left $\CO_{X,F}$-modules which are acyclic for the functor $\CO_{F,Y \leftarrow X} \derotimes_{\CO_{F,X}} \usc$, as in the proof of Proposition 4.4.2 of ibid.\ Now $\mathcal P^{\bullet}=\CO_{F,X} \otimes_{\CO_X} P^{\bullet}$ is a complex of locally free $\CO_{F,X}$-modules. We obtain a commutative diagram
\[
	\xymatrix@C12pt{
		\RSHom_{\CO_X}^{\bullet}(P^{\bullet},\CJ^{\bullet}) \ar@{.>}[r]^-{\sim} \ar[d]^{\sim} & \RSHom_{\CO_{F,X}}^{\bullet}(\mathcal P^{\bullet},\CJ^{\bullet}) \ar[d]^{\sim} \\
		\SHom_{\CO_X}^{\bullet}(P^{\bullet},\CJ^{\bullet}) \ar[r]^-{\sim} \ar[d]^{\sim} & \SHom_{\CO_{F,X}}^{\bullet}(\mathcal P^{\bullet},\CJ^{\bullet}) \ar[d] \\
		\SHom_{\CO_X}^{\bullet}(\omega_f \otimes_{\CO_X} P^{\bullet},\omega_f \otimes_{\CO_X} \CJ^{\bullet}) \ar[d]^{\sim} \ar[r] & \SHom_{f^{-1}\CO_{F,Y}}^{\bullet}(\CO_{F,f} \otimes_{\CO_{F,X}} \mathcal P^{\bullet},\CO_{F,f} \derotimes_{\CO_{F,X}} \CJ^{\bullet}) \ar[d] \\
		\RSHom_{\CO_X}^{\bullet}(\omega_f \otimes_{\CO_X} P^{\bullet},\omega_f \otimes_{\CO_X} \CJ^{\bullet}) \ar@{.>}[r] & \RSHom_{f^{-1}\CO_{F,Y}}^{\bullet}(\CO_{F,f} \otimes_{\CO_{F,X}} \mathcal P^{\bullet},\CO_{F,f} \derotimes_{\CO_{F,X}} \CJ^{\bullet})
	}
\]
of canonical maps. The last two vertical arrows are the canonical morphisms from a functor to its right derived functor. Here the left one is an isomorphism because $\omega_{X/Y} \otimes P^{\bullet}$ is a locally free $\CO_X$-module. 

For the second square, we check that the natural map
\[
	\xymatrix{
		\RSHom_{\CO_X}^{\bullet}(\omega_f \otimes_{\CO_X} M,\omega_f \otimes_{\CO_X} R\Gamma_{Z'}\CN^{\bullet}) \ar[d] \\
		\RSHom_{f^{-1}\CO_{F,Y}}^{\bullet}(\CO_{F,f} \otimes_{\CO_X} M,\CO_{F,f} \derotimes_{\CO_{F,X}} R\Gamma_{Z'}\CN^{\bullet})
	}
\]
factors through $\RSHom_{f^{-1}\CO_Y}^{\bullet}(\omega_f \otimes_{\CO_X} M,\omega_f \otimes_{\CO_X} R\Gamma_{Z'}\CN^{\bullet})$.
For this we replace $\omega_f \otimes_{\CO_X} R\Gamma_{Z'}f^!\CN^{\bullet}$ by a complex $\CI^{\bullet}$ of injective $f^{-1}\CO_Y$-modules and $\CO_{F,f} \derotimes_{\CO_{F,X}} R\Gamma_{Z'}f^!\CN^{\bullet}$ by a complex $\tilde{\CI}^{\bullet}$ of injective $f^{-1}\CO_{F,Y}$-modules. The functor $f^{-1}\CO_{F,Y} \otimes_{f^{-1}\CO_Y} \usc$ is exact because the right $\CO_Y$-module $\CO_{F,Y}$ is free (\cite[Lemma 1.3.1]{EmKis.Fcrys}). Furthermore, it is left adjoint to the forgetful functor from $f^{-1}\CO_{F,Y}$-modules to $f^{-1}\CO_Y$-modules. Hence the latter functor preserves injectives. This implies that $\tilde{\CI}^{\bullet}$ is a complex of injective $f^{-1}\CO_Y$-modules and the canonical morphism $\omega_f \otimes_{\CO_X} R\Gamma_{Z'}f^!\CN^{\bullet} \to \CO_{F,f} \derotimes_{\CO_{F,X}} R\Gamma_{Z'}f^!\CN^{\bullet}$ yields a map $\CI^{\bullet} \to \tilde{\CI}^{\bullet}$. After replacing  $M$ by a complex $P^{\bullet}$ of locally free $\CO_X$-modules as above we have reduced the three $\RSHom$ to $\SHom$ and the claimed factorization is trivial. 

We return to the second square of the diagram on page \pageref{bigdiagram}, where we replace $M$ by a complex $F^{\bullet}$ of flasque $\CO_X$-sheaves. The complexes $\omega_f \otimes_{\CO_X} F^{\bullet}$ and $\CO_{F,f} \otimes_{\CO_X} F^{\bullet}$ are also flasque because locally they are direct sums of flasque sheaves. Hence $f_*(\omega_f \otimes F^{\bullet})$ and $f_*(\CO_{F,f} \otimes F^{\bullet})$ represent $Rf_*(\omega_f \otimes F^{\bullet})$ and $Rf_*(\CO_{F,f} \otimes F^{\bullet})$. As above, we resolve $\omega_f \otimes_{\CO_X} R\Gamma_{Z'}f^!\CN^{\bullet}$ by $\CI^{\bullet}$ and $\CO_{F,f} \derotimes_{\CO_{F,X}} R\Gamma_{Z'}f^!\CN^{\bullet}$ by $\tilde{\CI}^{\bullet}$. The injectivity of $\CI^{\bullet}$ and $\tilde{\CI}^{\bullet}$ implies that $\SHom_{f^{-1}\CO_Y}^{\bullet}(\omega_f \otimes_{\CO_X} F^{\bullet},\CI^{\bullet})$ and $\SHom_{f^{-1}\CO_{F,Y}}^{\bullet}(\CO_{F,f} \otimes_{\CO_X} F^{\bullet},\tilde{\CI}^{\bullet})$ are flasque (\cite[Lemme II.7.3.2]{Godement}) and hence may be used to compute $Rf_*$. As $f_*$ is right adjoint to the exact functor $f^{-1}$, the complex $f_*\CI^{\bullet}$ is a complex of injective $\CO_Y$-modules and $f_*\tilde{\CI}^{\bullet}$ is a complex of injective $\CO_{F,Y}$-modules. Therefore 
\[
	\RSHom_{\CO_X}^{\bullet}(\usc,f_*\CI^{\bullet}) \cong \SHom_{\CO_X}^{\bullet}(\usc,f_*\CI^{\bullet})
\]
and 
\[
	\RSHom_{\CO_{F,X}}^{\bullet}(\usc,f_*\tilde{\CI}^{\bullet}) \cong \SHom_{\CO_{F,X}}^{\bullet}(\usc,f_*\tilde{\CI}^{\bullet}).
\]
This finishes the proof of the commutativity of the second square because the diagram
\[
	\xymatrix{
		f_*\SHom_{f^{-1}\CO_Y}^{\bullet}(\omega_f \otimes_{\CO_X} F^{\bullet},\CI^{\bullet}) \ar[dd] \ar[r] & f_*\SHom_{f^{-1}\CO_{F,Y}}^{\bullet}(\CO_{F,f} \otimes_{\CO_X} F^{\bullet},\tilde{\CI}^{\bullet}) \ar[d] \\
		 & \SHom_{\CO_{F,Y}}^{\bullet}(f_*(\CO_{F,f} \otimes_{\CO_X} F^{\bullet}),f_*\tilde{\CI}^{\bullet}) \ar[d] \\
		\SHom_{\CO_Y}^{\bullet}(f_*(\omega_f \otimes_{\CO_X} F^{\bullet}),f_*\CI^{\bullet}) \ar[r] & \SHom_{\CO_{F,Y}}^{\bullet}(\CO_{F,Y} \otimes_{\CO_Y} f_*(\omega_f \otimes_{\CO_X} F^{\bullet}),f_*\tilde{\CI}^{\bullet})
	}
\]
of natural morphisms commutes. 

The commutativity of the third and the fifth square can be shown similarly. The fourth square commutes by the functoriality of the corresponding horizontal isomorphisms.

For the adjunction of $f_+$ and $R\Gamma_{Z'}f^!$ we proceed as in the proof of \autoref{qcohadjunction}.
\end{proof}
\begin{landscape} \label{bigdiagram}
\small
\begin{align*} 
	\xymatrix@C10pt{
		Rf_*\RSHom_{\CO_X}^{\bullet}(M,R\Gamma_{Z'}f^!\CN^{\bullet}) \ar[r] \ar[d]^{\sim} & Rf_*\RSHom_{\CO_{F,X}}^{\bullet}(\CM,R\Gamma_{Z'}f^!\CN^{\bullet}) \ar[d] \\
		Rf_*\RSHom_{\CO_X}^{\bullet}(\omega_{X/Y} \otimes_{\CO_X} M,\omega_{X/Y} \otimes_{\CO_X} R\Gamma_{Z'}f^!\CN^{\bullet}) \ar[r] \ar[dd] & Rf_*\RSHom_{f^{-1}\CO_{F,Y}}^{\bullet}(\CO_{F,Y \leftarrow X} \derotimes_{\CO_{F,X}} \CM,\CO_{F,Y \leftarrow X} \derotimes_{\CO_{F,X}} R\Gamma_{Z'}f^!\CN^{\bullet}) \ar[d] \\
		& \RSHom_{\CO_{F,Y}}^{\bullet}(Rf_*(\CO_{F,Y \leftarrow X} \derotimes \CM),Rf_*(\CO_{F,Y \leftarrow X} \derotimes R\Gamma_{Z'}f^!\CN^{\bullet})) \ar[d] \\
		\RSHom_{\CO_Y}^{\bullet}(Rf_*(\omega_{X/Y} \otimes M),Rf_*(\omega_{X/Y} \otimes R\Gamma_{Z'}f^!\CN^{\bullet})) \ar[r] \ar[d] & \RSHom_{\CO_{F,Y}}^{\bullet}(\CO_{F,Y} \otimes_{\CO_Y}Rf_*(\omega_{X/Y} \otimes_{\CO_X} M),Rf_*(\CO_{F,Y \leftarrow X} \derotimes R\Gamma_{Z'}f^!\CN^{\bullet})) \ar[d] \\
		\RSHom_{\CO_Y}^{\bullet}(Rf_*(\omega_{X/Y} \derotimes_{\CO_X} M),Rf_*R\Gamma_{Z'}\omega_{X/Y}[d] \derotimes_{\CO_Y} \CN^{\bullet}) \ar[r] \ar[d]^-{\tr} & \RSHom_{\CO_{F,Y}}^{\bullet}(\CO_{F,Y} \otimes_{\CO_Y} Rf_*(\omega_{X/Y} \derotimes_{\CO_X} M),f_+R\Gamma_{Z'}\CO_{F,X}[d] \derotimes_{\CO_{F,Y}} \CN^{\bullet}) \ar[d]^-{\tr} \\
		\RSHom_{\CO_Y}^{\bullet}(Rf_*(\omega_{X/Y} \derotimes_{\CO_X} M),\CO_Y \otimes_{\CO_Y} \CN^{\bullet}) \ar[r] \ar[d]^-{\sim} & \RSHom_{\CO_{F,Y}}^{\bullet}(\CO_{F,Y} \otimes_{\CO_Y} Rf_*(\omega_{X/Y} \derotimes_{\CO_X} M),\CO_{F,Y} \otimes_{\CO_{F,Y}} \CN^{\bullet}) \ar[d]^-{\sim} \\
		\RSHom_{\CO_Y}^{\bullet}(Rf_*(\omega_{X/Y} \derotimes_{\CO_X} M),\CN^{\bullet}) \ar[r] & \RSHom_{\CO_{F,Y}}^{\bullet}(f_+\CM,\CN^{\bullet})
	}
\end{align*}
\end{landscape}

\subsection{Definition of lfgu modules on singular schemes}

As mentioned earlier, for a regular scheme $X$, the Frobenius $F_X\colon X \to X$ is a flat morphism and hence $F_X^*$ is exact (\cite[Theorem 2.1]{Kunz}). For varieties, the exactness of $F_X^*$ plays an important role in the definition of (locally finitely generated) unit $\CO_{F,X}$-modules. For example, it implies that the category of unit $\CO_{F,X}$-modules is abelian. In this section we define the abelian category $\mu_{\lfgu}(X)$ of locally finitely generated unit $\CO_{F,X}$-modules for schemes $X$ which admit a closed immersion $i\colon X \to Y$ into a smooth $k$-scheme as a certain subcategory of $\mu_{\lfgu}(Y)$. Note that this definition generally works for unit $\CO_{F,X}$-modules. We restrict to locally finitely generated modules due to our application to Cartier crystals and perverse constructible \'etale $p$-torsion sheaves. For the rest of this section we assume that the base field $k$ is perfect.

For the motivation of our approach to $\mu_{\lfgu}(X)$ for embeddable $X$, recall the Kashiwara equivalence:
\begin{theorem} \label{Kashiwaraunit}
Let $i\colon Z \to X$ be a closed immersion of smooth $k$-schemes. If $\CM$ is a unit $\CO_{F,X}$-module supported on $Z$, the adjunction $i_+i^!\CM \to \CM$ is an isomorphism. Consequently, $H^0(i^!\CM) \overset{\sim}{\longrightarrow} i^!\CM$ and the functors $i_+$ and $i^!$ are equivalences between the categories of unit $\CO_{F,Z}$-modules and unit $\CO_{F,X}$-modules supported on $Z$.   
\end{theorem}
\begin{proof}
This is Theorem 5.10.1 of \cite{EmKis.Fcrys}.
\end{proof}
Hence, keeping the notation of the preceding theorem, we can canonically interpret unit $\CO_{F,Z}$-modules as a certain subcategory of unit $\CO_{F,X}$-modules, namely the subcategory of unit $\CO_{F,X}$-modules with support on (the image of) $Z$. If $Z$ is not smooth this subcategory still exists because it may be characterized as the subcategory of unit $\CO_{F,X}$-modules $\CM$ with $j^!\CM \cong 0$, where $j$ is the immersion of the open complement of $Z$ in $X$. This motivates the definition of unit $\CO_{F,Z}$-modules for $Z$ possibly not smooth but embeddable into a smooth scheme. But first we introduce some notation.
\begin{definition}
We call a $k$-scheme $X$ \emph{embeddable} if there is a closed immersion $i\colon X \to Y$ of $k$-schemes where $Y$ is smooth. We call $X$ \emph{$F$-finite embeddable} if there is a closed immersion $i\colon X \to Y$ of $k$-schemes where $Y$ is smooth and $F$-finite.
\end{definition}
\begin{example}
Let $X = \Spec k[x_1,\dots,x_n]/I$ be an affine variety. Then $X$ is embeddable into the affine space $\mathbb A_k^n$ by the closed immersion corresponding to the canonical projection
\[
	k[x_1,\dots,x_n] \to k[x_1,\dots,x_n]/I.
\]
\end{example}
\begin{example}
Let $X$ be a quasi-projective $k$-scheme. By definition, there exists an open immersion $j\colon X \to Z$ and a projective morphism $p\colon Z \to \Spec k$ such that $f = p \circ j$. In turn, the morphism $p$ factors into a closed immersion $i\colon Z \to \mathbb P_k^n$ followed by the natural morphism $\mathbb P_k^n \to \Spec k$. Let $U$ be an open subset of $\mathbb P_k^n$ such that $U \cap i(Z) = i(j(X))$. Then $X \cong U \times_{\mathbb P_n^k} Z$ and the projection $X \to U$ is a closed immersion of $X$ into an open subset of the projective space. Thus $X$ is embeddable.
\end{example}         
\begin{definition}  
Assume that $k$ is perfect. Let $X$ be an embeddable $k$-scheme. Let $i\colon X \to Y$ be a closed immersion into a smooth $k$-scheme $Y$. The category of lfgu $\CO_{F,X}$-modules is defined as the full subcategory of lfgu $\CO_{F,Y}$-modules $\CM$ supported on the image of $X$, i.e.\ $j^!\CM \cong 0$, where $j\colon Y \backslash X \to X$ is the open immersion of the complement of $X$.

The category $D_{\lfgu}^b(\CO_{F,X})$ is the full subcategory $D_{\lfgu}^b(\CO_{F,Y})_X$ of those objects in $D_{\lfgu}^b(\CO_{F,Y})$ whose cohomology sheaves are supported on $X$. 
\end{definition}
\begin{remark} \label{reducedclosed}
With the notation of the preceding definition, let $\CM$ be an lfgu module on $Y$. Whether $\CM$ is supported in $X$ only depends on the closed subset $i(X)$ in $Y$. For example, the preceding definition does not distinguish between the categories $D_{\lfgu}^b(\CO_{F,X})$ and $D_{\lfgu}^b(\CO_{F,X_{\operatorname{red}}})$, where $X_{\operatorname{red}}$ is the unique closed subscheme of $X$ whose underlying topological space equals the one of $X$ and which is reduced.   
\end{remark}
By \autoref{Kashiwaraunit}, it is clear that this definition generalizes the already existing notion of lfgu $\CO_{F,X}$-modules for smooth $X$. Of course the crucial point is to see that the definition for not-necessarily smooth $X$ is -- up to natural equivalence -- independent of a chosen embedding into a smooth scheme. 
\begin{theorem} \label{adjunctionequiv}
Assume that $k$ is a perfect field. Let $f\colon X \to Y$ be a flat morphism between smooth $k$-schemes and let $i_X\colon Z \to X$ and $i_Y\colon Z \to Y$ be closed immersions of $k$-schemes such that the diagram
\[
	\xymatrix{
		Z \ar[r]^{i_X} \ar[dr]_{i_Y} & X \ar[d]^f \\
		& Y
	}
\]
commutes. Then there are natural isomorphisms of functors
\begin{enumerate}[(i)]
	\item $f_+ \circ R\Gamma_Z f^! \cong \id_{D_{\lfgu}^b(\CO_{F,Y})_Z}$,
	\item $R\Gamma_Z f^! \circ f_+ \cong \id_{D_{\lfgu}^b(\CO_{F,X})_Z}$.
\end{enumerate}
\end{theorem}
\begin{proof}
The proof proceeds by an excision argument, in a similar way as the proof of \cite[Theorem 4.5]{Ohkawa}. In the case of a smooth scheme $Z$ we can use the isomorphism of functors $R\Gamma_Z \cong i_{X+}i_X^!$ from $D_{\lfgu}^b(\CO_{F,X})$ to $D_{\lfgu}^b(\CO_{F,X})_Z$ and $R\Gamma_Z \cong i_{Y+}i_Y^!$ from $D_{\lfgu}^b(\CO_{F,Y})$ to $D_{\lfgu}^b(\CO_{F,Y})_Z$ (\cite[Proposition 5.11.5]{EmKis.Fcrys}):
\[
	f_+R\Gamma_Zf^! \cong f_+i_{X+}i_X^!f^! \cong i_{Y+}i_Y^! \cong R\Gamma_Z \cong \id.
\]
We may assume that $Z$ is reduced, see \autoref{reducedclosed}. Since a finite set of closed points with the reduced scheme structure is always smooth, this verifies the claim if $Z$ is $0$-dimensional. For the general case, i.e.\ $Z$ is not necessarily smooth, let $V$ be a smooth and dense\footnote{In order to guarantee the existence of a smooth, dense subset, we assumed that $k$ is perfect.} open subscheme of $Z$ and assume that the claim holds for all closed subschemes $Z'$ with $\dim Z' < \dim Z$ . Let $g$ denote the immersion $V \into Y$. After choosing an open subset $U \subseteq Y$ with $U \cap Z = V$, we can factor $g$ as $g = u \circ i'$ where $u$ is the open immersion of $U$ into $Y$ and $i'$ is the closed immersion of $V$ into $U$, i.e.\ the base change of $i_Y$. 

For an object $\CM^{\bullet}$ of $D_{\lfgu}^b(\CO_{F,Y})$, there is a natural morphism $\phi\colon \CM^{\bullet} \to g_+g^!\CM^{\bullet}$ whose cone $\CN^{\bullet}$ is supported on $Z \backslash U$ (\cite[Proposition 5.12.1]{EmKis.Fcrys}). This means that there is a distinguished triangle 
\[
	\CN^{\bullet} \longrightarrow \CM^{\bullet} \overset{\phi}{\longrightarrow} g_+g^!\CM^{\bullet} \longrightarrow \CN^{\bullet}[1]
\]
in $D_{\lfgu}^b(\CO_{F,Y})$. Applying $f_+R\Gamma_Zf^!$, the trace yields a morphism of triangles 
\[
	\xymatrix{
		f_+R\Gamma_Zf^!\CN^{\bullet} \ar[r] \ar[d]^{\tr_f(\CN^{\bullet})} & f_+R\Gamma_Zf^!\CM^{\bullet} \ar[r]^-{\phi} \ar[d]^{\tr_f(\CM^{\bullet})} & f_+R\Gamma_Zf^!g_+g^!\CM^{\bullet} \ar[d]^{\tr_f(g_+g^!\CM^{\bullet})} \\
		\CN^{\bullet} \ar[r] & \CM^{\bullet} \ar[r]^-{\phi} & g_+g^!\CM^{\bullet}.
	}
\]
Since $f$ is the identity on $Z$, i.e.\ $i_Y = f \circ i_X$, we have $Z \cap f^{-1}(Z \backslash U) = Z \backslash U$. Therefore, $R\Gamma_Zf^!\CN^{\bullet} \cong R\Gamma_{Z \backslash U}f^!\CN^{\bullet}$ and $\tr_f(\CN^{\bullet})$ factors through $f_+R\Gamma_{Z \backslash U}f^!\CN^{\bullet}$. This means that the diagram
\[
	\xymatrix{
		f_+R\Gamma_Zf^!\CN^{\bullet} \ar[rr]^-{\tr_{f,Z}(\CN^{\bullet})} \ar[dr]^-{\sim} & & \CN^{\bullet} \\
		& f_+R\Gamma_{Z \backslash U}f^!\CN^{\bullet} \ar[ur]^-{\sim}_/8pt/{\quad \tr_{f,Z \backslash U}(\CN^{\bullet})} &
	}
\]
is commutative. The dimension of the support $Z \backslash U$ of $\CN^{\bullet}$ is less than that of $Z$ as $V$ is dense in $Z$. By induction hypothesis, $\tr_{f,Z \backslash U}(\CN^{\bullet})$ is an isomorphism and hence $\tr_{f,Z}(\CN^{\bullet})$ is an isomorphism.

It remains to show that $\tr_f(g_+g^!\CM^{\bullet}) \cong \tr_f(u_+i'_+i'^!u^!\CM^{\bullet})$ is an isomorphism. By the Kashiwara equivalence, the object $\CM_U^{\bullet} := i'_+i'^!u^!\CM^{\bullet}$ of $D_{\lfgu}^b(\CO_{F,U})$ is supported on $V$. Let $f'$ denote the projection $U \times_Y X \to U$. The map $\tr_f(u_+\CM_U^{\bullet})$ equals the composition
\[
	f_+R\Gamma_Zf^!u_+\CM_U^{\bullet} \overset{\sim}{\longrightarrow} u_+f'_+R\Gamma_Vf'^!\CM_U^{\bullet} \xrightarrow{u_+\tr_{f'}(\CM_U^{\bullet})} u_+\CM_U^{\bullet}
\]
(\autoref{traceforward} (a)). Here the second map is an isomorphism because $V$ is smooth. Consequently, the map $\tr_f(\CM^{\bullet})$ is an isomorphism. This proves (i). The isomorphism of (ii) can be constructed similarly, using the unit of the adjunction between $f_+$ and $R\Gamma_Zf^!$ (i.e.\ the \emph{cotrace}) instead of the trace map, and applying \autoref{traceforward} (b).  
\end{proof}
The next corollary shows that the definition of $D_{\lfgu}^b(\CO_{F,X})$ for embeddable varieties $X$ is independent of the chosen embedding. 
\begin{corollary} \label{embeddingindep}
If $i_1\colon X \to Y_1$ and $i_2\colon X \to Y_2$ are two embeddings of a $k$-scheme $X$ into smooth $k$-schemes $Y_1$ and $Y_2$, where $k$ is perfect, then there exists a natural equivalence 
\[
	D_{\lfgu}^b(\CO_{F,Y_1})_X \overset{\sim}{\longrightarrow} D_{\lfgu}^b(\CO_{F,Y_2})_X.
\]
\end{corollary}
\begin{proof}
The universal property of $Y_1 \times_k Y_2$ yields a morphism $(i_1,i_2)\colon X \to Y_1 \times_k Y_2$. It equals the composition
\[
	X \xrightarrow{(\id,\id)} X \times_k X \xrightarrow{(i_1,\id)} Y_1 \times_k X \xrightarrow{(\id,i_2)} Y_1 \times_k Y_2,
\]
where all maps are closed immersions, the first one because $X$ is assumed to be separated over $k$. Hence $(i_1,i_2)$ is a closed immersion. We obtain a commutative diagram
\[
	\xymatrix@C50pt{
		& Y_1 \\
		X \ar[r]^{(i_1,i_2)} \ar[ur]^-{i_1} \ar[dr]_-{i_2} & Y_1 \times_k Y_2 \ar[u]_{p_1} \ar[d]^{p_2} \\
		& Y_2,
	}
\]
where $p_1$ and $p_2$ are the projections. By \autoref{adjunctionequiv}, the compositions $p_{2+}R\Gamma_Zp_1^!$ and $p_{1+}R\Gamma_Zp_2^!$ are inverse equivalences between $D_{\lfgu}^b(\CO_{F,Y_1})_X$ and $D_{\lfgu}^b(\CO_{F,Y_2})_X$.
\end{proof}

\section{The Riemann-Hilbert correspondence for Cartier crystals}

As its title suggests, one of the main results of Emerton and Kisins ``The Riemann-Hilbert correspondence for unit F-crystals'' (\cite{EmKis.Fcrys}) is a characteristic $p$-analogue of the Riemann-Hilbert correspondence for $D$-modules. More precisely, for a smooth $k$-scheme $X$, the authors construct inverse equivalences of categories 
\[
	\xymatrix{
     D_{\lfgu}^b(\CO_{F,X}) \ar@<.5ex>[r]^-{\Solu} & D_c^b(X_{\et},\ZZ/p\ZZ) \, . \ar@<.5ex>[l]^-{\M}
   }
\]
Furthermore, $\Solu(D_{\lfgu}^{\leq 0}(\CO_{F,X})) \subseteq {^pD}^{\geq 0}$ and $\Solu(D_{\lfgu}^{\geq 0}(\CO_{F,X})) \subseteq {^p}D^{\leq 0}$ where ${^p}D^{\geq 0}$ and ${^p}D^{\geq 0}$ are two subcategories of $D_c^b(X_{\et},\ZZ/p\ZZ)$ defining the perverse $t$-structure of \cite{Gabber.tStruc}. Hence $\Solu$ establishes an equivalence between the hearts of the corresponding $t$-structures, namely the locally finitely generated unit $\CO_{F,X}$-modules and the so-called perverse constructible $p$-torsion sheaves.

Using this correspondence of Emerton and Kisin, we will establish a Riemann-Hilbert correspondence between Cartier crystals and perverse constructible \'etale $\ZZ/p\ZZ$-sheaves on a scheme which admits an embedding into a smooth scheme. In the following two subsections we extend the equivalences $\G\colon D_{\crys}^b(\QCrysC(X)) \to D_{\lfgu}^b(\CO_{F,X})$ and $\Solu\colon D_{\lfgu}^b(\CO_{F,X}) \to D_c^b(X_{\et},\ZZ/p\ZZ)$ to singular varieties embeddable into a smooth variety. Throughout the whole section, $k$ denotes a perfect field of characteristic $p$.

\subsection{Review of Emerton and Kisin's Riemann-Hilbert correspondence}

Let $X_\et$ denote the small \'etale site of a scheme $X$. A reference for the \'etale topology is, for example, \cite[Chapter II]{Milne}. A $\ZZ/p\ZZ$-sheaf on $X_{\et}$ is an \'etale sheaf of modules over the constant sheaf $\ZZ/p\ZZ$. Let $D_c^b(X_{\et},\ZZ/p\ZZ)$ denote the derived category of complexes of $\ZZ/p\ZZ$-sheaves on $X_{\et}$ whose cohomology sheaves are \emph{constructible}.
\begin{definition}
A sheaf $\CL$ of $\ZZ/p\ZZ$-modules on $X_{\et}$ is called \emph{constructible} if there is a stratification $X = \coprod_{i \in I} S_i$ such that the restrictions of $\CL$ to the $S_i$ are locally constant sheaves of $\ZZ/p\ZZ$-modules for the \'etale topology with finite stalks.
\end{definition} 
For $x \in X$, let $i_x\colon x \to X$ be the inclusion, which is the composition of the inclusion of the closed point of $\Spec \CO_{X_{\et},x}$ followed by the canonical morphism $\Spec \CO_{X_{\et},x} \to X$. In \cite{Gabber.tStruc}, Gabber showed that the two subcategories
\begin{align*}
	{^p}D^{\leq 0} &= \{ \CL^{\bullet} \in D_c^b(X_{\et},\ZZ/p\ZZ) \, | \, H^i(i_x^*\CL^{\bullet})=0 \text{ for } i > -\dim \overline{ \{x\}} \}, \\
	{^p}D^{\geq 0} &= \{ \CL^{\bullet} \in D_c^b(X_{\et},\ZZ/p\ZZ) \, | \, H^i(i_x^!\CL^{\bullet})=0 \text{ for } i < -\dim \overline{ \{x\}} \}
\end{align*}
define a $t$-structure on $D_c^b(X_{\et},\ZZ/p\ZZ)$.
\begin{remark} \label{Gabberrecollement}
Indeed, Gabber shows that these subcategories define a $t$-structure on the ambient category $D^b(X_{\et},\ZZ/p\ZZ)$. For a closed immersion $i\colon Z \to X$ and the open immersion $j\colon U \to X$ of the complement $U$ of $Z$, it is obtained from the perverse $t$-structures on $D^b(U_{\et},\ZZ/p\ZZ)$ and $D^b(Z_{\et},\ZZ/p\ZZ)$ by \emph{recollement}: 
\begin{align*}
	{^p}D^{\leq 0} &= \{ \CL^{\bullet} \in D^b(X_{\et},\ZZ/p\ZZ) \, | \, i^*\CL^{\bullet} \in {^p}D^{\leq 0}(Z_{\et},\ZZ/p\ZZ) \text{ and } j^*\CL^{\bullet} \in {^p}D^{\leq 0}(U_{\et},\ZZ/p\ZZ) \}, \\
	{^p}D^{\geq 0} &= \{ \CL^{\bullet} \in D^b(X_{\et},\ZZ/p\ZZ) \, | \, i^!\CL^{\bullet} \in {^p}D^{\geq 0}(Z_{\et},\ZZ/p\ZZ) \text{ and } j^*\CL^{\bullet} \in {^p}D^{\geq 0}(U_{\et},\ZZ/p\ZZ) \}.
\end{align*}
This follows directly from the construction of the perverse $t$-structure on $D^b(X_{\et},\ZZ/p\ZZ)$. 
\end{remark}
In this subsection let $X$ be a smooth $k$-scheme. The Riemann-Hilbert correspondence between $D_{\lfgu}^b(\CO_{F,X})$ and $ D_c^b(X_{\et},\ZZ/p\ZZ)$ is realized in two steps: first passing to the \'etale site and then applying a certain duality functor.
\begin{theorem} \label{Solproperties}
\begin{enumerate}
\item For every smooth $k$-scheme $X$, the functor 
\[
	\Solu = \RSHom_{\CO_{F,X_{\et}}}^{\bullet}(\usc_{\et},\CO_{X_{\et}})[d_X]\colon D_{\lfgu}^b(\CO_{F,X}) \to D_c^b(X_{\et},\FF_p)
\]
is an equivalence of categories. A quasi-inverse is given by 
\[
	\operatorname{M} = \RSHom_{\ZZ/p\ZZ}^{\bullet}(\usc,\CO_{X_{\et}})[d_X].
\]
\item For a morphism $f\colon X \to Y$ of smooth $k$-schemes, there is a natural isomorphism of functors
\[
	\Solu \circ f^! \cong f^* \circ \Solu.
\]
For an \emph{allowable} morphism $f\colon X \to Y$, i.e.\ a morphism $f$ which factors as $g \circ h$, where $h$ is an immersion and $g$ is a proper smooth morphism, there is also a natural isomorphism of functors
\[
	\Solu \circ f_+ \cong f_! \circ \Solu.
\]
\item The essential image of the full subcategory $D_{\lfgu}^{\geq 0}(\CO_{F,X})$ is equal to the full subcategory ${^p}D^{\leq 0}$ of $D_c^b(X_{\et},\ZZ/p\ZZ)$ while the essential image of $D_{\lfgu}^{\leq 0}(\CO_{F,X})$ is equal to the full subcategory ${^p}D^{\geq 0}$ of $D_c^b(X_{\et},\ZZ/p\ZZ)$.
\end{enumerate}
\end{theorem}
\begin{proof}
This is \cite[Theorem 11.4.2 and Theorem 11.5.4]{EmKis.Fcrys}. 
\end{proof}

\subsection{Cartier crystals and lfgu modules on singular schemes}

We show that the equivalence 
\[
	\G\colon D_{\crys}^b(\QCrysC(X)) \overset{\sim}{\longrightarrow} D_{\lfgu}^b(\CO_{F,X})
\]
for smooth $X$ extends to an equivalence for embeddable $X$. As a consequence, for a morphism $f$ between smooth schemes, the inverse equivalences $f_+$ and $R\Gamma_Zf^!$ between the subcategories of complexes supported on a closed subscheme are $t$-exact.   
\begin{proposition} \label{singularG}
Let $X$ be an $F$-finite embeddable $k$-scheme. The functor $\G$ induces an equivalence of categories
\[
	\G\colon D_{\crys}^b(\QCrysC(X)) \to D_{\lfgu}^b(\CO_{F,X}).
\]
\end{proposition}
\begin{proof}
Choose a closed immersion $i\colon X \to Y$ into a smooth, $F$-finite $k$-scheme $Y$ and let $j$ denote the open immersion of the complement of $X$ in $Y$. The Kashiwara equivalence (\autoref{Kashiwara}) identifies $D_{\crys}^b(\QCrysC(X))$ with the subcategory $D_{\crys}^b(\QCrysC(Y))_X$ of $D_{\crys}^b(\QCrysC(Y))$. For $\CM^{\bullet} \in D_{\crys}^b(\QCrysC(Y))_X$ we have
\[
	(j^! \circ \G_Y)\CM^{\bullet} \cong (\G_U \circ j^*)\CM^{\bullet} \cong 0
\]
by \autoref{Gopenresult}. As $\G$ is an equivalence of categories, there is also a natural isomorphism of functors $j^* \circ \G_Y^{-1} \cong \G_U^{-1} \circ j^!$ for the inverse $\G^{-1}$ of $\G$. It follows that $\G$ induces an equivalence of subcategories 
\[
	\G\colon D_{\crys}^b(\QCrysC(Y))_X \to D_{\lfgu}^b(\CO_{F,Y})_X.
\]
It remains to show that this equivalence is independent of the choice of the embedding. In the same way as in the proof of \autoref{embeddingindep} we can reduce to the case of two closed immersions $i_1\colon X \to Y_1$ and $i_2\colon X \to Y_2$ into smooth, $F$-finite $k$-schemes $Y_1$ and $Y_2$ together with a morphism $f\colon Y_1 \to Y_2$ such that $i_2 = f \circ i_1$. The composition $i_{2*} \circ i_1^!$ is a natural equivalence between $D_{\crys}^b(\QCrysC(Y_1))_X$ and $D_{\crys}^b(\QCrysC(Y_2))_X$. Note that 
\[
	i_{2*}i_1^! \CM^{\bullet} \cong Rf_*i_{1*}i_1^! \CM^{\bullet} \cong Rf_*\CM^{\bullet}
\]
for $\CM^{\bullet} \in D_{\crys}^b(\QCrysC(Y_1))_X$. Hence $Rf_*$ is a natural equivalence of categories
\[
	Rf_*\colon D_{\crys}^b(\QCrysC(Y_1))_X \to D_{\crys}^b(\QCrysC(Y_2))_X
\]
which is compatible with $\G$, i.e.
\[
	f_+ \circ \G_{Y_1} \cong \G_{Y_2} \circ Rf_*
\]
by \autoref{GammalfguForward}.
\end{proof}
\begin{remark}
This also implies that $f_+$ provides a natural equivalence 
\[
	f_+\colon D_{\lfgu}^b(\CO_{F,Y_1})_X \to D_{\lfgu}^b(\CO_{F,Y_2})_X
\] 
because $f_+ \cong \G_{Y_2} \circ Rf_* \circ \G_{Y_1}^{-1}$.
\end{remark} 
Keeping the notation of the proof of \autoref{singularG}, the canonical $t$-structure of $D_{\lfgu}^b(\CO_{F,Y})$ obviously induces a $t$-structure on the subcategory $D_{\lfgu}^b(\CO_{F,Y})_X$ defined by the two subcategories
\[
	D_{\lfgu}^b(\CO_{F,Y})_X \cap D_{\lfgu}^{\geq 0}(\CO_{F,Y}) \text{ and } D_{\lfgu}^b(\CO_{F,Y})_X \cap D_{\lfgu}^{\leq 0}(\CO_{F,Y}). 
\]
\begin{corollary} \label{lfguexact}
Let $f\colon Y_1 \to Y_2$ be a morphism between smooth, $F$-finite $k$-schemes. Let $i_1\colon X \to Y_1$ and $i_2\colon X \to Y_2$ be closed immersions such that $i_2 = f \circ i_1$. The equivalence $f_+$ of \autoref{embeddingindep} between $D_{\lfgu}^b(\CO_{F,Y_1})_X$ and $D_{\lfgu}^b(\CO_{F,Y_2})_X$ is $t$-exact for the canonical $t$-structures of both derived categories. In particular, by taking $0$-th cohomology, it gives rise to an equivalence of abelian categories
\[
	\{\mu_{\lfgu}(Y_1)_X\} \overset{\sim}{\longrightarrow} \{\mu_{\lfgu}(Y_2)_X\}.
\]
\end{corollary}
\begin{proof}
The functor $f_+$ is a composition of $t$-exact functors:  
\[
	f_+ \cong \G_{Y_2} \circ Rf_* \circ \G_{Y_1}^{-1} \cong \G_{Y_2} \circ i_{2*} \circ i_1^! \circ \G_{Y_1}^{-1},
\]
where $Rf_*$ denotes the restricted functor $D_{\crys}^b(\QCrysC(Y_1))_X \to D_{\crys}^b(\QCrysC(Y_2))_X$. It is exact because $Rf_* \cong Rf_*i_{1*}i_1^! \cong i_{2*} i_1^!$.
\end{proof}

\subsection{A Riemann-Hilbert correspondence on singular schemes}

Now we extend the Riemann-Hilbert correspondence between lfgu modules and constructible \'etale $\ZZ/p\ZZ$-sheaves to embeddable schemes. The corresponding equivalence of categories 
\[
	D_{\lfgu}^b(\CO_{F,X}) \overset{\sim}{\longrightarrow} D_c^b(X_{\et},\ZZ/p\ZZ)
\]
for embeddable $X$ will be $t$-exact for the canonical $t$-structure on $D_{\lfgu}^b(\CO_{F,X})$ and Gabber's perverse $t$-structure on $D_c^b(X_{\et},\ZZ/p\ZZ)$. Again, for a closed subscheme $Z$ of $X$, let $j\colon U \to X$ denote the open immersion of the complement of $Z$ into $X$. 

Recall that there are distinguished triangles
\[
	j_!j^* \longrightarrow \id \longrightarrow i_*i^* \longrightarrow j_!j^*[1]
\]
and
\[
	i_*i^! \longrightarrow \id \longrightarrow j_*j^* \longrightarrow i_*i^![1]
\]
in $D^+(X_{\et},\ZZ/p\ZZ)$ (\cite[1.4.1.1]{BBD}). Defining $\Gamma_Z\colon D(X_{\et}(\ZZ/p\ZZ) \to D(X_{\et}(\ZZ/p\ZZ)$ as the composition $i_*i^*$ of exact functors we obtain a \emph{fundamental triangle of local cohomology}
\[
	j_!j^* \longrightarrow \id \longrightarrow \Gamma_Z \to j_!j^*[1].
\]
Note that $i_!=i_*$ because $i$ is a closed immersion.
\begin{lemma} \label{SolGamma}
Let $Z$ be a closed subscheme of a smooth $k$-scheme $X$. Then there is a natural isomorphism of functors
\[
	\Solu \circ R\Gamma_Z \overset{\sim}{\longrightarrow} \Gamma_Z \circ \Solu.
\]
\end{lemma}
\begin{proof}
We show that there is a natural isomorphism 
\[
	\M \circ \Gamma_Z \overset{\sim}{\longrightarrow} R\Gamma_Z \circ \M,
\] 
where $\M$ is the quasi-inverse of $\Solu$, see \autoref{Solproperties}. The natural isomorphism $j^! \circ \M_X \cong \M_U \circ j^*$ implies that $\M(\Gamma_Z\CL^{\bullet})$ is supported on $Z$ for every complex $\CL^{\bullet}$. Consequently, the morphism $\M(\Gamma_Z\CL^{\bullet}) \to \M(\CL^{\bullet})$ induced by the natural map $\CL^{\bullet} \to \Gamma_Z\CL^{\bullet}$, which is defined by the fundamental triangle of local cohomology above, factors through $R\Gamma_Z\M(\CL^{\bullet})$. This gives rise to a morphism of distinguished triangles
\[
	\xymatrix{
		\M(\Gamma_Z\CL^{\bullet}) \ar[r] \ar[d] & \M(\CL^{\bullet}) \ar@{=}[d] \ar[r] & M(j_!j^*\CL^{\bullet}) \ar[d]^{\sim} \\
		R\Gamma_Z\M(\CL^{\bullet}) \ar[r] & \M(\CL^{\bullet}) \ar[r] & j_+j^!\M(\CL^{\bullet}),
	}
\]
where the horizontal arrows are the natural morphisms and the second and third vertical arrow is an isomorphism. Hence the vertical arrow on the left is an isomorphism.
\end{proof} 
\begin{lemma}
For a closed subscheme $Z$ of a scheme $X$, Gabber's perverse $t$-structure on $D_c^b(X_{\et},\ZZ/p\ZZ)$ induces a $t$-structure on $D_c^b(X_{\et},\ZZ/p\ZZ)_Z$ given by
\begin{align*}
	{^p}D^{\geq 0}(X_{\et})_Z = D_c^b(X_{\et},\ZZ/p\ZZ)_Z \cap {^p}D^{\geq 0}(X_{\et}), \\
	{^p}D^{\leq 0}(X_{\et})_Z = D_c^b(X_{\et},\ZZ/p\ZZ)_Z \cap {^p}D^{\leq 0}(X_{\et}).
\end{align*}
\end{lemma}
\begin{proof}
We consider the construction of the perverse truncation functor ${^p}\tau_{\leq 0}$ in \cite{Gabber.tStruc} in more detail. It will turn out that ${^p}\tau_{\leq 0}\CL^{\bullet}$ is supported on $Z$ for all $\CL^{\bullet} \in D_c^b(X_{\et},\ZZ/p\ZZ)_Z$. For simplicity we write $\tau_{\leq p}$ for ${^p}\tau_{\leq 0}$ where $p$ is a perversity function, see the first section of \cite{Gabber.tStruc}. For a complex $\CF^{\bullet}$, $C(\CF^{\bullet})$ denotes the total complex of the double complex $C^{\bullet}(\CF^{\bullet})$, where $C^{\bullet}(\CF^n)$ is the Godement resolution of $\CF^n$.   

Let $c = -\dim X$. It is a lower bound for the perversity function $p(x) = -\dim \overline{\{x\}}$. For a complex $\CL^{\bullet}$, $d \geq c$ and $p_d(x)=\min (d,p(x))$, Gabber iteratively constructs a direct system $\tau_{\leq p_d}\CL^{\bullet}$ and defines $\tau_{\leq p}\CL^{\bullet}$ as the direct limit. We start with $p_c = c$ and the usual truncation $\tau_{p_c}\CL^{\bullet}=\tau_{\leq c}\CL^{\bullet}$. Clearly, if $\CL^{\bullet}$ is supported on $Z$, then so is $\tau_{p_c}\CL^{\bullet}$. Now for $\tau_{\leq p_d}\CF^{\bullet}$ of some complex $\CF^{\bullet}$, we construct $\tau_{^\leq p_{d+1}}\CF^{\bullet}$ as a subcomplex of $C(\CF^{\bullet})$. By the construction of the Godement resolution, $C(\CF^{\bullet})$ is supported on $Z$ if $\CF^{\bullet}$ is supported on $Z$. It follows that for every $d \geq c$, the complex $\tau_{\leq p_d}\CL^{\bullet}$ is supported on $Z$ and therefore the direct limit $\tau_{\leq p}\CL^{\bullet}$ is supported on $Z$. 
\end{proof}      
\begin{proposition} \label{constructibleKashiwara}
Let $i\colon Z \to X$ be a closed immersion of schemes. 
\begin{enumerate}
\item The exact functors $i_*$ and $i^*$ are inverse equivalences of categories
\[
	\xymatrix{
		D_c^b(Z_{\et},\ZZ/p\ZZ) \ar@<.5ex>[r]^-{i_*} & D_c^b(X_{\et},\ZZ/p\ZZ)_Z. \ar@<.5ex>[l]^-{i^*}
	}
\]
\item These functors $i_*\colon D_c^b(Z_{\et},\ZZ/p\ZZ) \to D_c^b(X_{\et},\ZZ/p\ZZ)_Z$ and $i^*\colon D_c^b(X_{\et},\ZZ/p\ZZ)_Z \to D_c^b(Z_{\et},\ZZ/p\ZZ)$ are also $t$-exact with respect to the perverse $t$-structures of both categories.
\end{enumerate}
\end{proposition} 
\begin{proof}
(a) This is a formal consequence of the distinguished triangle
\[
	j_!j^* \longrightarrow \id \longrightarrow i_*i^* \longrightarrow j_!j^*[1]
\]
in $D^b(X_{\et},\ZZ/p\ZZ)$ and the fact that the natural map $\id \to i^*i_*$ is always an isomorphism.

(b) It suffices to show that $i_*$ is $t$-exact with respect to the perverse $t$-structures, i.e.\ the essential image of ${^p}D^{\leq 0}(Z_{\et})$ under $i_*$ is contained in ${^p}D^{\leq 0}$ and the essential image of ${^p}D^{\geq 0}(Z_{\et})$ under $i_*$ is contained in ${^p}D^{\geq 0}$. For $x \in Z$ let $\tilde{i}_x$ denote the composition  
\[
	\{x\} \to \Spec \CO_{Z_{\et},x} \to Z
\] of canonical morphisms. For $\CL^{\bullet}$ in ${^p}D^{\leq 0}(Z_{\et})$, i.e.\ $H^n(\tilde{i}_x^*\CL^{\bullet})=0$ for every $x \in Z$ and every $n > -\dim \overline{\{x\}}$, we have 
\begin{align*}
	H^n(i_x^*i_*\CL^{\bullet}) &\cong	H^n(\tilde{i}_x^*i^*i_*\CL^{\bullet}) \\
	&\cong H^n(\tilde{i}_x^*\CL^{\bullet}) \\
	&\cong 0
\end{align*}
for every $x \in Z$ and every $n>-\dim \overline{\{x\}}$. For $x \in U=X \backslash Z$, we even have $i_x^*i_*\CL^{\bullet} \cong 0$ because $j^*i_*\CL^{\bullet} \cong 0$ and hence $\tau^*i_*\CL^{\bullet} \cong 0$, where $\tau\colon \Spec \CO_{X_{\et},x} \to X$ is the natural morphism.

Now let $\CL^{\bullet}$ be in ${^p}D^{\geq 0}(Z_{\et})$, i.e.\ $H^n(\tilde{i}_x^!\CL^{\bullet})=0$ for every $x \in Z$ and every $n<-\dim \overline{\{x\}}$. There is a natural isomorphism of functors
\[
	i^!i_* \cong i^*i_*
\]
given by the composition of the natural isomorphisms $i^!i_* \to \id$ and $\id \to i^*i_*$ of \cite[1.4.1.2]{BBD}. Whence 
\begin{align*}
	H^n(i_x^!i_*\CL^{\bullet}) &\cong H^n(\tilde{i}_x^!i^!i_*\CL^{\bullet}) \\
	&\cong H^n(\tilde{i}_x^!i^*i_*\CL^{\bullet}) \\
	&\cong H^n(\tilde{i}_x^!\CL^{\bullet}) \\
	&\cong 0
\end{align*}
for every $x \in Z$ and every $n<-\dim \overline{\{x\}}$. We have already seen that there is nothing to show for $x \in U$.
\end{proof}
\begin{definition}
For a $k$-scheme $X$ which admits an embedding into a smooth scheme $Y$, we define
\begin{align*}
	D_{\lfgu}^{\geq 0}(\CO_{F,X}) &= D_{\lfgu}^b(\CO_{F,Y})_X \cap D_{\lfgu}^{\geq 0}(\CO_{F,Y}), \\
	D_{\lfgu}^{\leq 0}(\CO_{F,X}) &= D_{\lfgu}^b(\CO_{F,Y})_X \cap D_{\lfgu}^{\leq 0}(\CO_{F,Y}). 
\end{align*}
These subcategories of $D_{\lfgu}^b(\CO_{F,X})$ form a natural $t$-structure. 
\end{definition}
The independence of these subcategories of the embedding into a smooth scheme follows from the fact that for a morphism $f\colon Y_1 \to Y_2$ between two smooth schemes over $k$, together with closed immersions $i_1\colon X \to Y_1$ and $i_2\colon X \to Y_2$ with $i_2 = f \circ i_1$, the equivalence
\begin{align*}
	R\Gamma_Xf^! &\cong \operatorname{M} \circ \Sol \circ R\Gamma_Xf^!\\
	&\cong \operatorname{M} \circ \Gamma_Xf^* \circ \Sol \\
	&\cong \operatorname{M} \circ i_{1*}i_1^*f^*i_{2*}i_2^* \circ \Sol \\
	&\cong \operatorname{M} \circ i_{1*}i_2^* \circ \Sol
\end{align*}
is a composition of $t$-exact functors, where $D_c^b(Y_{1,\et},\ZZ/p\ZZ)$ and $D_c^b(Y_{2,\et},\ZZ/p\ZZ)$ are equipped with the perverse $t$-structures (\autoref{Solproperties} and \autoref{constructibleKashiwara}). Therefore, $R\Gamma_Xf^!$ is $t$-exact.
\begin{theorem} \label{singularSol}
Let $X$ be an embeddable $k$-scheme.
\begin{enumerate}
\item The equivalence $\Solu$ for smooth schemes induces an anti-equivalence of categories 
\[
	\Solu\colon D_{\lfgu}^b(\CO_{F,X}) \to D_c^b(X_{\et},\ZZ/p\ZZ)
\]
\item The essential image of $D_{\lfgu}^{\geq 0}(\CO_{F,X})$ under this equivalence equals ${^p}D^{\leq 0}$ and the essential image of $D_{\lfgu}^{\leq 0}(\CO_{F,X})$ equals ${^p}D^{\geq 0}$. 
\end{enumerate}
\end{theorem}
\begin{proof}
After choosing a closed immersion $i\colon X \to Y$ into a smooth $k$-scheme $Y$, we see that $\Solu$ restricts to an anti-equivalence
\[
	\Solu\colon D_{\lfgu}^b(\CO_{F,Y})_X \to D_c^b(Y_{\et},\ZZ/p\ZZ)_X
\] 
in the same way as in the proof of \autoref{singularG}. For the proof of the independence of the choice of an embedding we again reduce to the situation of two closed immersions $i_1\colon X \to Y_1$ and $i_2\colon X \to Y_2$ together with a morphism $f\colon Y_1 \to Y_2$ such that $i_2=f \circ i_1$. We obtain natural equivalences of categories
\[
	R\Gamma_Xf^!\colon D_{\lfgu}^b(\CO_{F,Y_2})_X \to D_{\lfgu}^b(\CO_{F,Y_1})_X
\]
and
\[
	f^{-1}\Gamma_X\colon D_c^b(Y_{2,\et},\ZZ/p\ZZ)_X \to  D_c^b(Y_{1,\et},\ZZ/p\ZZ)_X,
\]
which are compatible with $\Solu$ because
\[
	\Solu \circ R\Gamma_Xf^! \overset{\sim}{\longrightarrow} \Gamma_X \circ \Solu \circ f^! \overset{\sim}{\longrightarrow} \Gamma_X f^* \circ \Solu
\]
by \autoref{Solproperties} and \autoref{SolGamma}. This proves (a).

The equivalence $\Solu$ not only restricts to an equivalence between $D_{\lfgu}^b(\CO_{F,Y})_X$ and $D_c^b(Y_{\et},\ZZ/p\ZZ)_X$ but also between $D_{\lfgu}^{\geq 0}(\CO_{F,Y})$ and ${^p}D^{\leq 0}(Y_{\et})$ (\autoref{Solproperties}). Therefore, $\Solu$ induces an equivalence
\[
	D_{\lfgu}^b(\CO_{F,Y})_X \cap D_{\lfgu}^{\geq 0}(\CO_{F,Y}) \overset{\sim}{\longrightarrow} D_c^b(Y_{\et},\ZZ/p\ZZ)_X \cap {^p}D^{\leq 0}(Y_{\et}).
\]
By \autoref{constructibleKashiwara}, $D_c^b(Y_{\et},\ZZ/p\ZZ)_X \cap {^p}D^{\leq 0}(Y_{\et})$ is canonically equivalent to ${^p}D^{\leq 0}(X_{\et})$. 

Analogously, one can show that the essential image of $D_{\lfgu}^{\leq 0}(\CO_{F,X})$ equals ${^p}D^{\geq 0}$.
\end{proof}
\begin{theorem} \label{RHCorrespondenceCrystals}
Let $X$ be an $F$-finite embeddable $k$-scheme.
\begin{enumerate}
\item The equivalences $\G$ and $\Solu$ for smooth schemes induce equivalences 
\[
	\G\colon D_{\crys}^b(\QCrysC(X)) \to D_{\lfgu}^b(\CO_{F,X}) \text{ and } \Solu\colon D_{\lfgu}^b(\CO_{F,X}) \to D_c^b(X_{\et},\ZZ/p\ZZ).
\]
This means that for a closed immersion $i\colon X \to Y$ with $Y$ smooth, there is a commutative diagram
\[
	\xymatrix{
		& D_c^b(Y_{\et},\ZZ/p\ZZ) & D_c^b(X_{\et},\ZZ/p\ZZ) \ar@_{(->}[l] \\
		D_{\crys}^b(\QCrysC(Y)) \ar[r]^-{\G_Y} & D_{\lfgu}^b(\CO_{F,Y}) \ar[u]^{\Solu_Y} & \\
		D_{\crys}^b(\QCrysC(X)) \ar@^{(->}[u] \ar[rr]^-{\G_X} & & D_{\lfgu}^b(\CO_{F,X}). \ar@_{(->}[ul] \ar[uu]^{\Solu_X}
	}
\]
\item Let $\Sol$ be the composition $\Solu \circ \G$. The essential image of $D_{\crys}^{\geq 0}(\QCrysC(X))$ under $\Sol$ equals the subcategory ${^p}D^{\leq 0}$, while the essential image of $D_{\crys}^{\leq 0}(\QCrysC(X))$ equals the subcategory ${^p}D^{\geq 0}$.
\item If $h\colon W \to X$ is an open or a closed immersion, then there are natural isomorphisms of functors
\[
	\Sol \circ Rh_* \overset{\sim}{\longrightarrow} h_! \circ \Sol \text{ and } \Sol \circ h^! \overset{\sim}{\longrightarrow} h^* \circ \Sol
\]
where $h^!$ denotes the functor $h^*$ if $h$ is an open immersion and $Rh_* = h_*$ for a closed immersion. 
\end{enumerate}
\end{theorem}
\begin{proof}
(a) and (b) follow from \autoref{singularG} and \autoref{singularSol}. 

It remains to prove (c). Let $h\colon W \to X$ be an open immersion. We will construct the natural transformations by choosing embeddings of $X$. Since we have to make sure that this construction is independent of the embedding, we will consider two closed immersions $i_1\colon X \to Y_1$ and $i_2\colon X \to Y_2$ of $X$ into $F$-finite, smooth $k$-schemes ab initio. We may assume that there is a morphism $f\colon Y_1 \to Y_2$ with $i_2 = f \circ i_1$, see the proof of \autoref{embeddingindep}. Let $h_2'\colon V_2 \to Y_2$ be an open immersion such that $(i_2 \circ h)(W) = h_2'(V_2) \cap i_2(X)$ and hence $W \cong X \times_{Y_2} V_2$. Let $i_2'\colon W \to V_2$ be the closed immersion induced by $i_2$, i.e.\ the projection $X \times_{Y_2} V_2 \to V_2$. We have a natural equivalence $Rh_{2*}' i_{2*}' \cong i_{2*} Rh_*$ of functors from $D_{\crys}^b(\QCrysC(W))$ to $D_{\crys}^b(\QCrysC(Y_2))_X$. Moreover, we have a natural equivalence $h_{2+}' \G_{V_2} \cong \G_{Y_2} Rh_{2*}'$ of functors from $D_{\crys}^b(\QCrysC(V_2))_W$ to $D_{\lfgu}^b(\CO_{F,Y_2})_X$. Therefore, $Rh_*\colon D_{\crys}^b(\QCrysC(W)) \to D_{\crys}^b(\QCrysC(X))$ induces a functor $h_+\colon D_{\lfgu}^b(\CO_{F,W}) \to D_{\lfgu}^b(\CO_{F,X})$ such that $\G_X h_* \cong h_+ \G_W$. Note that this functor $h_+$ does not depend on the choice of $V_2$. In particular, to show the independence of the embedding of $X$, we may choose an open immersion $h_1'\colon V_1 \to Y_1$ with $(i_1 \circ h)(W) = h_1'(V_1) \cap i_1(X)$ and such that $(f \circ h_1')(V_1) \subseteq h_2'(V_2)$. Here we set $V_1 = f^{-1}(V_2)$, which means $V_1 \cong Y_1 \times_{Y_2} V_2$. Let $f'\colon V_1 \to V_2$ be the projection and let $i_1'$ be the projection of $W \cong  X \times_{Y_2} V_2 \cong X \times_{Y_1} (Y_1 \times_{Y_2} V_2)$ to $V_1 \cong Y_1 \times_{Y_2} V_2$. It is a closed immersion because it is the base change of the morphism $i_1$. We obtain the following commutative diagram: 
\[
	\xymatrix{
		W \ar[d]^-h \ar[r]^-{i_1'}  & V_1 \ar[d]^-{h_1'} \ar[r]^-{f'} & V_2 \ar[d]^-{h_2'} \\
		X \ar[r]^-{i_1} & Y_1 \ar[r]^f & Y_2.
	}
\]
The following cube demonstrates the natural equivalences, which we have by \autoref{GammalfguForward}:
\[
	\xymatrix@C15pt{
		&  D_{\crys}^b(\QCrysC(V_2))_W \ar[rr]^-{\G_{V_2}} \ar[dd]^/20pt/{Rh_{2*}'}  &  &  D_{\lfgu}^b(\CO_{F,V_2})_W \ar[dd]^-{h_{2+}'}  \\
		D_{\crys}^b(\QCrysC(V_1))_W \ar[rr]^/40pt/{\G_{V_1}} \ar[dd]^-{Rh_{1*}'} \ar[ru]^-{Rf'_*}  &  &  D_{\lfgu}^b(\CO_{F,V_1})_W \ar[dd]_/20pt/{h_{1+}'} \ar[ru]^-{f'_+}  &  \\ 
		&   D_{\crys}^b(\QCrysC(Y_2))_X \ar[rr]^/40pt/{\G_{Y_2}}  &  &  D_{\lfgu}^b(\CO_{F,Y_2})_X  \\
		D_{\crys}^b(\QCrysC(Y_1))_X \ar[rr]^-{\G_{Y_1}} \ar[ru]^-{Rf_*}  &  &  D_{\lfgu}^b(\CO_{F,Y_1})_X \ar[ru]^-{f_+}  &
  }
\]
Here every cube face indicates a natural equivalence, for example, the front refers to the isomorphism of functors $\G_{Y_1} \circ Rh_{1*}' \cong h_{1+}' \circ \G_{V_1}$. This shows the independence of the isomorphism of functors $\G_X h_* \cong h_+ \G_W$ from the chosen embedding of $X$. By adjunction, we obtain a canonical isomorphism $\G_W h^* \cong h^! G_X$ as well. 

Similarly, one shows that we have a natural isomorphism $\Solu h^! \cong h^* \Solu$, using the fact that $\Solu$ commutes with the local cohomology functors (\autoref{SolGamma}) and with pull-backs for morphisms between smooth schemes (\autoref{Solproperties}). Again, by adjunction, we obtain a natural isomorphism $\Solu h_+ \cong h_! \Solu$. Composing these isomorphisms of functors yields the desired one:
\[
	\Sol \circ Rh_* \cong \Solu \circ \G_X \circ Rh_* \cong \Solu \circ h_+ \circ \G_W \cong h_! \circ \Solu \circ \G_W \cong h_! \circ \Sol
\]
and analogously for $\Sol \circ h^* \cong h^* \circ \Sol$. 

If $h\colon W \to X$ is a closed immersion, we proceed similarly, but the proof is simpler because a closed immersion $i\colon X \to Y$ of $X$ into a smooth, $F$-finite $k$-scheme $Y$ yields a closed immersion of $W$ into $Y$ by composing $h$ and $i$.
\end{proof}
\begin{definition}
The abelian category $\Perv_c(X_{\et},\ZZ/p\ZZ)$ of \emph{perverse constructible \'etale $p$-torsion sheaves} is the heart ${^p}D^{\leq 0} \cap {^p}D^{\geq 0}$ of the perverse $t$-structure on $D_c^b(X_{\et},\ZZ/p\ZZ)$.
\end{definition}
\begin{corollary}
For an $F$-finite embeddable $k$-scheme $X$, the functor $\Sol$ induces an anti-equivalence
\[
	\CrysC(X) \to \Perv_c(X_{\et},\ZZ/p\ZZ)
\]
between the abelian categories of Cartier crystals on $X$ and perverse constructible $\ZZ/p\ZZ$-sheaves on $X_ {\et}$. 
\end{corollary}
We conclude this section with an explanation, how resolution of singularities would simplify the establishment of lfgu modules on singular schemes. If one had resolution of singularities in positive characteristic, then any morphism $f\colon X \to Y$ between smooth $k$-schemes would admit a factorization into an open immersion $X \to \LX$ into a \emph{smooth} scheme $\LX$ followed by a proper morphism $\overline{f}\colon \LX \to Y$. 

One possibility would be to define $D_{\lfgu}^b(\CO_{F,Z})$ for embeddable schemes $Z$ in a slightly different way. Choose a closed immersion of $Z$ into a smooth $k$-scheme $X$ and an \emph{open} immersion $j\colon X \to \LX$ into a smooth complete variety, this means a smooth scheme $\LX$ whose structural morphism $\LX \to \Spec k$ is proper. This is possible by Nagata's compactification theorem applied to the structural morphism of the $k$-scheme $Z$ and, of course, the assumption of resolution of singularities. Let $\tilde{Z}$ be a closed subset of $\LX$ such that $\tilde{Z} \cap j(X) = Z$. Then define $D_{\lfgu}^b(\CO_{F,Z})$ as the subcategory of $D_{\lfgu}^b(\CO_{F,\LX})$ given by those complexes $\CM^{\bullet}$ that are supported in $\tilde{Z}$ and with $R\Gamma_{\LX \backslash X} \CM^{\bullet} \cong 0$. For the independence of the chosen immersion we could proceed as in section 4 of \cite{Ohkawa}, where the adjunction between $\overline{f}_+$ and $\overline{f}^!$ for proper morphisms of \cite[Theorem 4.4.1]{EmKis.Fcrys} is used. 

Another way to prove \autoref{adjunctionequiv} would be to use the equivalence $\Solu$ and work with the derived categories of constructible sheaves. First we show that (under the assumption of resolution of singularities) every morphism $f\colon X \to Y$ between smooth schemes is allowable, see \autoref{Solproperties} for the definition of allowable. Choose a compactification $X \xrightarrow{j} \LX \xrightarrow{\overline{s}_X} \Spec k$ of the structural morphism $s_X\colon X \to \Spec k$ with a smooth scheme $\LX$. Then $f$ factors through $(j,f)\colon X  \to \LX \times_k Y$ and the projection $\pr_Y\colon\LX \times_k Y \to Y$. Furthermore, $(j,f)$ is an immersion because it is the composition of the closed immersion $(\id,f)\colon X \to X \times_k Y$ (i.e.\ the graph embedding) and the open immersion $j \times \id\colon X \times_k Y \to \LX \times_k Y$. The projection $\pr_Y$ is a proper smooth morphism since these properties are stable under base change and the morphism $\overline{s}_X$ is proper and smooth. Thus $f$ is allowable. 

Let $i_X\colon Z \to X$ and $i_Y\colon Z \to Y$ be closed immersions with $i_Y = f \circ i_X$. Roughly speaking, the functors $f_+$ and $R\Gamma_Z f^!$ are inverse equivalences between $D_{\lfgu}^b(\CO_{F,X})_Z$ and $D_{\lfgu}^b(\CO_{F,Y})_Z$ because the corresponding functors $f_!$ and $\Gamma_Zf^*$ are inverse equivalences between $D_c^b(X_{\et},\ZZ/p\ZZ)_Z$ and $D_c^b(Y_{\et},\ZZ/p\ZZ)_Z$. The following diagram illustrates the categories and functors that we work with. 
\[
	\xymatrix{
		D_{\lfgu}^b(\CO_{F,X})_Z \ar@<.5ex>[rrr]^{\Solu} \ar@<-0.5ex>[dd]_-{f_+} & & & D_c^b(X_{\et},\ZZ/p\ZZ)_Z \ar@<.5ex>[lll]^-{M} \ar@<.5ex>[dr]^-{i_X^*} \ar@<-0.5ex>[dd]_-{f_!} & \\
		& & & & D_c^b(Z_{\et},\ZZ/p\ZZ) \ar@<.5ex>[ul]^-{i_{X*}} \ar@<.5ex>[dl]^-{i_{Y*}} \\
		D_{\lfgu}^b(\CO_{F,Y})_Z \ar@<-0.5ex>[uu]_-{R\Gamma_Zf^!} \ar@<.5ex>[rrr]^{\Solu} & & & D_c^b(Y_{\et},\ZZ/p\ZZ)_Z \ar@<.5ex>[lll]^-{M} \ar@<-0.5ex>[uu]_-{\Gamma_Zf^*} \ar@<.5ex>[ur]^-{i_Y^*} &
	}
\]
We have a natural equivalence $f_+R\Gamma_Zf^! \to \id$ of endofunctors on $D_{\lfgu}^b(\CO_{F,Y})_Z$ by the composition of the natural equivalences
\begin{align*}
	f_+R\Gamma_Zf^! &\overset{\sim}{\longrightarrow} M \circ \Solu \circ f_+R\Gamma_Zf^! \\
	&\overset{\sim}{\longrightarrow} M \circ f_!\Gamma_Zf^* \circ \Solu \\
	&\overset{\sim}{\longrightarrow} M \circ f_!i_{X*}i_X^*f^* \circ \Solu \\
	&\overset{\sim}{\longrightarrow} M \circ i_{Y*}i_Y^* \circ \Solu \\
	&\overset{\sim}{\longrightarrow} M \circ \id_{D_c^b(Y_{\et},\ZZ/p\ZZ)_Z} \circ \Solu \\
	&\overset{\sim}{\longrightarrow} \id_{D_{\lfgu}^b(\CO_{F,Y})_Z},
\end{align*}
where the second one is obtained from the isomorphisms $\Solu f_+ \overset{\sim}{\longrightarrow} f_! \Solu$ and $\Solu f^! \overset{\sim}{\longrightarrow} f^* \Solu$ of \autoref{Solproperties} and the isomorphism $\Solu R\Gamma_Z \overset{\sim}{\longrightarrow} \Gamma_Z \Solu$ of \autoref{SolGamma}.
 
Moreover, there is a natural equivalence $\id \to R\Gamma_Zf^!f_+$ of endofunctors on $D_{\lfgu}^b(\CO_{F,X})_Z$ by the composition of the natural equivalences
\begin{align*}
	\id_{D_{\lfgu}^b(\CO_{F,X})_Z} &\overset{\sim}{\longrightarrow} M \circ \id_{D_c^b(X_{\et},\ZZ/p\ZZ)_Z} \circ \Solu \\
	&\overset{\sim}{\longrightarrow} M \circ i_{X*}i_X^* \circ \Solu \\
	&\overset{\sim}{\longrightarrow} M \circ i_{X*} i_Y^*i_{Y*}i_X^* \circ \Solu \\
	&\overset{\sim}{\longrightarrow} M \circ i_{X*}i_Y^*f_!i_{X*}i_X^* \circ \Solu \\
	&\overset{\sim}{\longrightarrow} M \circ i_{X*}i_Y^*f_! \circ \Solu \\
	&\overset{\sim}{\longrightarrow} M \circ i_{X*}i_X^*f^*f_! \circ \Solu \\
	&\overset{\sim}{\longrightarrow} M \circ \Gamma_Z f^*f_! \circ \Solu \\
	&\overset{\sim}{\longrightarrow} M \circ \Solu \circ R\Gamma_Z f^!f_+ \\
	&\overset{\sim}{\longrightarrow} R\Gamma_Z f^!f_+,
\end{align*} 
where the last but one isomorphism is obtained from the inverses of the isomorphisms $\Solu f_+ \overset{\sim}{\longrightarrow} f_! \Solu$ and $\Solu f^! \overset{\sim}{\longrightarrow} f^* \Solu$ of \autoref{Solproperties} and $\Solu R\Gamma_Z \overset{\sim}{\longrightarrow} \Gamma_Z \Solu$ of \autoref{SolGamma}.

\section{Intermediate extensions}

In this last section we define the intermediate extension for Cartier crystals. First we recall the general definition of the intermediate extension $j\me$ for an open immersion $j$ in \cite{BBD}. This definition uses the extension by zero functor $j_!$, which we do not have defined for Cartier crystals. However, there are characterizations which also make sense for the categories of Cartier crystals and which will be used for our definition of the intermediate extension in the context of these categories. In the last part of the section we show that the equivalence between Cartier crystals and perverse constructible \'etale $p$-torsion sheaves established in the previous section is compatible with the intermediate extension functors of both categories.

\subsection{Review of the intermediate extension}

Suppose that $D$, $D_U$ and $D_Z$ are triangulated categories and $i_*\colon D_Z \to D$ and $j^*\colon D \to D_U$ are exact functors fulfilling the following axioms:
\begin{enumerate}[(1)]
	\item The functor $i_*$ admits a left adjoint $i^*$ and a right adjoint $i^!$. Likewise $j^*$ admits a left adjoint $j_!$ and a right adjoint $j_*$. 
	\item We have $j^*i_*=0$. By adjunction, this is equivalent to $i^*j_!=0$ or equivalent to $i^!j_*=0$.
	\item The functors $i_*$, $j_*$ and $j_!$ are full and faithful. Hence the adjunction morphisms $i^*i_* \to \id \to i^!i_*$ and $j^*j_* \to \id \to j^*j_!$ are isomorphisms.  
	\item For every $F \in D$, there are distinguished triangles $i_*i^!F \longrightarrow F \longrightarrow j_*j^*F \longrightarrow i_*i^!F[1]$ and $j_!j^*F \longrightarrow F \longrightarrow i_*i^*F \longrightarrow j_!j^*F[1]$.
\end{enumerate}
Furthermore, let $(D_Z^{\leq 0},D_Z^{\geq 0})$ and $(D_U^{\leq 0},D_U^{\geq 0})$ be $t$-structures on $D_Z$ and $D_U$. Then the subcategories
\begin{align*}
	D^{\leq 0} &= \{ F \in D \ | \ i^*F \in D_Z^{\leq 0} \text{ and } j^*F \in D_U^{\leq 0}\}, \\
	D^{\geq 0} &= \{ F \in D \ | \ i^!F \in D_Z^{\geq 0} \text{ and } j^*F \in D_U^{\geq 0}\}
\end{align*}
define a $t$-structure on $D$ (\cite[Th\'eor\`eme 1.4.10]{BBD}). We say that this $t$-structure on $D$ is obtained from the $t$-structures on $D_Z$ and $D_U$ by \emph{recollement}. The functors $j^*$ and $i_*$ remain exact with respect to this $t$-structure. Let $C$, $C_Z$ and $C_U$ be the hearts of this $t$-structures on $D$, $D_Z$ and $D_U$, i.e.\ for example $C = D^{\leq 0} \cap D^{\geq 0}$. By Th\'eor\`eme 1.3.6 of ibid., we know that the heart of a $t$-structure is an abelian category. 

Recall that for any triangulated category $\mathcal T$ which is equipped with a $t$-structure, the inclusion $\mathcal T^{\leq 0} \to \mathcal T$ admits a right adjoint $\tau_{\leq 0}$ and the inclusion $\mathcal T^{\geq 0} \to \mathcal T$ admits a left adjoint $\tau_{\geq 0}$ (\cite[Proposition 1.3.3]{BBD}). Let $H^0\colon \mathcal T \to \mathcal T^{\leq 0} \cap \mathcal T^{\geq 0}$ be the composition $\tau_{\geq 0} \tau_{\leq 0}$ as in Th\'eor\`eme 1.3.6 of ibid. If $T$ is one of the functors $j_!$, $j^*$, $j_*$, $i^*$, $i_*$, $i^!$, we write ${^p}T$ for the composition $H^0 \circ T \circ \epsilon$, where $\epsilon$ denotes the respective inclusion $C \to D$, $C_Z \to D_Z$ or $C_U \to D_U$. By \cite[Proposition 1.4.16 (ii)]{BBD}, there are the same adjunction relations for the functors ${^p}T$ as there are for the original functors $T$. A functor $T\colon \CA \to \CB$ between triangulated categories with $t$-structures is called \emph{left $t$-exact} if it restricts to a functor $\CA^{\geq 0} \to \CB^{\geq 0}$, \emph{right $t$-exact} if it restricts to a functor $\CA^{\leq 0} \to \CB^{\leq 0}$ and \emph{$t$-exact} if it is left and right $t$-exact. We have the following
\begin{proposition} \cite[Proposition 1.4.16, Proposition 1.4.17]{BBD} \label{propertiesRecollement}
The functors $j_!$ and $i^*$ are right $t$-exact, the functors $j^*$ and $i_*$ are $t$-exact and the functors $j_*$ and $i^!$ are left $t$-exact. The compositions ${^p}j^*{^p}i_*$, ${^p}i^*{^p}j_!$ and ${^p}i^!{^p}j_*$ are isomorphic to the zero functor. Furthermore, the functors ${^p}i_*$, ${^p}j_!$ and ${^p}j_*$ are full and faithful, the adjunction morphisms ${^p}i^*{^p}i_* \to \id \to {^p}i^!{^p}i_*$ and ${^p}j^*{^p}j_* \to \id \to {^p}j^*{^p}j_!$ are isomorphisms. 
\end{proposition}
The functor $i_*$ identifies $D_Z$ with a thick subcategory of $D$ and $j^*$ identifies the quotient category $D/D_Z$ with $D_U$ (\cite[Remarque 1.4.8]{BBD}). Furthermore, ${^p}i_*$ identifies $C_Z$ with a thick subcategory of $C$, which has the following characterization: An object $M$ of $C$ is contained in the essential image of ${^p}i_*$ if and only if ${^p}j^*M=0$ (\cite[Amplification 1.4.17.1]{BBD}). Following the situation of sheaves, we will say that an object $M$ of $C$ is \emph{supported on $Z$} if and only if ${^p}j^*M=0$.
   
An \emph{extension} of an object $B \in D_U$ to $D$ is an object $A$ of $D$ together with an isomorphism $j^*A \simeq B$. Let $a\colon {^p}j^*{^p}j_* \to \id$ and $\tilde{b}\colon {^p}j_!{^p}j^* \to \id$ denote the adjunction morphisms. The composition 
\[
	{^p}j_! \xrightarrow{{^p}j_!a^{-1}} {^p}j_!{^p}j^*{^p}j_* \xrightarrow{\tilde{b} \, {^p}j_*} {^p}j_*
\]
yields a natural transformation $\eta\colon {^p}j_! \to {^p}j_*$.
\begin{definition} \cite[Definition 1.4.22]{BBD}
For an object $B \in C_U$, the intermediate extension $j\me B$ is the image of the natural morphism $\eta(B)\colon {^p}j_!B \to {^p}j_*B$.
\end{definition} 
Let $f\colon A \to B$ be a morphism of objects in $C_U$. The natural transformation $\eta\colon {^p}j_! \to {^p}j_*$ yields a commutative diagram
\[
	\xymatrix{
		{^p}j_!A \ar[r]^-{{^p}j_!f} \ar[d]^-{\eta(A)} & {^p}j_!B \ar[d]^-{\eta(B)} \\
		{^p}j_*A \ar[r]^-{{^p}j_*f} & {^p}j_*B.
	}
\]
It follows that $j\me$ is a functor between the abelian categories $C_U$ and $C$.  
\begin{lemma} \label{restricteta}
The natural transformation ${^p}j^*\eta$ is an isomorphism of functors. (In the situation of derived categories of sheaves on a ringed space $X$ with an open immersion $j\colon U \into X$, this means that the restriction of $\eta$ to $U$ is an isomorphism of functors.) 
\end{lemma}
\begin{proof}
As the adjunction morphisms $a\colon {^p}j^*{^p}j_* \to \id$ and $b\colon \id \to {^p}j^*{^p}j_!$ are isomorphisms, it suffices to show that the diagram
\[
	\xymatrix@C50pt{
		\id \ar[d]^b & {^p}j^*{^p}j_* \ar[l]_-a \ar[d]^{b({^p}j^*{^p}j_*)} \ar[r]^-a & \id \\
		{^p}j^*{^p}j_! & {^p}j^*{^p}j_!{^p}j^*{^p}j_* \ar[l]_-{({^p}j^*{^p}j_!)a} \ar[r]^-{{^p}j^*\tilde{b} \, {^p}j_*} & {^p}j^*{^p}j_* \ar[u]^a
	}
\]
commutes. Here $\tilde{b}$ denotes the adjunction map ${^p}j_!{^p}j^* \to \id$. The commutativity of the left hand square follows from the naturality of $b$. For the second square, we note that ${^p}j^*\tilde{b} \, {^p}j_* \circ b({^p}j^*{^p}j_*) \cong \id$, because for an adjunction, it is required that the composition
\[
	{^p}j^* \xrightarrow{b \, {^p}j^*} {^p}j^*{^p}j_!{^p}j^* \xrightarrow{{^p}j^*\tilde{b}} {^p}j^*
\]
is the identity.
\end{proof}
\begin{lemma} \label{subobquot}
Let $B \in D_U$. There are no non-trivial subobjects of ${^p}j_*B$ supported on $Z$ and no non-trivial quotients of ${^p}j_!B$ supported on $Z$.
\end{lemma}
\begin{proof} 
Let $\beta\colon A \to {^p}j_*B$ be a monomorphism with $A$ supported on $Z$. Applying the left exact functor ${^p}i^!$ and then the exact functor ${^p}i_*$ yields a monomorphism
\[
	{^p}i_*{^p}i^!\beta\colon i_*{^p}i^!A \to {^p}i_*{^p}i^!{^p}j_*B. 
\]
But ${^p}i_*{^p}i^!A$ is isomorphic to $A$ because $A$ is in the essential image of ${^p}i_*$. Moreover, ${^p}i_*{^p}i^!{^p}j_*B \cong 0$ because ${^p}i^!{^p}j_* \cong 0$ (\autoref{propertiesRecollement}). It follows that $A \cong 0$.

For an epimorphism ${^p}j_!B \to C$ with $C$ supported on $Z$, we apply ${^p}i_*{^p}i^*$ and obtain an exact sequence
\[
	{^p}i_*{^p}i^*{^p}j_!B \longrightarrow {^p}i_*{^p}i^*C \longrightarrow 0.
\]
Since ${^p}i^*{^p}j_! \cong 0$ (\autoref{propertiesRecollement}), we have $C \cong {^p}i_*{^p}i^*C \cong 0$.
\end{proof}
\begin{lemma} \label{extension}
An object $A \in C$ is an extension of an object $B \in C_U$ to $D$ if and only if there are morphisms $f\colon {^p}j_!B \to A$ and $g\colon A \to {^p}j_*B$ making the diagram
\[
	\xymatrix{
		{^p}j_!B \ar[dr]_-f \ar[rr]^-{\eta} & & {^p}j_*B \\
		& A \ar[ur]_-g & 
	}
\]
commutative and ${^p}j^*f$ is epic or ${^p}j^*g$ is monic. 
\end{lemma}
\begin{proof}
Let $A$ be an extension of $B$, i.e.\ there is an isomorphism $\tau\colon {^p}j^*A \to B$. With the notation from \autoref{restricteta}, we consider the following diagram: 
\[
	\xymatrix@C50pt{
		{^p}j_!B \ar[r]^-{{^p}j_!a^{-1}} & {^p}j_!{^p}j^*{^p}j_*B \ar[r]^-{\tilde{b} \, {^p}j_*} & {^p}j_*B \\
		{^p}j_!{^p}j^*A \ar[u]^-{\tau} \ar[r]^-{{^p}j_!a^{-1} \, {^p}j^*} \ar@/_3mm/[dr]_{\tilde{b}} & {^p}j_!{^p}j^*{^p}j_*{^p}j^*A \ar[u]^-{\tau} \ar[r]^-{\tilde{b} \, {^p}j_*{^p}j^*} & {^p}j_*{^p}j^*A \ar[u]^-{\tau} \\
		& A. \ar@/_3mm/[ur]_{\tilde{a}} & 
	}
\]
Here $\tilde{a}\colon \id \to {^p}j_*{^p}j^*$ denotes the unit of adjunction. The three squares at the top commute by functoriality. We show that the lower part of the diagram is commutative. Then we may define $f$ as the composition of ${^p}j_!\tau^{-1}$ and the adjunction morphism ${^p}j_!{^p}j^*A \to A$ and $g$ as the composition of the adjunction morphism $A \to {^p}j_*{^p}j^*A$ and ${^p}j_*\tau$. We have a commutative square 
\[
	\xymatrix@C40pt{
		{^p}j_!{^p}j^*{^p}j_*{^p}j^*A \ar[r]^-{\tilde{b} \, {^p}j_*{^p}j^*} & {^p}j_*{^p}j^*A \\
		{^p}j_!{^p}j^*A \ar[u]_-{{^p}j_!{^p}j^* \tilde{a}} \ar[r]^-{\tilde{b}} & A. \ar[u]_-{\tilde{a}}
	}
\]
Therefore, it suffices to show that the two morphisms ${^p}j_!{^p}j^*\tilde{a}$ and ${^p}j_! a \, {^p}j^*$ are equal. This is nothing but the requirement that the composition
\[
	{^p}j^* \xrightarrow{{^p}j^*\tilde{a}} {^p}j^*{^p}j_*{^p}j^* \xrightarrow{a \, {^p}j^*} {^p}j^*
\]
of the adjunction morphisms is the identity. Another consequence of this requirement is the commutativity of the diagram
\[
	\xymatrix@C60pt{
		{^p}j^*A \ar[r]^-{b \, {^p}j^*}_-{\sim} \ar[dr]_-{\id} & {^p}j^*{^p}j_!{^p}j^*A \ar[d]^-{{^p}j^*\tilde{b}} \\
		& {^p}j^*A.
	}
\]
Hence the adjunction morphism ${^p}j_!{^p}j^* \to \id$ applied to ${^p}j^*A$ and thus, also ${^p}j^*f$ is an isomorphism. It follows that ${^p}j^*g$ is an isomorphism since ${^p}j^*\eta(B)$ is an isomorphism (\autoref{restricteta}) and ${^p}j^*\eta(B) = {^p}j^*g \circ {^p}j^*f$.
	
Conversely, let $f\colon {^p}j_!B \to A$ and $g\colon A \to {^p}j_*B$ be morphisms with $g \circ f = \eta(B)$. Recall that the adjunction morphisms $a$ and $b$ are isomorphisms. By a similar reasoning as above, we see that the diagram
\[
	\xymatrix@C50pt{
		B \ar[r]_-{\sim}^-{a^{-1}} \ar[d]_-b^-{\sim} \ar@/^6mm/[rr]^{\id} & {^p}j^*{^p}j_*B \ar[r]_-{\sim}^-a \ar[d]_-{\sim}^-{b({^p}j^*{^p}j_*)} & B \ar[d]_{\sim}^{a^{-1}} \\
		{^p}j^*{^p}j_!B \ar[r]_-{\sim}^-{({^p}j^*{^p}j_!)a^{-1}} \ar@/_3mm/[dr]_-{{^p}j^*f} & {^p}j^*{^p}j_!{^p}j^*{^p}j_*B \ar[r]^-{{^p}j^* \tilde{b} \, {^p}j_*} & {^p}j^*{^p}j_*B \\
		& {^p}j^*A \ar@/_3mm/[ur]_{{^p}j^*g} &
	}
\]
is commutative. In fact, we have already encountered the two squares in the middle in the proof of \autoref{restricteta}. Therefore, ${^p}j^*f$ is monic and ${^p}j^*g$ is epic. If in addition ${^p}j^*f$ is epic or ${^p}j^*g$ is monic, then ${^p}j^*f$ and ${^p}j^*g$ are isomorphisms, inverse to each other. Consequently, $A$ is an extension of $B$ to $D$.
\end{proof}
\begin{definition}
Let $\CA$ be an abelian category, let $A$ be an object of $\CA$ and let $A_1,\dots,A_n$ be a finite set of subobjects of $A$. The \emph{intersection} $A_1 \cap A_2 \cap \dots \cap A_n$ is the kernel of the map
\[
	A \to \prod_{i=1}^n A/A_i
\]
given by the projections $A \to A/A_i$ for $1 \leq i \leq n$.
\end{definition}
\begin{lemma}  \label{abelianintersection}
Let $A$ be an object of $C$. Let $A_1,\dots,A_n$ be subobjects of $A$ such that ${^p}j^*A_i\cong B$ for some object $B$ of $C_U$ and every $1 \leq i \leq n$. Then ${^p}j^*(A_1 \cap \dots \cap A_n) \cong B$.
\end{lemma}
\begin{proof}
As an exact functor, ${^p}j^*$ commutes with finite products. Hence the exact sequence
\[
	0 \longrightarrow {^p}j^*\bigcap_{i=1}^n A_i \longrightarrow {^p}j^* A \longrightarrow {^p}j^*\prod_{i=1}^n A/A_i
\]
yields an exact sequence
\[
	0 \longrightarrow {^p}j^*\bigcap_{i=1}^n A_i \longrightarrow {^p}j^* A \longrightarrow \prod_{i=1}^n {^p}j^*A/B.
\]
This shows that ${^p}j^*\bigcap_{i=1}^n A_i$ is isomorphic to $B$ because the kernel of the natural map ${^p}j^* A \to \prod_{i=1}^n {^p}j^*A/B$ equals the kernel of the projection ${^p}j^* A \to {^p}j^*A/B$. 
\end{proof}
There are different characterizations of the intermediate extension summarized in the next proposition.
\begin{proposition} \label{ienewchar}
Let $B$ be an object of $C_U$. For $A \in C$, the following characterizations (up to isomorphism) are equivalent:
\begin{enumerate}[(i)]
	\item The object $A$ is the image of the natural morphism ${^p}j_!B \to {^p}j_*B$, in other words, $A = j\me B$.
	\item The object $A$ is an extension of $B$ to $D$ without non-trivial subobjects or quotients in the essential image of $i_*$.
	\item The object $A$ is the smallest subobject of ${^p}j_*B$ such that the canonical morphism ${^p}j^*A \to B$ is an isomorphism. Here the \emph{canonical morphism} ${^p}j^*A \to B$ is the 	morphism obtained from the inclusion $A \into {^p}j_*B$ by the adjunction between ${^p}j^*$ and ${^p}j_*$.
\end{enumerate}
In particular, an object $A$ of $D$ satisfying (ii) is unique.
\end{proposition}
\begin{proof}
We will proof (i) $\Leftrightarrow$ (ii) and (ii) $\Leftrightarrow$ (iii). The reason for this seemingly pedestrian approach is that later on the proof of (ii) $\Leftrightarrow$ (iii) can be directly transferred to Cartier crystals, where we do not have the extension by zero $j_!$.
 
(i) $\Rightarrow$ (ii). Let $f\colon {^p}j_!B \to A$ be the natural epimorphism and let $g\colon A \to {^p}j_*B$ be the natural monomorphism. By definition, $g \circ f = \eta(B)$. Since ${^p}j^*$ is exact, ${^p}j^*f$ is epic and ${^p}j^*g$ is monic. It follows from \autoref{extension} that $A$ is an extension of $B$ to $D$. By construction, $j\me B$ is a subobject of ${^p}j_*B$ and a subquotient of ${^p}j_!B$. This implies that $j\me B$ does not have any non-trivial subobjects or quotients supported on $Z$ (\autoref{subobquot}).  

(ii) $\Rightarrow$ (i). Let $f\colon {^p}j_!B \to A$ and $g\colon A \to {^p}j_*B$ be the adjoint morphisms of the given isomorphism $\tau\colon {^p}j^*A \overset{\sim}{\longrightarrow} B$. By \autoref{extension} we have $g \circ f = \eta(B)$. Applying the exact functor ${^p}j^*$ to the exact sequence
\[
	0 \longrightarrow \ker g  \longrightarrow A \overset{g}{\longrightarrow} {^p}j_*B,
\]
we see that ${^p}j^*\ker g \simeq 0$: By construction the composition of ${^p}j^*g$ and the natural isomorphism ${^p}j^*{^p}j_*B \to B$ equals $\tau$. Therefore ${^p}j^*g$ is an isomorphism. By assumption, ${^p}j^* \ker g \cong 0$ implies $\ker g \cong 0$. Similarly, we see that $f$ is epic. This means that $A$ is the image of $\eta(B)$. 

(ii) $\Rightarrow$ (iii). In the proof of (ii) $\Rightarrow$ (i) we have already seen that $A$ is a subobject of ${^p}j_*B$. It remains to show that it is the smallest subobject of ${^p}j_*B$ such that the natural morphism ${^p}j^*A \to B$ is an isomorphism. Suppose that $E$ is a subobject of ${^p}j_*B$ with ${^p}j^*E \cong B$. By \autoref{abelianintersection} we may assume that $E$ is a subobject of $A$. It follows that 
\[
	{^p}j^*(A/E) \cong {^p}j^*A/{^p}j^*E \cong 0,
\]
which means that the quotient $A/E$ is in the essential image of the functor ${^p}i_*$. By assumption, this quotient is trivial, in other words $E \cong A$.

(iii) $\Rightarrow$ (ii). By assumption, $A$ is an extension of $B$ to $D$. Every subobject of $A$ is a subobject of ${^p}j_*B$ and ${^p}j_*B$ has no non-trivial subobjects supported on $Z$ (\autoref{subobquot}). Now let $F$ be a subobject of $A$ such that the quotient $A/F$ is in the essential image of ${^p}i_*$. This implies ${^p}j^*(A/F) \cong 0$ and, because of ${^p}j^*(A/F) \cong {^p}j^*A / {^p}j^*F$, it follows that ${^p}j^*A \cong {^p}j^*F$ and hence ${^p}j^*F \cong B$. This contradicts the minimality of $A$.
\end{proof}
\begin{remark} \label{iebyzeroext}
Similar to the characterization (iii) of the preceding proposition, the intermediate extension $j\me B$ is a quotient of ${^p}j_!B$ such that the adjoint map $B \to {^p}j^*j\me B$ is an isomorphism and with the following universal property: For every epimorphism ${^p}j_!B \to Q$ such that the adjoint $B \to {^p}j^*Q$ is an isomorphism, the canonical epimorphism ${^p}j_!B \to j\me B$ factors uniquely through $Q$. In other words, $j\me B$ is the quotient of ${^p}j_!B$ by the maximal subobject supported on $Z$. 

For the proof, let $p\colon {^p}j_!B \to Q$ be an epimorphism such that the adjoint $B \to {^p}j^*Q$ is an isomorphism. Then there is a morphism $q\colon Q \to {^p}j_*B$ making the diagram
\[
	\xymatrix{
		{^p}j_!B \ar[r]^-{\eta} \ar[d]_p & {^p}j_*B \\
		Q \ar[ur]_-q & 
	}
\]
commutative (\autoref{extension}). Thus $q$ induces an epimorphism $q'\colon Q \to j\me B$ such that the canonical epimorphism ${^p}j_!B \to j\me B$ equals the composition $q' \circ p$.  
\end{remark}  
With the intermediate extension functor we can describe the simple objects of $C$ in dependence of the simple objects of $C_U$ and $C_Z$. 
\begin{proposition}[\protect{{\cite[Proposition 1.4.26]{BBD}}}] \label{simpleobjects}
An object of $C$ is simple if and only if it is of the form $i_*T$ for some simple object of $C_Z$ or if it is of the form $j\me T$ for some simple object $T$ in $C_U$.
\end{proposition}
\begin{proof}
If an object $S$ is contained in the essential image of $i_*$, then $S$ is a simple object of $C$ if and only if  $S$ is simple in $i_*C_Z$ because this subcategory is thick in $C$. 

Now we suppose that $T = {^p}j^*S \ncong 0$. Let $S$ be a simple object of $C$. It follows that $T$ is a simple object of $C_U$: if there would exist a proper subobject $T'$ of $T$, then ${^p}j_*T'$ would be a proper subobject of ${^p}j_*T$ because ${^p}j_*$ is left exact and ${^p}j^*{^p}j_* \cong \id$. By definition, $S$ is an extension of $T$ to $D$ and $S$ has no non-trivial subobjects, in particular no non-trivial subobjects or quotients supported on $Z$. Hence $S=j\me T$ (\autoref{ienewchar}).

Conversely, let $T$ be a simple object of $C_U$. To see that $S=j\me T$ is a simple object, consider an arbitrary short exact sequence
\[
	0 \longrightarrow A \overset{\alpha}{\longrightarrow} S \overset{\beta}{\longrightarrow} B \longrightarrow 0
\]
in $C$. The sequence
\[
	0 \longrightarrow {^p}j^*A \xrightarrow{{^p}j^*\alpha} T \xrightarrow{{^p}j^*\beta} {^p}j^*B \longrightarrow 0
\]
is also exact and by assumption ${^p}j^*A \cong 0$ or ${^p}j^*B \cong 0$. But $S$ neither has non-trivial subobjects nor non-trivial quotients supported on $Z$. Therefore, $A \cong 0$ or $B \cong 0$. 
\end{proof} 
\begin{lemma} \label{endexact}
We have the following exactness properties of the intermediate extension:
\begin{enumerate}[(i)]
	\item Let $0 \longrightarrow A \overset{i}{\longrightarrow} B$ be an exact sequence in $C_U$. Then the sequence $0 \longrightarrow j\me A \xrightarrow{j\me(i)} j\me B$ is also exact. 
	\item Let $B \overset{p}{\longrightarrow} C \longrightarrow 0$ be an exact sequence in $C_U$. Then the sequence $j\me B \xrightarrow{j\me(p)} j\me C \longrightarrow 0$ is also exact.
\end{enumerate}
\end{lemma}
\begin{proof} 
We have ${^p}j^*j\me(i) \cong i$ and ${^p}j^*j\me(p) \cong p$. This implies that the kernel of $j\me(i)$ and the cokernel of $j\me(p)$ are supported on $Z$. Hence $\ker j\me(i) \cong 0$ and $\coker j\me(p) \cong 0$ because $\ker j\me(i)$ is a subobject of $j\me A$ and $\coker j\me(p)$ is a quotient of $j\me C$. 
\end{proof}

\subsection{The intermediate extension of Cartier crystals}

For the rest of this section let $X$ be an $F$-finite, locally Noetherian scheme of characteristic $p$. Furthermore, let $j\colon U \to X$ be an open immersion and let $i\colon Z \to X$ be a closed immersion such that $j(U)$ is the complement of $i(Z)$. If one had a left adjoint $j_!$ for the exact functor $j^*\colon \CrysC(X) \to \CrysC(U)$, one could directly apply the formalism of Beilinson, Bernstein and Deligne from the previous subsection to define the intermediate extension of Cartier crystals. However, we can draw inspiration from \autoref{ienewchar} and define the intermediate extension of Cartier crystals by a universal property. Then we show the existence.
\begin{definition} 
Let $\CM$ be a Cartier crystal on $U$. The \emph{intermediate extension} $j\me \CM$ of $\CM$ is the smallest subcrystal $\CN$ of $j_*\CM$ such that $j^!\CN \cong \CM$, i.e.\ $\Phi\colon j\me \CM \into j_*\CM$ is a subcrystal with $j^!j\me\CM \cong \CM$ and the following universal property: For every inclusion $\alpha\colon \CN \into j_*\CM$ with $j^*\CN \cong \CM$, there is a unique inclusion $\tilde{\alpha}\colon j\me\CM \into \CN$ such that $\Phi=\alpha \circ \tilde{\alpha}$.
\end{definition}     
\begin{theorem} 
Let $\CM$ be a Cartier crystal on $U$. Then the intermediate extension $j\me \CM$ exists and it is unique up to unique isomorphism.
\end{theorem} 
\begin{proof}
We can construct a descending sequence of Cartier crystals
\[
	j_*\CM=:\CN_0 \supsetneq \CN_1 \supsetneq \CN_2 \supsetneq \dots
\]
such that $j^*\CN_i \cong \CM$ by iteratively choosing proper subcrystals. The descending chain condition for Cartier crystals on $F$-finite schemes (\cite[Corollary 4.7]{BliBoe.CartierFiniteness}) ensures that this sequence stabilizes at some crystal $\CN_m$. 

We claim that $\CN_m$ satisfies the universal porperty of $j\me \CM$. Let $\CN \subseteq j_*\CM$ be a Cartier crystal such that $j^*\CN \cong \CM$. But then $j^*(\CN_m \cap \CN) \cong \CM$ (\autoref{abelianintersection}) and by construction of $\CN_m$ it follows that $\CN_m \cap \CN = \CN_m$. Hence there is a monomorphism $\CN_m \into \CN$. The uniqueness of $j\me \CM$ is an immediate consequence of its universal property.        
\end{proof}
\begin{definition}
Let $\CA$ be an abelian category. Let $g\colon A \to B$ be a morphism in $\CA$. For a subobject $C$ of $B$, the \emph{preimage} $g^{-1}C$ is the kernel of the composition
\[
	A \overset{g}{\longrightarrow} B \overset{\pr}{\longrightarrow} B/C,
\]
where $\pr$ is the natural projection.
\end{definition}
\begin{proposition}
The assignment
\[
	\CM \mapsto j\me \CM
\]
defines a functor $j\me\colon \CrysC(U) \to \CrysC(X)$.
\end{proposition}
\begin{proof}
Let $\phi\colon \CM \to \CN$ be a morphism of Cartier crystals on $U$. We claim that $j_*\phi$ restricts to a morphism on the subcrystals $j\me\phi\colon j\me \CM \to j\me \CN$. For this we have to show that $\phi(j\me \CM) \subseteq j\me \CN$, which is equivalent to $\phi^{-1}(j\me \CN) \supseteq j\me \CM$. By definition of $j\me \CM$, it suffices to show that $j^*\phi^{-1}(j\me \CN) \cong \CM$. This is clear after applying the exact functor $j^*$ to the exact sequence
\[
	0 \longrightarrow \phi^{-1}(j\me \CN) \overset{i}{\longrightarrow} j_*\CM \overset{\alpha}{\longrightarrow} j_*\CN / j\me \CN,
\]
where $i$ is the natural inclusion and $\alpha$ is the composition of $j_*\phi$ and the projection $j_*\CN \to j_*\CN/j\me\CN$. 

By construction, it is clear that $j\me(\id_{\CM})=\id_{j\me \CM}$ for every Cartier crystal $\CM$ on $U$. For two morphisms $\phi\colon \CM_1 \to \CM_2$ and $\psi\colon \CM_2 \to \CM_3$, we have $j\me(\psi \circ \phi)=j\me \psi \circ j\me \phi$ because the diagram
\[
	\xymatrix{
		j_* \CM_1 \ar@/^15pt/[rr]^{j_*(\psi \circ \phi)} \ar[r]_{j_* \phi} & j_* \CM_2 \ar[r]_{j_* \psi} & j_* \CM_3 \\
		j\me \CM_1 \ar@/_15pt/[rr]_{j\me(\psi \circ \phi)} \ar@^{(->}[u] \ar[r]^{j\me \phi} & j\me \CM_2 \ar@^{(->}[u] \ar[r]^{j\me \psi} & j\me \CM_3. \ar@^{(->}[u]
	}	
\]		
commutes. Here the vertical maps are the natural monomorphisms. 
\end{proof}
Now we show that the intermediate extension of Cartier crystals is compatible with compositions of open immersions. In a recollement situation, this follows from the natural isomorphisms ${^p}(v \circ u)_! \cong {^p}v_! \circ {^p}u_!$ and ${^p}(v \circ u)_* \cong {^p}v_* \circ {^p}u_*$.
\begin{lemma}
Let $u\colon U \to V$ and $v\colon V \to X$ be open immersions. Then there is a natural isomorphism of functors
\[
	v\me \circ u\me \cong (v \circ u)\me.
\]
\end{lemma}
\begin{proof}
Let $\CM$ be a Cartier crystal on $U$. By the left exactness of $v_*$, we have natural inclusions
\[
	v\me u\me \CM \subseteq v_* u\me \CM \subseteq v_*u_*\CM \cong (v \circ u)_*\CM.
\]
This means that $v\me u\me \CM$ is a subcrystal of $(v \circ u)_*\CM$. Moreover,
\[
	(v \circ u)^*(v\me u\me \CM) \cong u^*v^*v\me u\me \CM \cong u^*u\me \CM \cong \CM.
\]
It remains to show that $v\me u\me \CM$ is the smallest subcrystal of $(v \circ u)_*\CM$ with this property. Let $i\colon \CN \to v\me u\me \CM$ be an injective morphism of Cartier crystals on $X$ such that $(v \circ u)^*i$ is an isomorphism. We obtain an inclusion $v^*i\colon v^*\CN \to u\me \CM$. Since $u^*v^*\CN \cong \CM$, the universal property of $u\me \CM$ implies that $v^*i$ is an isomorphism. By the defining property of $v\me (u\me \CM)$, the inclusion $i$ must be an isomorphism.
\end{proof}
Further properties of the intermediate extension in a recollement situation remain valid in the situation of Cartier crystals, even with the same proofs if we replace ${^p}j^*$ and ${^p}j_*$ by the functors $j^*$ and $j_*$ of Cartier crystals. By an \emph{extension of a Cartier crystal $\CM$ on $U$ to $X$} we mean a Cartier crystal $\CN$ on $X$ together with an isomorphism $j^*\CN \cong \CM$. In other words, we set $D_U = D_{\crys}^+(\QCrysC(U))$, $D_Z = D_{\crys}^+(\QCrysC(Z))$ and $D = D_{\crys}^+(\QCrysC(X))$, each equipped with the canonical $t$-structure.
\begin{proposition}
Let $\CM$ be a Cartier crystal on $U$. The intermediate extension $j\me \CM$ is the unique extension of $\CM$ to $X$ which neither has non-trivial subcrystals nor quotients supported on $Z$.
\end{proposition}
\begin{proof}
The proof of the equivalence $(ii) \Leftrightarrow (iii)$ of \autoref{ienewchar} can be adopted. 
\end{proof}
\begin{proposition}
A Cartier crystal on $X$ is simple if and only if 
\begin{enumerate}[(i)]
\item it is of the form $i_*\CN$ for some simple Cartier crystal $\CN$ on $Z$ or 
\item it is of the form $j\me \CN$ for some simple Cartier crystal $\CN$ on $U$.
\end{enumerate}
\end{proposition}
\begin{proof}
We proceed as in the proof of \autoref{simpleobjects}. Recall that $i_*\CrysC(Z)$ is a thick subcategory because a Cartier crystal $\CM$ on $X$ is contained in the essential image of $i_*$ if and only if $j^*\CM=0$ (see \autoref{Cartiersupport}) and the functor $j^*$ is exact.
\end{proof}   
\begin{lemma}
The intermediate extension of Cartier crystals preserves injections and surjections, i.e.\ we have the following properties:
\begin{enumerate}[(i)]
	\item Let $0 \longrightarrow \CA \overset{i}{\longrightarrow} \CB$ be an exact sequence in $\CrysC(U)$. Then the sequence $0 \longrightarrow j\me \CA \xrightarrow{j\me(i)} j\me \CB$ is also exact. 
	\item Let $\CB \overset{p}{\longrightarrow} \CC \longrightarrow 0$ be an exact sequence in $\CrysC(U)$. Then the sequence $j\me \CB \xrightarrow{j\me(p)} j\me \CC \longrightarrow 0$ is also exact.
\end{enumerate}
\end{lemma}
\begin{proof}
We can adopt the proof of \autoref{endexact}.
\end{proof}

\subsection{Intermediate extensions of Cartier crystals and perverse sheaves}

As mentioned in \autoref{Gabberrecollement}, the perverse $t$-structure of $D^b(X_{\et},\ZZ/p\ZZ)$ is obtained by recollement from the perverse $t$-structures on $D^b(U_{\et},\ZZ/p\ZZ)$ and $D^b(Z_{\et},\ZZ/p\ZZ)$. Moreover, the functors $j_!$, $j_*$, $j^*$, $i^*$, $i^*$ and $i^!$ have finite cohomological amplitude. Hence the formalism of recollement applies for these categories. 
\begin{proposition}
The intermediate extension $j\me$, defined as the image of the canonical morphism ${^p}j_! \to {^p}j_*$, restricts to a functor $\Perv_c(U_{\et},\ZZ/p\ZZ) \to \Perv_c(X_{\et},\ZZ/p\ZZ)$. 
\end{proposition}
\begin{proof}
Let $\CL$ be an object of $\Perv_c(U_{\et},\ZZ/p\ZZ)$. The functor $j_!$ restricts to a functor $D_c^b(U_{\et},\ZZ/p\ZZ) \to D_c^b(X_{\et},\ZZ/p\ZZ)$ (\cite[8.3]{EmKis.Fcrys}). Thus ${^p}j_!\CL \in \Perv_c(X_{\et},\ZZ/p\ZZ)$. Let $K$ be the kernel of the natural surjection ${^p}j_!\CL \to j\me\CL$. Since $j\me\CL$ is a perverse sheaf, we know that $K$ is an object of $\Perv(X_{\et},\ZZ/p\ZZ)$. Moreover, by \cite[Corollary 12.4]{Gabber.tStruc}, $K$ has constructible cohomology sheaves because as a perverse sheaf it is a subobject of an object in the subcategory $\Perv_c(X_{\et},\ZZ/p\ZZ)$. It follows that ${^p}j_!\CL/K \cong j\me \CL \in \Perv_c(X_{\et},\ZZ/p\ZZ)$. 
\end{proof}
\begin{corollary}
Assume that $k$ is a perfect field of characteristic $p$. Let $X$ be an $F$-finite embeddable $k$-scheme. Let $j\colon U \to X$ be an open immersion. There is a natural isomorphism of functors
\[
	j\me \circ \Sol \to \Sol \circ j\me.
\]
\end{corollary}
\begin{proof}
Let $\CM$ be a Cartier crystal on $U$. Let $\CL$ denote the perverse constructible \'etale sheaf $\Sol(\CM)$. Let $\phi\colon j\me \CM \to j_*\CM$ be the natural monomorphism. Applying $\Sol$ to $\phi$, we obtain an epimorphism $\psi\colon {^p}j_!\CL \to \Sol(j\me \CM)$ because, by \autoref{RHCorrespondenceCrystals}, we have natural isomorphisms 
\[
	\Sol(j_*\CM) \cong \Sol(H^0Rj_*\CM) \cong {^p}H^0 \Sol(Rj_*\CM) \cong {^p}H^0 j_!\Sol(\CM) = {^p}j_! \Sol(\CM).
\]
It suffices to show that $\Sol(j\me \CM)$ verifies the universal property of \autoref{iebyzeroext}. Let $q\colon {^p}j_!\CL \to \mathcal{Q}$ be an epimorphism in $\Perv_c(X_{\et},\ZZ/p\ZZ)$ such that the adjoint $\CL \to {^p}j^*\mathcal{Q}$ is an isomorphism. Under the equivalence $\Sol$, the epimorphism $q$ corresponds to a monomorphism $p\colon \CN \to j_*\CM$ for a Cartier crystal $\CN$ on $U$ with $\Sol(\CN) \cong \mathcal{Q}$. Furthermore, the adjoint $\tilde{p}\colon j^*\CN \to \CM$ is an isomorphism. For this we consider the following diagram:
\[
	\xymatrix@C50pt{
		\Sol(\CM) \ar[r]_-{\sim}^-{\Sol(a)} \ar@{=}[dd] & \Sol(j^*j_*\CM) \ar[r]^-{\Sol(j^*p)} \ar[d]^{\sim} & \Sol(j^*\CN) \ar[d]^{\sim} \\
		& j^*\Sol(j_*\CM) \ar[r]^-{j^*\Sol(p)} \ar[d]^{\sim} & j^*\Sol(\CN) \ar@{=}[d] \\
		\Sol(\CM) \ar[r]_-{\sim}^-{b} & j^*{^p}j_!\Sol(\CM) \ar[r]^-{j^*q} & j^*\Sol(\CN),
	}
\]
where $a\colon j^*j_* \to \id$ and $b\colon \id \to j^*{^p}j_!$ are the adjunction morphisms. The rectangle on the left commutes because $\Sol$ is an equivalence of categories and the counit and unit of adjunction are unique. The naturality of the isomorphism $\Sol \circ j^* \cong j^* \circ \Sol$ implies the commutativity of the upper right square and the square below commutes by construction of the morphism $p$. The composition of the lower horizontal morphisms is the adjoint of $q$ and therefore an isomorphism by assumption. It follows that the composition of the morphisms of the top row is an isomorphism. But this composition equals $\Sol(\tilde{p})$. Hence $\tilde{p}$ is an isomorphism.

The universal property of $j\me \CM$ yields a unique monomorphism $p'\colon j\me \CM \to \CN$ such that $\phi = p \circ p'$. Applying $\Sol$ we obtain a unique epimorphism $q'\colon \mathcal{Q} \to \Sol(j\me\CM)$ making the diagram
\[
	\xymatrix{
		{^p}j_!\CL \ar[d]_q \ar[r]^-{\psi} & \Sol(j\me \CM) \\
		\mathcal{Q} \ar[ur]_-{q'} &
	}
\]
commutative.
\end{proof}
\bibliographystyle{amsalpha}
\bibliography{Bibliography}
\end{document}